\documentclass{amsart}[10pt]

\usepackage[left=3.18cm,right=3.18cm,top=2.54cm,bottom=2.54cm]{geometry}
\usepackage{amsmath, amssymb, amsfonts, amsthm}
\usepackage{cases}
\usepackage[all]{xy}
\usepackage{hyperref}

\usepackage[T1]{fontenc} 

\theoremstyle{plain}
\newtheorem{thm}{Theorem}[section]
\newtheorem{conj}[thm]{Conjecture}
\newtheorem{cor}[thm]{Corollary}
\newtheorem{lem}[thm]{Lemma}
\newtheorem{prop}[thm]{Proposition}

\newtheorem{question}{Question}[section]

\newtheorem{assum}[thm]{Assumption}

\renewcommand{\Re}{\operatorname{Re}}
\renewcommand{\Im}{\operatorname{Im}}

\theoremstyle{definition}
\newtheorem{defn}[thm]{Definition}
\newtheorem{ex}[thm]{Example}
\newtheorem{rmk}[thm]{Remark}

\newenvironment{pf}{\begin{proof}[Proof]}{\end{proof}}

\begin{document}

\title{Stability conditions on the canonical line bundle of $\mathbb{P}^{3}$}
\author{Tianle Mao}

\normalsize

\newpage
\maketitle

\begin{abstract}
	We study the space of stability conditions on the total space of the canonical line bundle over the three dimensional projective space. We construct a family of geometric stability conditions and some subset of the boudary of them, which are algebraic. We also use spherical twists to construct some other stability conditions. 
\end{abstract}

\section{Introduction}
In this paper, we study the space of stability conditions on the derived category of the local $\mathbb{P}^{3}$. Our construction depends on Bogomolov-Gieseker type inequalities and we prove these inequalities in local $\mathbb{P}^{3}$ case.

\subsection{Background and History}

The concept of stability conditions comes from the study of Dirichlet branes in string theory, especially from Douglas's work on $\pi$-stability in \cite{douglas2000d,douglas2002dirichlet}. In mathematics, we are concerned about the space of stability conditions $\mathrm{Stab}(\mathcal{D})$ on a fixed triangulated category $\mathcal{D}$. This was first introduced by Bridgeland in \cite{bridgeland2007stability} and it is a significant object in algebraic geometry. However, the existence of stability condition on the bounded derived category of coherent sheaves on smooth projective varieties is still unknown in general. So far it has been proved that $\mathrm{Stab}(D^{b}(X))$ is not empty when $X$ is a curve or a surface in \cite{okada2006stability,macri2007stability,arcara2012bridgeland}. Moreover, Bayer, Macr\`i and Toda gave a conjectural construction of stability conditions on general threefolds in \cite{bayer2014bridgeland} and they proved that the existence of geometric stability conditions on $D^{b}(\mathbb{P}^{3})$. Besides, stability conditions are known to exist on few families of smooth projective threefolds: Fano threefolds \cite{macri2014generalized,li2018stability}, Abelian threefolds \cite{bayer2016space} and quintic threefolds \cite{li2019stability}.

Meanwhile, we also consider stability conditions on the total space of a vector bundle. Explicitly, suppose that $Y$ is a smooth projective variety and $X$ is the total space of a vector bundle of $Y$ and we are also concerned about the space of stability conditions on the bounded derived category $D^{b}_{0}(X)$ of coherent sheaves supported on the zero-section. The space $\mathrm{Stab}(D^{b}_{0}(X))$ of stability conditions was first studied by Bridgeland in \cite{bridgeland2009spaces}, where $X$ is the total space of the canonical bundle of some Fano variety $Y$. Then, Bayer and Macr\`i put their attention on the case of $X=\omega_{\mathbb{P}^{2}}$ in \cite{bayer2011space}. They explicitly describe
a chamber of geometric stability conditions, and show that its translates via autoequivalences cover a whole connected component. They also explain its relation to autoequivalences of $\mathrm{Stab}(D^{b}_{0}(X))$, and provide a global mirror symmetry picture on it.

\subsection{Classical Bogomolov-Gieseker inequality}

In this paper, we will not only study the stability conditions on the derived category of the local $\mathbb{P}^{3}$, but also discuss the basic settings of Bridgeland stability conditions on the derived category of the general vector bundle. Suppose that $Y$ is a smooth projective variety over $\mathbb{C}$ with dimension three and $X=\mathrm{Tot}_{Y}(\mathcal{E}_{0})$ is the total space of a locally free sheaf $\mathcal{E}_{0}$ on $Y$. We denote by $\mathrm{Coh}_{0}(X)$ the abelian category of coherent sheaves supported on the zero-section and its bounded derived category by $D_{0}^{b}(X)$.

We want to use the same approach as in the article \cite{bayer2014bridgeland} to construct "geometric" stability conditions on $D_{0}^{b}(X)$, which have the property that
all skyscraper sheaves $k(y)$ are stable of the same phase, where $y \in Y$. To begin with, we fix a polarization $H$ of $Y$. For any $E \in \mathrm{Coh}_{0}(X)$, the slope function $\displaystyle\mu(E)=-\frac{H^{2}\mathrm{ch}_{1}(\pi_{*}E)}{H^{3}\mathrm{ch}_{0}(\pi_{*}E)}$ gives a notion of slope stability, where $\pi \colon X \to Y$ is the natural projection. However, one of the biggest problem to construct stability conditions is that the classical Bogomolov-Gieseker inequality may not hold, i.e.
\begin{align*}
	(H^{2}\mathrm{ch}_{1}(\pi_{*}E))^{2}-2H^{3}\mathrm{ch}_{0}(\pi_{*}E)H\mathrm{ch}_{2}(\pi_{*}E) \geqslant 0
\end{align*} may not hold for any semistable sheaf $E \in \mathrm{Coh}_{0}(X)$. Therefore, we need to impose one important assumption on the vector bundle in this paper as follows:
\begin{assum}\label{crucialassumptionintro}(Assumption \ref{crucialassumption})
	A sheaf $E \in \mathrm{Coh}_{0}(X)$ is a slope stable sheaf if and only if it is the pushforward $E=i_{*}E_{0}$ of some slope stable sheaf $E_{0} \in \mathrm{Coh}(Y)$. Here $i \colon Y \hookrightarrow X$ is the zero section of the vector bundle.
\end{assum} 
Under this assumption, we get the classical Bogomolov-Gieseker inequality:
\begin{prop}\label{classicalBGintro}(Propostion \ref{classicalBG})
	If $X$ satisfies Assumption \ref{crucialassumptionintro}, then for any slope semistable sheaf $E \in \mathrm{Coh}_{0}(X)$, we have the inequality:
	\begin{align*}
		(H^{2}\mathrm{ch}_{1}(\pi_{*}E))^{2}-2H^{3}\mathrm{ch}_{0}(\pi_{*}E)H\mathrm{ch}_{2}(\pi_{*}E) \geqslant 0.
	\end{align*}
\end{prop}

\subsection{Sketch of construction and Bogomolov-Gieseker type inequality with $\mathrm{ch}_{3}$}
We now give a sketch of our construction; the details will be given in Section \ref{section3}. For any $\beta \in \mathbb{R}$, we define:
\begin{itemize}
	\item $\mathrm{Coh}^{>\beta}_{0}(X)$ is generated by slope semistable sheaves $E \in \mathrm{Coh}_{0}(X)$ of slope $>\beta$;
	\item $\mathrm{Coh}^{\leqslant \beta}_{0}(X)$ is generated by slope semistable sheaves $E \in \mathrm{Coh}_{0}(X)$ of slope $\leqslant \beta$.
\end{itemize}Following
the classical construction for surfaces and threefolds in \cite{arcara2012bridgeland,bayer2014bridgeland}, we define: 
\begin{align*}
	\mathrm{Coh}^{\beta}_{0}(X)=\langle\mathrm{Coh}^{>\beta}_{0}(X),\mathrm{Coh}^{\leqslant \beta}_{0}(X)[1]\rangle \subset D_{0}^{b}(X).
\end{align*}where $[1]$ denotes the shift and $\langle\cdot\rangle$ the extension-closure. By Proposition \ref{classicalBGintro} and tilting theory, this is the heart of a bounded t-structure in $D_{0}^{b}(X)$, allowing us to define stability function on it.

We then define the following tilt-slope functions on $\mathrm{Coh}^{\beta}_{0}(X)$:
\begin{align*}
	\nu^{\beta,\alpha}(E)=\frac{H\mathrm{ch}_{2}(\pi_{*}E)-\alpha H^{3} \mathrm{ch}_{0}(\pi_{*}E)}{H^{2}\mathrm{ch}_{1}(\pi_{*}E)-\beta H^{3}\mathrm{ch}_{0}(\pi_{*}E)}
\end{align*}for any $\alpha>\displaystyle\frac{1}{2}\beta^{2}$ and $0 \neq E \in \mathrm{Coh}^{\beta}_{0}(X)$. It produces the notion of tilt stability. By the same process, we have a double tilting heart $\mathrm{Coh}_{0}^{\beta,\alpha}(X)$ whose explicit definition will be provided in Subsection \ref{subsectionBG}. Motivated by the construction of $\pi$-stability in string theory, we expect there are some suitable stability functions on $\mathrm{Coh}^{\beta,\alpha}_{0}(X)$. And we propose the following conjecture:
\begin{conj}\label{conj2intro}(Conjecture \ref{conj2})
	For any $\beta,\alpha \in \mathbb{R}$ with $\alpha>\displaystyle\frac{1}{2}\beta^{2}$ and $a>\displaystyle\frac{2\alpha-\beta^{2}}{6}$, the central charge $Z^{\beta,\alpha,a}$ is a stability function on the bounded heart $\mathrm{Coh}^{\beta,\alpha}_{0}(X)$, where
	\begin{align*}
		Z^{\beta,\alpha,a}(E)=-\mathrm{ch}_{3}^{\beta}(\pi_{*}E)+aH^{2}\mathrm{ch}_{1}^{\beta}(\pi_{*}E)+\sqrt{-1}\left(H\mathrm{ch}_{2}^{\beta}(\pi_{*}E)-(\alpha-\frac{\beta^{2}}{2})H^{3}\mathrm{ch}_{0}^{\beta}(\pi_{*}E)\right).
	\end{align*}
\end{conj} As an analogue of Bogomolov-Gieseker type inequality conjecture for $\mathrm{ch}_{3}$ in \cite{bayer2014bridgeland} and \cite{bayer2016space}, the above conjecture is equivalent to the following:
\begin{conj}(Conjecture \ref{conj})
	For any $\beta,\alpha \in \mathbb{R}$ with $\alpha>\displaystyle\frac{1}{2}\beta^{2}$ and any $\nu^{\beta,\alpha}$-semistable object $E \in \mathrm{Coh}_{0}^{\beta}(X)$ with tilt-slope $\beta$, we have the following Bogomolov-Gieseker type inequality:
	\begin{align*}
		\mathrm{ch}_{3}(\pi_{*}E) \leqslant \frac{2\alpha-\beta^{2}}{6}H^{2}\mathrm{ch}_{1}(\pi_{*}E).
	\end{align*}
\end{conj}
Meanwhile, by using the method introduced by Macr\`i in \cite{macri2014generalized}, we can reduce the Bogomolov-Gieseker type inequality to the case of small $\omega=\sqrt{2\alpha-\beta^{2}}$ with $\beta,\alpha$ rational.  Moreover, we also use shift functor and relative derived dual functor, which we will introduce explicitly in Appendix \ref{appendix2}, to reduce Bogomolov-Gieseker type inequality to the case of $-\displaystyle\frac{1}{2}\leqslant \beta \leqslant 0$. In conclusion, we can reduce Bogomolov-Gieseker type inequality to the following conjecture:
\begin{conj}(Conjecture \ref{conj3})
	For any $\beta,\alpha \in \mathbb{Q}$ with $\alpha>\displaystyle\frac{1}{2}\beta^{2}$, $-\displaystyle\frac{1}{2}\leqslant \beta \leqslant 0$, $\sqrt{2\alpha-\beta^{2}}<\displaystyle\frac{1}{2}$ and for any $\nu^{\beta,\alpha}$-semistable object $E \in \mathrm{Coh}_{0}^{\beta}(X)$ with tilt-slope $\beta$, we have the following Bogomolov-Gieseker type inequality:
	\begin{align*}
		\mathrm{ch}_{3}(\pi_{*}E) \leqslant \frac{2\alpha-\beta^{2}}{6}H^{2}\mathrm{ch}_{1}(\pi_{*}E).
	\end{align*}
\end{conj}
The above three conjectures are all equivalent to each other. Finally, if one of them holds, then there is a family of stability conditions $(Z^{\beta,\alpha,a},\mathrm{Coh}_{0}^{\beta,\alpha}(X))$ on $D_{0}^{b}(X)$.

\subsection{Geometric stability conditions on then local $\mathbb{P}^{3}$}

As a significant example, we focus on  $X=\mathrm{Tot}(\omega_{\mathbb{P}^{3}})$. To begin with, we need to prove Assumption \ref{crucialassumptionintro} holds in this case. The following is our first main theorem:
\begin{thm}(Theorem \ref{thm1})
	Let $X=\mathrm{Tot}(\omega_{\mathbb{P}^{3}})$. Then for any $\beta,\alpha \in \mathbb{R}$ with $\alpha>\displaystyle\frac{1}{2}\beta^{2}$ and $a>\displaystyle\frac{2\alpha-\beta^{2}}{6}$, the central charge $Z^{\beta,\alpha,a}$ is a stability function on the bounded heart $\mathrm{Coh}_{0}^{\beta}(X)$.
\end{thm}
On the other hand, the most important idea to prove the above theorem is the following lemma, which we will prove at the end of Section \ref{section3}. It is a variant of \cite[Proposition 8.1.1]{bayer2014bridgeland}.
\begin{lem}(Lemma \ref{BGlem})
Suppose $\mathcal{D}$ is a triangulated category and $\mathcal{A}$ is a bounded heart of $\mathcal{D}$. Let $Z \colon K(\mathcal{D}) \to \mathbb{C}$ be a group homomorphism such that for any $E \in \mathcal{A}$, $\Im Z(E) \geqslant 0$. Let $\mathcal{I}=\{E \in \mathcal{A} \mid \Im Z(E)=0\}$ be the full abelian subcategory of $\mathcal{D}$ which is of finite length, and $F[1]$ be a simple object of $\mathcal{I}$.

Assume that there exists another bounded heart $\mathcal{B}$ of  $\mathcal{D}$ with the following properties:

(1) $\mathcal{B} \subseteq \langle \mathcal{A},\mathcal{A}[1]\rangle$,

(2) there exists $\phi_{0} \in (0,1)$ such that:
\begin{align*}
	Z(\mathcal{B}) \subseteq \{r \mathrm{exp}(\sqrt{-1}\pi\phi) \mid r \geqslant 0, \phi_{0} \leqslant \phi \leqslant \phi_{0}+1\},
\end{align*}

(3) $F[2] \notin \mathcal{B}$. 

Then $\Re Z(F) \geqslant 0$.
\end{lem}

We will choose $\mathcal{B}$ to be the full extension-closed subcategory of $D_{0}^{b}(X)$ generated by the following objects: $i_{*}\mathcal{O}(1),i_{*}\mathcal{O}[1],i_{*}\mathcal{T}(-2)[2],i_{*}\mathcal{O}(-1)[3]$. It turns out that $\mathcal{B}$ is a bounded heart of $D_{0}^{b}(X)$ and is useful to prove the Bogomolov-Gieseker type inequality when $\omega=\sqrt{2\alpha-\beta^{2}}$ is small enough.

\subsection{Boundary points}
We will continue to focus on the local $\mathbb{P}^{3}$ case. By  \cite[Proposition 3.3]{bayer2011space}, we know that there is a wall and chamber structure on $\mathrm{Stab}(D_{0}^{b}(X))$ with respect to the skyscraper sheaf $k(y)$ with $y \in Y$. Thus, the set of geometric stability conditions $\mathrm{Stab}^{\mathrm{geo}}_{H}(D_{0}^{b}(X))$ is open in $\mathrm{Stab}(D_{0}^{b}(X))$. Its boundary $\partial \mathrm{Stab}^{\mathrm{geo}}_{H}(D_{0}^{b}(X))$ is given
by a locally finite union of walls. We aim to provide a complete description of these walls. However, this is quite difficult and we have only identified some special points on the boundary.

The key point of this construction is "tilting" an exceptional object of the double tilting heart $\mathrm{Coh}_{0}^{\beta,\alpha}(X)$. For example, if $\beta<0$, then there is a torsion pair $(\mathcal{T},\mathcal{F})$ of $\mathrm{Coh}_{0}^{\beta,\beta^{2}}(X)$ as follows:
\begin{align*}
	& \mathcal{T}=\{i_{*}\mathcal{O}[1]^{\oplus n} \mid n \in \mathbb{N}^{+}\}, \\
	& \mathcal{F}=\{F \mid \mathrm{Hom}(i_{*}\mathcal{O}[1],F)=0\}.
\end{align*}Then we obtain a tilting heart $\mathrm{Coh}_{0}^{\beta,i_{*}\mathcal{O}}(X)[1]$ with respect to the above torsion pair. We can also define the central charge $Z^{\beta,\alpha,a}$ on the heart $\mathrm{Coh}_{0}^{\beta,i_{*}\mathcal{O}}(X)$, which has the same form as in Conjecture \ref{conj2intro}. The following is the second main theorem:
\begin{thm}(Theorem \ref{simplecase} and Corollary \ref{boundary})
	For any rational number $\beta$ with $-\displaystyle\frac{1}{2}<\beta<0$ and $\displaystyle\frac{3\beta^{3}+6\beta^{2}-4}{6(3\beta+2)}<a<\displaystyle\frac{1}{6}\beta^{2}$, the pair $(Z^{\beta,\beta^{2},a},\mathrm{Coh}_{0}^{\beta,i_{*}\mathcal{O}}(X))$ is a stability condition on $\partial\mathrm{Stab}_{H}^{\mathrm{geo}}(D_{0}^{b}(X))$.
\end{thm}

\subsection{Open questions}
There are several questions we have not solved in this paper. For example, we are concerned about the connection between stability conditions on $D^{b}(Y)$ and $D^{b}_{0}(X)$. Therefore, one question arises:
\begin{question}\label{ques1}
	Is there any homeomorphism between $\mathrm{Stab}^{\mathrm{Geo}}(D^{b}(Y))$ with $\mathrm{Stab}^{\mathrm{Geo}}(D_{0}^{b}(X))$?
\end{question}In fact, Question \ref{ques1} is true in the case of the local $\mathbb{P}^{2}$ because we can verify $\mathrm{Stab}^{\mathrm{Geo}}(D^{b}(\mathbb{P}^{2}))$ and $\mathrm{Stab}^{\mathrm{Geo}}(D_{0}^{b}(\mathrm{Tot}(\omega_{\mathbb{P}^{2}})))$ are homomeomorphism to the same topology space.

Another question comes from mirror symmetry. The mirror partner for the local $\mathbb{P}^{2}$ is the universal family over the moduli
space $\mathcal{M}_{\Gamma_{1}(3)}$ of elliptic curves with $\Gamma_{1}(3)$-level structures. Its fundamental group is $\Gamma_{1}(3)$ and let $\widetilde{M}_{\Gamma_{1}(3)}$ be the universal cover. In \cite[Section 9]{bayer2011space}, Bayer and Macr\`i proved that there is an continous embedding $I$ from $\widetilde{M}_{\Gamma_{1}(3)}$ to $\mathrm{Stab}(D_{0}^{b}(\mathrm{Tot}(\omega_{\mathbb{P}^{2}})))$ which is equivariant with respect to the action by $\widetilde{M}_{\Gamma_{1}(3)}$ on both sides. Morevoer, the central charge of the image of $I$ is a solution of the Picard-Fuchs equation. 

\begin{question}
	Is there a similar consequence in the case of $X=\mathrm{Tot}(\omega_{\mathbb{P}^{3}})$?
\end{question}

\subsection{Plan of the paper}
The paper is organized as follows. In Section \ref{section2}, we review some fundamental definitions related to stability conditions on triangualted categories. In Section \ref{section3}, we discuss the basic settings of Bridgeland stability conditions on the derived category of the general vector bundle. In Section \ref{section4}, we focus on the critical case where $X=\mathrm{Tot}_{\mathbb{P}^{3}}(\omega_{\mathbb{P}^{3}})$, and we prove the Bogomolov-Gieseker type inequality in this case. Additionally, we prove Theorem \ref{thm1}. In Section \ref{section5}, we use full exceptional collection in $D^{b}(\mathbb{P}^{3})$ to construct some boundary points of geometric subspace $\mathrm{Stab}^{\mathrm{Geo}}(D_{0}^{b}(\mathrm{Tot}(\omega_{\mathbb{P}^{3}})))$. In Appendix \ref{appendix2}, we introduce the definitions and properties of the relative derived dual functor.

\subsection{Acknowledgement}
I would like to thank my supervisor Yukinobu Toda for his patient guidance and advice throughout the progress of my PhD's research. Without his help, I coundn't finish this paper. Besides, I would also like to thank Zhiyu Liu, Tomohiro Karube, Nantao Zhang, Dongjian Wu for their useful discussion.

\subsection{Notation}
In this paper, all the varieties are defined over $\mathbb{C}$. For a variety $X$, we denote by $\mathrm{Coh}(X)$ and $D^{b}(X)$ the Abelian category of coherent sheaves and its bounded derived category respectively.
For an object $E \in D^{b}(X)$, its support is defined by $\mathrm{Supp}(E)=\displaystyle\bigcup_{{l \in \mathbb{Z}}}\mathrm{Supp}(\mathcal{H}^{l}(E))$. For a sheaf $E \in \mathrm{Coh}(X)$, its dimension is defined by $\dim E=\dim \mathrm{Supp}(E)$. For a triangulated category $\mathcal{D}$, its Grothendieck group is denoted by $K(\mathcal{D})$.

\section{Preliminaries}\label{section2}

In this section, we will review some basic notions of Bridgeland stability conditions in \cite{bridgeland2007stability}.

\begin{defn}
	Let $\mathcal{D}$ be a triangulated category. A \emph{slicing} $\mathcal{P}$ of $\mathcal{D}$ consists of full subcategories $\mathcal{P}(\phi) \subseteq \mathcal{D}$ for each $\phi \in \mathbb{R}$ satisfying the following axioms:
	
	(1) For all $\phi \in \mathbb{R}$, $\mathcal{P}(\phi+1)=\mathcal{P}(\phi)[1]$,
	
	(2) If $\phi_{1}>\phi_{2}$ and $A_{i} \in \mathcal{P}(\phi_{i})$ then $\mathrm{Hom}_{\mathcal{D}}(A_{1},A_{2})=0$,
	
	(3) For any $0 \neq E \in \mathcal{D}$, there is a finite sequence of real numbers:
	$\phi_{1}> \cdots >\phi_{n}$ and a collection of distinguished triangles: $E_{i-1} \to E_{i} \to A_{i} \to E_{i-1}[1]$ in $\mathcal{D}$ with $A_{i} \in \mathcal{P}(\phi_{i})$ for $i \in \{1,\ldots,n\}$ and $E_{0}=0,E_{n}=E$. The decomposition is uniquely defined up to isomorphism.
\end{defn}

	In fact, we can check that $\mathcal{P}(\phi)$ is an extension-closed subcategory for any $\phi \in \mathbb{R}$. And if $I$ is an interval on $\mathbb{R}$, then we denote by $\mathcal{P}(I)$ the extension-closed subcategory of $\mathcal{D}$ generated by $\mathcal{P}(\phi)$, where $\phi \in I$.

In fact, slicing can be viewed as a generalization of bounded t-structure:

\begin{defn}\cite[Lemma 3.2]{bridgeland2007stability}\label{def:bounded heart}
	Let $\mathcal{A} \subseteq \mathcal{D}$ be a full additive subcategory of a triangulated category $\mathcal{D}$. Then we call $\mathcal{A}$ as the \emph{heart} of a bounded t-structure on $\mathcal{D}$ if and only if the following two conditions hold:
	
	(a) If $l_{1}>l_{2}$ are integers, then $\mathrm{Hom}_{\mathcal{D}}(E[l_{1}],F[l_{2}])=0$ for all $E,F$ of $\mathcal{A}$.
	
	(b) For any $0 \neq E \in \mathcal{D}$, there exists a finite sequence of integers $l_{1}> \cdots >l_{n}$ and a collection of triangles $E_{i-1} \to E_{i} \to A_{i} \to E_{i-1}[1]$ in $\mathcal{D}$ with $A_{i} \in \mathcal{A}[l_{i}]$ for $i \in \{1,\ldots,n\}$ and $E_{0}=0,E_{n}=E$. The decomposition is uniquely defined up to isomorphism.
	
	We will often simply say that $\mathcal{A}$ is a bounded heart of $\mathcal{D}$ and note that there is a natural isomorphism of groups: $K(\mathcal{A}) \to K(\mathcal{D})$.
\end{defn}

Given a bounded heart $\mathcal{A}$ of $\mathcal{D}$ and from the definition above, we have a family of cohomology functors $\mathcal{H}^{l}_{\mathcal{A}} \colon \mathcal{D} \to \mathcal{A}$. For example, if $E$ has a filtration which is mentioned in the above definition, then we defined $\mathcal{H}^{-l_{i}}_{\mathcal{A}}(E)=A_{i}[-l_{i}] \in \mathcal{A}$ and $\mathcal{H}^{l}_{\mathcal{A}}(E)=0$ for $l \neq -l_{i}$. 

\begin{prop}
	Let $\mathcal{A}$ be a bounded heart of $\mathcal{D}$ and $F \to E \to G \to F[1]$ be a distinguished triangle in $\mathcal{D}$. Then we have a long exact sequence in $\mathcal{A}$:
	\begin{align*}
		\cdots \to \mathcal{H}^{l}_{\mathcal{A}}(F) \to \mathcal{H}^{l}_{\mathcal{A}}(E) \to \mathcal{H}^{l}_{\mathcal{A}}(G) \to \mathcal{H}^{l+1}_{\mathcal{A}}(F) \to \mathcal{H}^{l+1}_{\mathcal{A}}(E) \to 
		\mathcal{H}^{l+1}_{\mathcal{A}}(G) \to \cdots
	\end{align*}
\end{prop}

In fact, for any $\phi \in \mathbb{R}$, we have a bounded heart $\mathcal{P}(\phi,\phi+1]$ of $\mathcal{D}$ and we call $\mathcal{P}(0,1]$ as the heart of slicing $\mathcal{P}$. 

\begin{defn}
	Let $\mathcal{D}$ be a triangulated category. A \emph{weak pre-stability condition} $\sigma=(Z,\mathcal{P})$ on $\mathcal{D}$ consists of a group homomorphism $Z \colon K(\mathcal{D}) \to \mathbb{C}$ called the \emph{central charge} and a slicing $\mathcal{P}$ of $\mathcal{D}$, such that for any $0 \neq E \in \mathcal{P}(\phi)$, $Z(E)=m_{\sigma}(E)e^{\sqrt{-1}\pi\phi}$ for some $m_{\sigma}(E) \in \mathbb{R}^{>0}$ when $\phi \notin \mathbb{Z}$ and $m_{\sigma}(E) \in \mathbb{R}^{ \geqslant 0}$ when $\phi \in \mathbb{Z}$. Moreover, if $m_{\sigma}(E)>0$ for any $0 \neq E \in \mathcal{P}(\phi)$, then we call $\sigma$ as \emph{pre-stability condition}. 
\end{defn}

A weak pre-stability condition $(Z,\mathcal{P})$ gives us a pair $(Z,\mathcal{A})$, where $\mathcal{A}=\mathcal{P}(0,1]$ is a bounded heart of $\mathcal{D}$. This construction gives an equivalent definition of weak pre-stability condition.

\begin{defn}
	(1) Let $\mathcal{A}$ be an abelian category. A group homomorphism $Z \colon K(\mathcal{A}) \to \mathbb{C}$ is called a \emph{weak stability function} on $\mathcal{A}$ if for any $E \in \mathcal{A},\mbox{ we have } Z(E) \in \mathbb{H} \cup \mathbb{R}^{\leqslant 0}$. Moreover, if $Z(E) \in \mathbb{H} \cup \mathbb{R}^{<0} \mbox{ for any } 0 \neq E \in \mathcal{A}$, then we call $Z$ as \emph{stability function} on $\mathcal{A}$.
	
	(2) Given a weak stability function $Z$ on $\mathcal{A}$, we can define:
	\begin{align*}    
		\mu(E)=
		\begin{cases}
			-\displaystyle\frac{\Re Z(E)}{\Im Z(E)}, & \Im Z(E)>0, \\
			+\infty, & \Im Z(E)=0.
		\end{cases}                
	\end{align*}
	for any $0 \neq E \in \mathcal{A}$. We call $\mu(E)$ the \emph{slope} of $E$ and $\displaystyle\frac{1}{\pi}\mathrm{arg}Z(E)$ the \emph{phase} of $E$, denoted by $\phi(E)$. (If $Z(E)=0$, then we define $\phi(E)=1$) Then $\mu(E)=-\cot \pi \phi(E)$ if we assume $+\infty=-\cot \pi$.
	
	(3) For $0 \neq E \in \mathcal{A}$, we call it $\mu$-\emph{semistable} if for any non-trivial subobject $F$ of $E$, we have $\mu(F)\leqslant \mu(E/F)$. It is equivalent to $\phi(F) \leqslant \phi(E/F)$. And we call it $\mu$-\emph{stable} if for any non-trivial subobject $F$ of $E$, we have $\mu(F)<\mu(E/F)$.
	
	(4) We say that a weak stability function $Z$ on $\mathcal{A}$ satisfies the \emph{Harder-Narasimhan property} if for any $0 \neq E \in \mathcal{A}$, there exists a filtration in $\mathcal{A}$:
	\begin{align*}
		0=E_{0} \subseteq E_{1} \subseteq \cdots \subseteq E_{n-1} \subseteq E_{n}=E
	\end{align*}
	such that $E_{i}/E_{i-1}$ is $\mu$-semistable for any $i \in \{1,\ldots,n\}$ and $\mu(E_{i}/E_{i-1})>\mu(E_{i+1}/E_{i})$ for any $i \in \{1,\ldots,n-1\}$. We can verify  that this filtration is unique and is called \emph{Harder-Narasimhan filtration of} $E$.	And we let $\mu_{\max}(E)=\mu(E_{1})$ and $\mu_{\min}(E)=\mu(E_{n}/E_{n-1})$.
\end{defn}

Suppose that we have a pair $(Z,\mathcal{A})$, where $\mathcal{A}$ is a bounded heart of the triangulated category $\mathcal{D}$ and $Z$ is a weak stability function on $\mathcal{A}$ with the Harder-Narasimhan property, then we can induce a canonical slicing of $\mathcal{D}$: for $\phi \in (0,1]$, 
\begin{align*}
	\mathcal{P}(\phi):=\{E \in \mathcal{A} \mid E \mbox{ is } \mu \mbox{-semistable with }\mu(E)=-\cot \pi \phi\} \cup \{0\},
\end{align*}
and other subcategories are defined by $\mathcal{P}(\phi+1)=\mathcal{P}(\phi)[1]$. Thus, we have the following proposition proved by Bridgeland:

\begin{prop}\cite[Prop 5.3]{bridgeland2007stability}
	To give a weak pre-stability condition on a triangulated category $\mathcal{D}$ is equivalent to giving a bounded heart  of $\mathcal{D}$ and a weak stability function on its heart with the Harder-Narasimhan property.
\end{prop}

We introduce a useful criterion to verify Harder-Narasimhan property of a weak stability function:

\begin{prop}\cite[Prop 4.10]{macri2017lectures}\label{HNproperty}
	Let $Z$ be a weak stability function on an abelian category $\mathcal{A}$ with $\mathcal{A}$ is Noetherian and the image of $\Im Z$ is discrete in $\mathbb{R}$. Then $Z$ has the Harder-Narasimhan property.
\end{prop}

In many cases, the central charge $Z$ may factor through a free abelian group $\Lambda$ with finite rank.

\begin{defn}
	Let $\mathcal{D}$ be a triangulated category and $v \colon K(\mathcal{D}) \to \Lambda$ be a group homomorphism where $\Lambda$ is a free abelian group with finite rank. Let $\mathcal{P}$ be a slicing of $\mathcal{D}$ and $Z \colon \Lambda \to \mathbb{C}$ be a group homomorphism. If the pair $(Z \circ v,\mathcal{P})$ is a weak pre-stability condition on $\mathcal{D}$, then we call the pair $(Z,\mathcal{P})$ as a weak pre-stability condition with respect to $v$.
\end{defn}

If $\sigma=(Z,\mathcal{P})$ is a weak pre-stability condition with respect to $v$, then we define the subgroup $\Lambda_{\sigma}$ of $\Lambda$ which is generated by $\{v(E) \in \Lambda \mid E \in \mathcal{A}\mbox{ with }Z(E)=0\}$. For any $v \in \Lambda$, we denote its image in $\Lambda/\Lambda_{\sigma}$ or $(\Lambda/\Lambda_{\sigma})_{\mathbb{R}}$ by $[v]$.

\begin{defn}
	Suppose that $\sigma=(Z,\mathcal{P})$ is a weak pre-stability condition with respect to $v$. If there exists a real number $C>0$, such that for any $\phi \in \mathbb{R}$, $0 \neq E \in \mathcal{P}(\phi)$ with $Z(E) \neq 0$, we have:
	\begin{align*}
		\|[v(E)]\| \leqslant C|Z(E)|,
	\end{align*}where $\|\cdot\|$ is a fixed norm on the real vector space $(\Lambda/\Lambda_{\sigma})_{\mathbb{R}}$, then we say that $(Z,\mathcal{P})$ has the \textit{support property} and $(Z,\mathcal{P})$ is a \textit{weak stability condition} on $\mathcal{D}$ with respect to $v$. If $(Z,\mathcal{P})$ is a pre-stability condition which has the support property, then we say that $(Z,\mathcal{P})$ is a \textit{stability condition} on $\mathcal{D}$ with respect to $v$.
\end{defn}

We denote the set of all weak stability conditions with respect to $v$ by $\mathrm{Stab}^{w}_{\Lambda}(\mathcal{D})$ and all stability conditions with respect to $v$ by $\mathrm{Stab}_{\Lambda}(\mathcal{D})$. Sometimes we will omit $\lambda$ if the meaning is clear in the context and just denote them by $\mathrm{Stab}^{w}(\mathcal{D})$ and $\mathrm{Stab}(\mathcal{D})$ respectively.

There is another way to give the support property:

\begin{prop}
	Suppose that $\sigma=(Z,\mathcal{A})$ is a weak pre-stability condition with respect to $v$. Then $\sigma$ has the support property if and only if there exists a real quadratic form $Q$ on $(\Lambda/\Lambda_{\sigma})_{\mathbb{R}}$, such that:
	
	(1) For any $0 \neq E \in \mathcal{A}$ which is $\sigma$-semistable, we have $Q([v(E)]) \geqslant 0$;
	
	(2) The quadratic form $Q|_{\ker Z}$ is negative definite on $(\Lambda/\Lambda_{\sigma})_{\mathbb{R}}$. Here $Z$ can be seen as a $\mathbb{R}$-linear map $Z \colon (\Lambda/\Lambda_{\sigma})_{\mathbb{R}} \to \mathbb{C}$ induced by the original one.
\end{prop}

\begin{pf}
	This proof can be checked in \cite[Lemma A.4]{bayer2016space} or \cite[Lemma 2.4]{piyaratne2019moduli}.
\end{pf}

\begin{defn}
	Let $\mathcal{A}$ be an abelian category. A \emph{torsion pair} of $\mathcal{A}$ is a pair of full subcategories $(\mathcal{T},\mathcal{F})$ satisfying two conditions below:
	
	(1) For any $T \in \mathcal{T}, F \in \mathcal{F}$, $\mathrm{Hom}(T,F)=0$.
	
	(2) For any $E \in \mathcal{A}$, there exists a short exact sequence in $\mathcal{A} \colon  0 \to T \to E \to F \to 0$, where $T \in \mathcal{T}$ and $F \in \mathcal{F}$.
\end{defn}

\begin{prop}\cite{happel1996tilting}
	Let $\mathcal{A}$ be a bounded heart of triangulated category $\mathcal{D}$. If $(\mathcal{T},\mathcal{F})$ is a torsion pair of $\mathcal{A}$, then there is a bounded heart of $\mathcal{D}$: 
\begin{align*}
	\mathcal{A}^{\dagger}=\{E \in \mathcal{D} \mid \mathcal{H}^{i}_{\mathcal{A}}(E)=0 \mbox{ for } i \neq 0,-1, \mathcal{H}^{0}_{\mathcal{A}}(E) \in \mathcal{T},\mathcal{H}^{-1}_{\mathcal{A}}(E) \in \mathcal{F}\}.
\end{align*}
\end{prop}

The heart $\mathcal{A}^{\dagger}$ is called \emph{tilting heart} with respect to torsion pair $(\mathcal{T},\mathcal{F})$ of $\mathcal{A}$. In fact, $\mathcal{A}^{\dagger}=\langle \mathcal{F}[1],\mathcal{T} \rangle$. And for any $E \in \mathcal{A}^{\dagger}$, we have a short exact sequence in $\mathcal{A}^{\dagger}$:
\begin{align*}
	0 \to F[1] \to E \to T \to 0,
\end{align*}where $F \in \mathcal{F}$ and $T \in \mathcal{T}$.
 
Here is one useful result about tilting heart.
 
\begin{prop}\cite[Lemma 1.1.2]{polishchuk2007constant}\label{tiltingheart}
 	Let $\mathcal{A},\mathcal{A}^{\dagger}$ be two bounded hearts of $\mathcal{D}$. If $\mathcal{A}^{\dagger} \subseteq \langle \mathcal{A},\mathcal{A}[1] \rangle$, then $\mathcal{A}^{\dagger}$ is a tilting heart with respect to a torsion pair $(\mathcal{T},\mathcal{F})$ of $\mathcal{A}$, where $\mathcal{T}=\mathcal{A}^{\dagger} \bigcap \mathcal{A}$, and $\mathcal{F}=\mathcal{A}^{\dagger}[-1] \bigcap \mathcal{A}$.
\end{prop}
 
 

Let $\mathrm{Aut}_{\Lambda}(\mathcal{D})$ denote the group of exact equivalences $\Phi \colon \mathcal{D} \to \mathcal{D}$ such that   the induced group isomorphism $\Phi_{*} \colon K(\mathcal{D}) \to K(\mathcal{D})$ is compatible with the map $v \colon K(\mathcal{D}) \to \Lambda$. Explicitly, for any $\Phi \in \mathrm{Aut}_{\Lambda}(\mathcal{D})$, there is a group isomorphism $\Phi_{*} \colon \Lambda \to \Lambda$ such that the following diagram commutes:
\begin{align*}
	\xymatrix{
		K(\mathcal{D}) \ar[d]^{v} \ar[r]^{\Phi_{*}} & K(\mathcal{D}) \ar[d]^{v} \\
		\Lambda \ar[r]^{\Phi_{*}} & \Lambda
	}.
\end{align*}Then, $\mathrm{Stab}_{\Lambda}^{w}(\mathcal{D})$ admits a left group action by $\mathrm{Aut}_{\Lambda}(\mathcal{D})$. Specifically, for any exact equivalence $\Phi \colon \mathcal{D} \to \mathcal{D}$, the action is given by $\Phi(Z,\mathcal{P})=(Z \circ (\Phi_{*})^{-1},\Phi(\mathcal{P}))$. We can verify that if $(Z, \mathcal{P}) \in \mathrm{Stab}_{\Lambda}(\mathcal{D})$, then $\Phi(Z,\mathcal{P}) \in \mathrm{Stab}_{\Lambda}(\mathcal{D})$.

\section{Geometric stability conditions on the vector bundle supported on the zero section}\label{section3}

\subsection{Basic settings}

Suppose that $Y$ is a smooth projective variety over $\mathbb{C}$ with dimension $n$, $X=\mathrm{Tot}_{Y}(\mathcal{E}_{0})$ is the total space of a locally free sheaf $\mathcal{E}_{0}$ with constant rank $r$ on $Y$ and $\pi \colon X \to Y$ is the natural projection. Meanwhile, let $i \colon Y \to X$ be the zero section of $\pi$. We denote by $\mathrm{Coh}_{0}(X) \subset \mathrm{Coh}(X)$ the full subcategory of coherent sheaves on $X$ set-theoretically supported on the zero-section of $\pi$, it is a Serre abelian subcategory. Similarly, we denote by $D_{0}^{b}(X) \subseteq D^{b}(X)$ the full subcategory of the bounded derived category of coherent sheaves on $X$ set-theoretically supported on the zero-section of $\pi$. It is a triangulated subcategory and note that $D_{0}^{b}(X) \cong D^{b}(\mathrm{Coh}_{0}(X))$ as observed in \cite{ishii2005autoequivalences,kashiwara2013sheaves}. Then $\mathrm{Coh}_{0}(X)$ is a bounded heart of $D_{0}^{b}(X)$. 

There are several functors we will use. First, we have an exact functor $i_{*} \colon \mathrm{Coh}(Y) \to \mathrm{Coh}_{0}(X)$ and it induces an exact functor between derived categories: $i_{*} \colon D^{b}(Y) \to D_{0}^{b}(X)$. Besides, we have the functor $\pi_{*}$:

\begin{prop}\label{basicpropofproj}
	For any $E \in \mathrm{Coh}_{0}(X)$, we have $\pi_{*}E \in \mathrm{Coh}(Y)$ and $\mathrm{supp}(E)=i(\mathrm{supp}(\pi_{*}E))$. It means that we have an exact functor: $\pi_{*} \colon \mathrm{Coh}_{0}(X) \to \mathrm{Coh}(Y)$. Moreover, there is a pushforaward functor $\pi_{*} \colon D^{b}_{0}(X) \to D^{b}(Y)$. And for any $E \in D^{b}_{0}(X)$, we still have $\mathrm{supp}(E)=i(\mathrm{supp}(\pi_{*}E))$. In particular, if $\pi_{*}E \cong 0$ in $D^{b}(Y)$ for some $E \in D^{b}_{0}(X)$, then $E \cong 0$.
\end{prop}

\begin{cor}
	There is a group isomorphism between the Grothendieck groups by pushforward: $i_{*} \colon K(Y) \to K(D^{b}_{0}(X))$, its inverse is $\pi_{*} \colon K(D^{b}_{0}(X)) \to K(Y)$.
\end{cor}

We fix an ample divisor class $H \in \mathrm{NS}(Y)$, where $\mathrm{NS}(Y)$ is the Neron-Severi group of $Y$. For any $i \in \{0,\ldots,n\}$, we define a group homomorphism as follows:
$v_{i}^{H} \colon K(D^{b}_{0}(X)) \to \mathbb{Q}, E \mapsto H^{n-i}\mathrm{ch}_{i}(\pi_{*}E)$.
For $\beta \in \mathbb{R}$, we also define $v_{i}^{\beta H}(E)$ as the intersection number of $H^{n-i}$ with $i$-th component of $\mathrm{ch}(\pi_{*}E)e^{-\beta H}$. For example, we have:
\begin{align*}
	& v_{0}^{\beta H}(E)=v_{0}^{H}(E); & v_{1}^{\beta H}(E)=v_{1}^{H}(E)-\beta v_{0}^{H}(E);\\
	& v_{2}^{\beta H}(E)=v_{2}^{H}(E)-\beta v_{1}^{H}(E)+\frac{\beta^{2}}{2}v_{0}^{H}(E); & v_{3}^{\beta H}(E)=v_{3}^{H}(E)-\beta v_{2}^{H}(E)+\frac{\beta^{2}}{2}v_{1}^{H}(E)-\frac{\beta^{3}}{6}v_{0}^{H}(E).
\end{align*}

In our following discussion, we always fix a polarization $H$ and we write $v_{i}$ and $v_{i}^{\beta}$ instead of $v_{i}^{H}$ and $v_{i}^{\beta H}$ respectively. Let $\Lambda$ be the image of the following morphism:
\begin{align*}
	v \colon K(D^{b}_{0}(X)) \to \mathbb{Q}^{\oplus n+1}, E \mapsto (v_{0}(E),\ldots,v_{n}(E)).
\end{align*}Then $\Lambda$ is a free abelian group with rank $n+1$. The space of weak stability conditions on $D_{0}^{b}(X)$ with respect to $\Lambda$ is denoted by $\mathrm{Stab}^{w}_{H}(D_{0}^{b}(X))$ and the space of stability conditions on $D_{0}^{b}(X)$ with respect to $\Lambda$ is denoted by $\mathrm{Stab}_{H}(D_{0}^{b}(X))$.

Before we study the Bridgeland stability conditions on $D_{0}^{b}(X)$, we list some basic propositions we may need:

\begin{prop}
	For any $E_{0} \in D^{b}(Y), F \in D^{b}_{0}(X)$, there is an isomorphism of $\mathbb{C}$-vector spaces:
	\begin{align*}
		\mathrm{Hom}_{D^{b}(X)}(\pi^{*}E_{0},F) \cong \mathrm{Hom}_{D^{b}(Y)}(E_{0},\pi_{*}F);
	\end{align*}
\end{prop}

\begin{prop}
	For any $E_{0},F_{0} \in D^{b}(Y)$, there is an isomorphism of $\mathbb{C}$-vector spaces:
	\begin{align*}
		\mathrm{Hom}_{D^{b}(X)}(\pi^{*}E_{0},\pi^{*}F_{0}) \cong \mathrm{Hom}_{D^{b}(\mathrm{Qcoh}(Y))}(E_{0},\bigoplus_{p=0}^{\infty}F_{0} \otimes_{Y} \mathrm{Sym}^{p}\mathcal{E}_{0}^{\lor}),
	\end{align*}
\end{prop}

\begin{proof}
	$\pi_{*}\pi^{*}F_{0} \cong \pi_{*}(\pi^{*}F_{0} \otimes_{X} \mathcal{O}_{X}) \cong F_{0} \otimes_{Y} \pi_{*}\mathcal{O}_{X} \cong F_{0} \otimes_{Y} \mathrm{Sym}^{\bullet}(\mathcal{E^{\lor}})=\displaystyle\bigoplus_{p=0}^{\infty}F_{0} \otimes_{Y} \mathrm{Sym}^{p}\mathcal{E}_{0}^{\lor}$.
\end{proof}

\begin{prop}\label{Homsetofinclusion}
	For any $E_{0},F_{0} \in D^{b}(Y)$, there is an isomorphism of $\mathbb{C}$-vector spaces:
	\begin{align*}
		\mathrm{Hom}_{D^{b}_{0}(X)}(i_{*}E_{0},i_{*}F_{0}) \cong \bigoplus_{p=0}^{r}\mathrm{Hom}_{D^{b}(Y)}(\bigwedge^{p}\mathcal{E}_{0}^{\lor} \otimes_{Y} E_{0},F_{0}[-p]).
	\end{align*}
\end{prop}

\begin{pf}
	First, $\mathrm{Hom}_{D^{b}(X)}(i_{*}E_{0},i_{*}F_{0}) \cong \mathrm{Hom}_{D^{b}(Y)}(Li^{*}i_{*}E_{0},F_{0})$, and
	\begin{align*}
		i_{*}E_{0} \cong i_{*}(Li^{*}\pi^{*}E_{0} \otimes_{Y} \mathcal{O}_{Y}) \cong \pi^{*}E_{0} \otimes_{X} i_{*}\mathcal{O}_{Y}.
	\end{align*}
	And we have
	\begin{align*}
		Li^{*}i_{*}\mathcal{O}_{Y} \cong \bigoplus_{p=0}^{r} \left(\bigwedge^{p}\mathcal{E}_{0}^{\lor}\right)[p].
	\end{align*}
\end{pf}

We know that the canonical bundle of $X$ is isomorphic to $\pi^{*}(\det \mathcal{E}_{0}^{\lor} \otimes_{Y} \omega_{Y})$ and
if $\det \mathcal{E}_{0} \cong \omega_{Y}$, then $\omega_{X}$ is trivial. This is a very important example of non-compact Calabi-Yau manifolds.

\begin{defn}
	A stability condition $\sigma$ on $D_{0}^{b}(X)$ is \textit{geometric} if all skyscraper sheaves $k(y)$ of closed points $y \in Y$ are $\sigma$-stable of the same phase.
\end{defn}

\subsection{slope stability}

\begin{prop}
	The group homomorphism
	\begin{align*}
		\mathcal{Z}=-v_{1}+\sqrt{-1}v_{0} \colon \Lambda \to \mathbb{C},
	\end{align*}
	is a weak stability function on $\mathrm{Coh}_{0}(X)$ which satisfies Harder-Narasimhan property and support property. Moreover, if $\dim Y=1$, then it is a stability function.
\end{prop}

\begin{pf}
	By the definition of $v_{0}$, we know that the image of $\mathcal{Z}$ is discrete. And $\mathrm{Coh}_{0}(X)$ is Noetherian. Then we have Harder-Narasimhan property by Proposition \ref{HNproperty}.
	
	Finally, we need to prove the support property of this weak pre-stability condition. First, 
	\begin{align*}
		\mathcal{C}_{0}=\{E \in \mathrm{Coh}_{0}(X) \mid v_{1}(E)=v_{0}(E)=0\}=\{E \in \mathrm{Coh}_{0}(X) \mid \dim E \leqslant n-2\}
	\end{align*}by Proposition \ref{basicpropofproj}. Moreover, $\ker \mathcal{Z}$ equals to the subgroup generated by $\mathcal{C}_{0}$ in $\Lambda$ by the definition of $\Lambda$. And we could take a quadratic form $Q=0$ as the support property of this weak pre-stability condition. Thus, $(\mathcal{Z},\mathrm{Coh}_{0}(X))$ is a weak stability condition on $D_{0}^{b}(X)$.
\end{pf} 

In this paper, we denote by $\mu$ the slope function of the above weak stability condition.

In order to continue our construction, we assume that $X$ satisfies the following property:

\begin{assum}\label{crucialassumption}
	A sheaf $E \in \mathrm{Coh}_{0}(X)$ is a slope stable sheaf if and only if it is the pushforward $E=i_{*}E_{0}$ of some slope stable sheaf $E_{0} \in \mathrm{Coh}(Y)$.
\end{assum}

For example, if $X$ is the canonical bundle of $\mathbb{P}^{3}$, then it satisfies Assumption \ref{crucialassumption}. We will give a proof in section 4.

\begin{prop}\label{ipushforwardofslopestablesheaf}
	If $0 \neq E \in \mathrm{Coh}(Y)$ is slope (semi)stable, then $0 \neq i_{*}E \in \mathrm{Coh}_{0}(X)$ is also slope (semi)stable with the same slope.
\end{prop}

\begin{pf}
	The proposition follows since the subobjects and quotients of $i_{*}E$ are both of the form $i_{*}F$.
\end{pf}

Now, if $X$ satisfies Assumption \ref{crucialassumption}, then we have:

\begin{prop}\label{pipushforwardofslopestablesheaf}
	If $0 \neq E \in \mathrm{Coh}_{0}(X)$ is slope (semi)stable, then $0 \neq \pi_{*}E \in \mathrm{Coh}(Y)$ is also slope (semi)stable with the same slope. 
\end{prop}

\begin{pf}
	If $\mu(E)=+\infty$, then $v_{0}(E)=\mathrm{ch}_{0}(\pi_{*}E)=0$. Thus $\pi_{*}E$ is also slope semistable with slope $+\infty$.
	
	If $E$ is slope semistable with slope $\mu<+\infty$, then $E$ has a Jordan-H\"older filtration whose factors $F_{i}$ are all slope stable sheaves with the same slope $\mu$. But $\pi_{*}$ is an exact functor. Thus, $\pi_{*}E$ is in the extension-closed subcategory of $\pi_{*}F_{i}$, which each $\pi_{*}F_{i}$ is still slope stable with slope $\mu$. Thus, $\pi_{*}E$ is slope semistable with slope $\mu$.
\end{pf}

\subsection{tilt stability}

We assume that $\dim Y \geqslant 2$ and $X$ satisfies Assumption \ref{crucialassumption} from now. We will introduce the concept of tilt stability, which is the same as \cite{bayer2014bridgeland,bayer2016space}.

\begin{defn}
	For any $\beta \in \mathbb{R}$, let $(\mathrm{Coh}^{>\beta}_{0}(X),\mathrm{Coh}^{\leqslant\beta}_{0}(X))$ be the torsion pair of $\mathrm{Coh}_{0}(X)$ determined by:
	\begin{itemize}
		\item $\mathrm{Coh}^{>\beta}_{0}(X)$ is generated by slope semistable sheaves $E \in \mathrm{Coh}_{0}(X)$ of slope $>\beta$;
		\item $\mathrm{Coh}^{\leqslant \beta}_{0}(X)$ is generated by slope semistable sheaves $E \in \mathrm{Coh}_{0}(X)$ of slope $\leqslant \beta$.
	\end{itemize}
	Let $\mathrm{Coh}^{\beta}_{0}(X)$ be the tilting heart of $\mathrm{Coh}_{0}(X)$ with respect to the torsion pair $(\mathrm{Coh}^{>\beta}_{0}(X),\mathrm{Coh}^{\leqslant\beta}_{0}(X))$.
\end{defn}

\begin{prop}\label{classicalBG}
	(Classical Bogomolov-Gieseker inequality of slope semistable sheaf in $\mathrm{Coh}_{0}(X)$): For any slope semistable sheaf $E \in \mathrm{Coh}_{0}(X)$, we have the inequality:
	\begin{align*}
		v_{1}(E)^{2}-2v_{0}(E)v_{2}(E)\geqslant H^{n}.H^{n-2}(\mathrm{ch}_{1}(\pi_{*}E)^{2}-2\mathrm{ch}_{0}(\pi_{*}E)\mathrm{ch}_{2}(\pi_{*}E)) \geqslant 0.
	\end{align*}
\end{prop}

\begin{pf}
	Let $E$ be a slope semistable sheaf in $\mathrm{Coh}_{0}(X)$, then by Proposition \ref{pipushforwardofslopestablesheaf}, $\pi_{*}E \in \mathrm{Coh}(Y)$ is also slope semistable, and we have the classical Bogomolov-Gieseker inequality:
	\begin{align*}
		v_{1}(E)^{2}-2v_{0}(E)v_{2}(E) & =(H^{n-1}\mathrm{ch}_{1}(\pi_{*}E))^{2}-2H^{n}\mathrm{ch}_{0}(\pi_{*}E)H^{n-2}\mathrm{ch}_{2}(\pi_{*}E)) \\
		& \geqslant H^{n}.H^{n-2}(\mathrm{ch}_{1}(\pi_{*}E)^{2}-2\mathrm{ch}_{0}(\pi_{*}E)\mathrm{ch}_{2}(\pi_{*}E)) \\
		& \geqslant 0.
		\qedhere
	\end{align*}
\end{pf}

We will denote the open subset $\{(\beta,\alpha) \in \mathbb{R}^{2} \mid \alpha>\displaystyle\frac{1}{2}\beta^{2}\}$ of $\mathbb{R}^{2}$ by $U$.

\begin{defn}
	Given $(\beta,\alpha) \in U$, we define the group morphism:
	\begin{align*}
		Z^{\beta,\alpha}=-v_{2}+\alpha v_{0}+\sqrt{-1}(v_{1}-\beta v_{0}) \colon \Lambda \to \mathbb{C}.
	\end{align*}
\end{defn}

\begin{prop}
	For any $(\beta,\alpha) \in U$, the group homomorphism $Z^{\beta,\alpha}$ is a weak stability function on $\mathrm{Coh}^{\beta}_{0}(X)$ and
	\begin{align*}
		\{E \in \mathrm{Coh}^{\beta}_{0}(X) \mid Z^{\beta,\alpha}(E)=0\}=\{E \in \mathrm{Coh}_{0}(X) \mid \dim E \leqslant n-3\}.
	\end{align*}Moreover, if $\dim Y=2$, then $Z^{\beta,\alpha}$ is a stability function.
\end{prop}

\begin{pf}
	The proof is similar to \cite[Lemma 3.2.1]{bayer2014bridgeland}.
\end{pf}

In this paper, we denote by $\nu^{\beta,\alpha}$ the slope function with respect to the weak stability function $Z^{\beta,\alpha}$.

In order to prove the existence of Harder-Narasimhan property of $Z^{\beta,\alpha}$, we will first discuss it when $\beta,\alpha$ are both rational numbers and then use the deformation property to prove the other case as in \cite{bayer2016space}.

\begin{prop}
	For any $\beta \in \mathbb{Q}$, the abelian category $\mathrm{Coh}^{\beta}_{0}(X)$ is Noetherian.
\end{prop}

\begin{pf}
	See Proposition \ref{Noetherian1}.
\end{pf}

Finally we have the following proposition:

\begin{prop}
	There is a continous embedding as follows:
	\begin{align*}
		U \to \mathrm{Stab}^{w}_{H}(D^{b}_{0}(X)), (\beta,\alpha) \mapsto \sigma^{\beta,\alpha}=(Z^{\beta,\alpha},\mathrm{Coh}^{\beta}_{0}(X)). 
	\end{align*}
\end{prop}

\begin{pf}
	The proof is similar to \cite[Appendix B]{bayer2016space}. 
\end{pf}

\begin{defn}
	We define the generalized discriminant:
	\begin{align*}
		\overline{\Delta}_{H}(E)=v_{1}(E)^{2}-2v_{0}(E)v_{2}(E)
	\end{align*} for any $E \in D_{0}^{b}(X)$.
\end{defn}

\begin{rmk}
	(1) The generalized discriminant $\overline{\Delta}_{H}$ determine a quadratic form on the $\mathbb{R}$-vector space $(\Lambda/\Lambda_{0})_{\mathbb{R}}$. Moreover we have $\overline{\Delta}_{H}(E) \geqslant 0$ for any slope semistable sheaf $E \in \mathrm{Coh}_{0}(X)$ by Bogomolov-Gieseker type inequality.
	
	(2) Let $x,y,z$ be the coordinate of the vector space $(\Lambda/\Lambda_{0})_{\mathbb{R}}$. Then the equation of $\Delta$ is $y^{2}=2xz$. For any $(\beta,\alpha) \in U$, the kernel of $Z^{\beta,\alpha}$ in $(\Lambda/\Lambda_{0})_{\mathbb{R}}$ is identified with the subspace $\{(x,y,z) \in \mathbb{R}^{3} \mid z=\alpha x, y=\beta x\}$. Meanwhile, the restriction of $\overline{\Delta}_{H}(v)$  on the kernel of $Z^{\beta,\alpha}$ is given by the function $(\beta^{2}-2\alpha)x^{2}$. Note that $(\beta^{2}-2\alpha)x^{2} \leqslant 0$ with equality holds if and only if $x=0$. In other words, $\overline{\Delta}_{H}$ is negative definite on $\ker Z^{\beta,\alpha}$ for all $(\beta,\alpha) \in U$.
	
	(3) For any $\beta \in \mathbb{R}$ and $E \in D_{0}^{b}(X)$, it is easy to check that:
	\begin{align*}
		\overline{\Delta}_{H}(E)=(v_{1}^{\beta}(E))^{2}-2v_{2}^{\beta}(E)v_{0}^{\beta}(E).
	\end{align*}
\end{rmk}

We have the following important property of the generalized discriminant:

\begin{prop}
	For any $(\beta,\alpha) \in U$ and $\nu^{\beta,\alpha}$-semistable obejct $E \in D_{0}^{b}(X)$, we have the following inequality:
	\begin{align*}
		\overline{\Delta}_{H}(E) \geqslant 0.
	\end{align*} 
\end{prop}

\begin{pf}
		The proof is similar to \cite[Theorem 3.5]{bayer2016space}. 
\end{pf}

There is a wall and chamber structure on $U$ associated with some object in $D_{0}^{b}(X)$. The reader can refer to \cite[Proposition 4.1]{feyzbakhsh2021application} for the proof. Before that, we give a useful lemma:

\begin{prop}\label{wallandchamberlem1}
	Suppose that $(\beta_{0},\alpha_{0}) \in U$ and $E \in \mathrm{Coh}^{\beta_{0}}_{0}(X)$ is a $\nu^{\beta_{0},\alpha_{0}}$-semistable object. Then there is a line $l$ in $\mathbb{R}^{2}$ which passes through the point $(\beta_{0},\alpha_{0})$, such that for any $(\beta_{1},\alpha_{1}) \in l \cap U$, $E \in \mathrm{Coh}^{\beta_{1}}_{0}(X)$ and it is $\nu^{\beta_{1},\alpha_{1}}$-semistable with the same slope $\nu^{\beta_{0},\alpha_{0}}(E)$. Moreover, if $\nu^{\beta_{0},\alpha_{0}}(E)=+\infty$, then the equation of line is $\beta=\beta_{0}$; if $\nu^{\beta_{0},\alpha_{0}}(E)=\nu<+\infty$, then the equation of line is $\alpha-\alpha_{0}=\nu(\beta-\beta_{0})$.
\end{prop}

\begin{pf}
	We assume that $E \in \mathrm{Coh}^{\beta_{0}}_{0}(X)$. If $\nu^{\beta_{0},\alpha_{0}}(E)=+\infty$, then $v_{1}^{\beta}(E)=0$ and we get our conclusion.
	
	Now we suppose that $\nu^{\beta_{0},\alpha_{0}}(E)=\nu_{0}<+\infty$. Because $E \in \mathrm{Coh}^{\beta_{0}}_{0}(X)$, then $\mathcal{H}^{-1}(E) \in \mathrm{Coh}^{\leqslant \beta_{0}}_{0}(X)$ and 
	$\mathcal{H}^{0}(E) \in \mathrm{Coh}^{>\beta_{0}}_{0}(X)$. Suppose that $(\beta_{1},\alpha_{1}) \in U$ with $\alpha_{1}-\alpha_{0}=\nu_{0}(\beta_{1}-\beta_{0})$ and $\mathcal{H}^{-1}(E) \notin \mathrm{Coh}^{\leqslant \beta_{1}}_{0}(X)$. Let $F_{0}$ be the factor of the Harder-Narasimhan filtration of $\mathcal{H}^{-1}(E)$ with maximal slope $\mu$. Then $\beta_{1}<\mu \leqslant \beta_{0}$. The point 
	$(\mu,\nu_{0}(\mu-\beta_{0})+\alpha_{0})$ is on the segement between $(\beta_{0},\alpha_{0})$ and $(\beta_{1},\alpha_{1})$ and it must belong to $U$ because $U$ is convex. Thus we have
	\[
		\nu_{0}(\mu-\beta_{0})+\alpha_{0}>\frac{1}{2}\mu^{2}.
	\] But $E$ is $\nu^{\beta_{0},\alpha_{0}}$-semistable and $F_{0}[1]$ can be seen as a subobject of $E$ in $\mathrm{Coh}^{\beta_{0}}_{0}(X)$. Thus $\nu^{\beta_{0},\alpha_{0}}(F_{0}[1]) \leqslant \nu_{0}$. In other words, we have:
	\[
		\frac{v_{2}(F_{0})}{v_{0}(F_{0})} \geqslant \nu_{0}(\mu-\beta_{0})+\alpha_{0}
	\] It contradicts with the above inequality by Proposition \ref{classicalBG}.

	Similarly, we can also prove that $\mathcal{H}^{0}(E) \in \mathrm{Coh}^{>\beta_{1}}_{0}(X)$. Thus $E \in \mathrm{Coh}^{\beta_{1}}_{0}(X)$.
	Now we calculate $\nu^{\beta_{1},\alpha_{1}}(E)$. If $\nu^{\beta_{1},\alpha_{1}}(E)=+\infty$ and $\nu^{\beta_{0},\alpha_{0}}(E)<+\infty$, then $\mathcal{H}^{-1}(E)$ is a slope semistable sheaf with slope $\beta_{1}$. By the same argument again, we have $\displaystyle\frac{1}{2}\beta_{1}^{2} \geqslant \alpha_{0}+\nu_{0}(\beta_{1}-\beta_{0})=\alpha_{1}$ and it is a contradicition. Thus
	\begin{align*}
		\nu^{\beta_{1},\alpha_{1}}(E)=\frac{v_{2}(E)-\alpha_{1}v_{0}(E)}{v_{1}(E)-\beta_{1}v_{0}(E)}=\frac{v_{2}(E)-\alpha_{0}v_{0}(E)-(\alpha_{1}-\alpha_{0})v_{0}(E)}{v_{1}(E)-\beta_{0}v_{0}(E)-(\beta_{1}-\beta_{0})v_{0}(E)}=\nu_{0}.
	\end{align*}
	
	Finally, we want to prove that $E$ is $\nu^{\beta_{1},\alpha_{1}}$-semistable. We denote $V$ to be the subset of the connected subset $l \cap U$:
	\begin{align*}
		V=\{(\beta_{1},\alpha_{1}) \in l \cap U \mid E \mbox{ is } \nu^{\beta_{1},\alpha_{1}}\mbox{-semistable}\}.
	\end{align*} First it is a nonempty closed subset of $l \cap U$. We only need to prove that it is also an open subset of $l \cap U$. Assume that $E \in \mathcal{P}^{\beta_{1},\alpha_{1}}(\phi)$ and take $\epsilon>0$, such that $(\phi-2\epsilon,\phi+2\epsilon) \subseteq (0,1)$. 
	If there exists $(\beta_{2},\alpha_{2}) \in l$ near $\sigma^{\beta_{1},\alpha_{1}}$ with $E$ is not $\nu^{\beta_{2},\alpha_{2}}$-semistable. Then there must be a short exact sequence in $\mathrm{Coh}^{\beta_{2}}_{0}(X)$:
	\begin{align*}
		0 \to F \to E \to G \to 0
	\end{align*} with $\nu^{\beta_{2},\alpha_{2}}(F)>\nu^{\beta_{2},\alpha_{2}}(E)=\nu_{0}$. In fact, the above sequence is also a short exact sequence in $\mathrm{Coh}^{\beta_{1}}_{0}(X)$ because $F,G \in  \mathcal{P}^{\beta_{2},\alpha_{2}}(\phi-\epsilon,\phi+\epsilon) \subset \mathcal{P}^{\beta_{1},\alpha_{1}}(\phi-2\epsilon,\phi+2\epsilon) \subset \mathrm{Coh}^{\beta_{1}}_{0}(X)$. 
	
	Then
	\begin{align*}
		\nu^{\beta_{1},\alpha_{1}}(F)=\frac{v_{2}(F)-\alpha_{1}v_{0}(F)}{v_{1}(F)-\beta_{1}v_{0}(F)}=\frac{v_{2}(F)-\alpha_{2}v_{0}(F)-(\alpha_{1}-\alpha_{2})v_{0}(F)}{v_{1}(F)-\beta_{2}v_{0}(F)-(\beta_{1}-\beta_{2})v_{0}(F)}.
	\end{align*}If $v_{0}(F)=0$, then $\nu^{\beta_{1},\alpha_{1}}(F)=\displaystyle\frac{v_{2}(F)}{v_{1}(F)}=\nu^{\beta_{2},\alpha_{2}}(F)>\nu_{0}$. If $v_{0}(F) \neq 0$, then we still have $\nu^{\beta_{1},\alpha_{1}}(F)>\nu_{0}$. However, it is a contradiction because $E$ is $\nu^{\beta_{1},\alpha_{1}}$-semistable.
	
	Thus $V$ is an open set of $l \cap U$. Thus $V=l \cap U$ and we complete the proof.
\end{pf}

\begin{prop}\cite[Proposition 4.1]{feyzbakhsh2021application}\label{wallandchamber1}
	Fix an object $E \in D^{b}_{0}(X)$ such that $v_{0}(E) \neq 0$ or $v_{1}(E) \neq 0$. (For example, any object in $\mathrm{Coh}_{0}^{\beta}(X)$ with finite tilt-slope satisfies this condition.) Then there exists a wall and chamber structure in $U$: there exists a set of lines $\{l_{i}\}_{i \in I}$ in $\mathbb{R}^{2}$ such that the collection of segments $\{l_{i} \bigcap U\}_{i \in I}$ (called "walls") are locally finite in $U$ and satisfy:
	
	(1) If $v_{0}(E)\neq 0$, then all lines $l_{i}$ pass through the point $\left(\displaystyle\frac{v_{1}(E)}{v_{0}(E)},\frac{v_{2}(E)}{v_{0}(E)}\right)$; 
	
	if $v_{0}(E)=0$, then all lines $l_{i}$ has the same slope $\displaystyle\frac{v_{2}(E)}{v_{1}(E)}$.
	
	(2) The $\nu_{0}^{\beta,\alpha}$-(semi)stability of $E$ is unchanged as $(\beta,\alpha)$ varies within any connected component of $U \setminus \bigcup_{i \in I}(l_{i} \bigcap U)$.
	
	(3) For any wall $l \in I$, there is $k_{i} \in \mathbb{Z}$ and a morphism $f_{i} \colon F_{i} \to E[k_{i}]$ in $D_{0}^{b}(X)$ such that:
	
	(i) For any $(\beta,\alpha) \in l_{i} \bigcap U$, $F_{i},E[k_{i}]$ and the cone $G_{i}$ of $f_{i}$ lie in the heart $\mathrm{Coh}_{0}^{\beta}(X)$. 
	
	(ii) For any $(\beta,\alpha) \in l_{i} \bigcap U$, $F_{i}, E[k_{i}], G_{i}$ are both $\nu^{\beta,\alpha}$-semistable with the same slope.
	
	(iii) $f_{i} \colon F_{i} \to E[k_{i}] $ strictly destabilizes $E[k_{i}]$ for $(\beta,\alpha)$ in one of the two chambers adjacent to the wall $l_{i} \bigcap U$.
\end{prop}

There are also some useful methods to give an example of tilt semistable object. Let $\mathfrak{D}$ denote the full subcategory of $D_{0}^{b}(X)$ consisting of $E \in \{E \in D_{0}^{b}(X) \mid \mathcal{H}^{l}_{\mathcal{A}}(E)=0 \mbox{ unless } l=-1,0\}$ satisfying one of the following conditions:

(1) $\mathcal{H}^{-1}(E)=0$ and $v_{0}(\mathcal{H}^{0}(E))>0$ with $\mathcal{H}^{0}(E)$ is slope semistable.

(2) $\mathcal{H}^{-1}(E)=0$ and $v_{0}(\mathcal{H}^{0}(E))=0$, $v_{1}(\mathcal{H}^{0}(E))\neq 0$ with any nonzero subobject $C$ of $\mathcal{H}^{0}(E)$ does not satisfy $\dim C \leqslant n-2$.

(3) $\mathcal{H}^{-1}(E)=0$ and $\dim \mathcal{H}^{0}(E) \leqslant n-2$.

(4) $\mathcal{H}^{-1}(E)$ is slope semistable and $\dim \mathcal{H}^{0}(E) \leqslant n-2$. 

\begin{prop}\cite[Lemma 7.2.1 \& Lemma 7.2.2]{bayer2014bridgeland}\label{smalltiltslope}
	Let $\beta \in \mathbb{Q}$. Suppose that
	\begin{align*}
		c=\min\{v_{1}^{\beta}(E)>0 \mid E \in \mathrm{Coh}_{0}^{\beta}(X)\}.
	\end{align*}
	If $E \in \mathrm{Coh}_{0}^{\beta}(X)$ satisfies $v_{1}^{\beta}(E)=c$, then we have:
	
	(1) If there is $\alpha_{0}>\displaystyle\frac{1}{2}\beta^{2}$, such that $E$ is $\nu^{\beta,\alpha_{0}}$-semistable, then $E$ is $\nu^{\beta,\alpha}$-stable for any $\alpha>\displaystyle\frac{1}{2}\beta^{2}$. Moreover, $E \in \mathfrak{D}$ and $\mathrm{Hom}_{D_{0}^{b}(X)}(C,E)=0$ for any $\dim C \leqslant n-2$.
	
	(2) If $E \in \mathfrak{D}$ and $\mathrm{Hom}_{D_{0}^{b}(X)}(C,E)=0$ for any $\dim C \leqslant n-2$, then $E$ is $\nu^{\beta,\alpha}$-stable for any $\alpha>\displaystyle\frac{1}{2}\beta^{2}$.
\end{prop}

\begin{prop}\label{ipushforwardoftiltstableobject}
	We fix $(\beta,\alpha) \in U$. If $0 \neq E_{0} \in \mathrm{Coh}^{\beta}(Y)$ is $\nu^{\beta,\alpha}$-(semi)stable, then $0 \neq i_{*}E_{0} \in \mathrm{Coh}_{0}^{\beta}(X)$ is also $\nu^{\beta,\alpha}$-(semi)stable with the same tilt-slope.
\end{prop}

\begin{pf}
	Consider any exact sequence
	\begin{align*}
		0 \to F \to i_{*}E_{0} \to G \to 0
	\end{align*} in $\mathrm{Coh}^{\beta}_{0}(X)$. Use the functor $\pi_{*}$ on it, the we have the exact sequence in $\mathrm{Coh}^{\beta}(\mathbb{P}^{n})$:
	\begin{align*}
		0 \to \pi_{*}F \to E_{0} \to \pi_{*}G \to 0.
	\end{align*} Because, $E_{0}$ is $\nu^{\beta,\alpha}$-(semi)stable, we have
	\begin{align*}
		\nu^{\beta,\alpha}(F)=\nu^{\beta,\alpha}(\pi_{*}F)<(\leqslant)\nu^{\beta,\alpha}(\pi_{*}G)=\nu^{\beta,\alpha}(G),
	\end{align*} which implies that $E$ is also $\nu^{\beta,\alpha}$-(semi)stable with the same tilt-slope.
\end{pf}

\subsection{Bogomolov-Gieseker type inequality with $\mathrm{ch}_{3}$}\label{subsectionBG}
Now, we will assume that $\dim Y \geqslant 3$. The following definition is a generalization of \cite{bayer2014bridgeland}.

Because the central charge $Z^{\beta,\alpha}$ on $\mathrm{Coh}_{0}^{\beta}(X)$ satisfies the Harder-Narasimhan property on $\mathrm{Coh}^{\beta}_{0}(X)$ for any $(\beta,\alpha) \in U$, we have the following definition:
 
\begin{defn}
	For any $(\beta,\alpha) \in U$, let $(\mathrm{Coh}^{>\beta,\alpha}_{0}(X),\mathrm{Coh}^{\leqslant\beta,\alpha}_{0}(X))$ be the torsion pair of $\mathrm{Coh}_{0}^{\beta}(X)$ determined by:
	\begin{itemize}
		\item $\mathrm{Coh}^{>\beta,\alpha}_{0}(X)$ is the smallest extension-closed subcategory which contains $\nu^{\beta,\alpha}$-semistable objects $E \in \mathrm{Coh}_{0}^{\beta}(X)$ of slope $>\beta$;
		\item $\mathrm{Coh}^{\leqslant \beta,\alpha}_{0}(X)$ is the smallest extension-closed subcategory which contains $\nu^{\beta,\alpha}$-semistable objects $E \in \mathrm{Coh}_{0}^{\beta}(X)$ of slope $\leqslant \beta$.
	\end{itemize}
	Let $\mathrm{Coh}^{\beta,\alpha}_{0}(X)$ be the tilting heart of $\mathrm{Coh}_{0}^{\beta}(X)$ with respect to the torsion pair $(\mathrm{Coh}^{>\beta,\alpha}_{0}(X),\mathrm{Coh}^{\leqslant\beta,\alpha}_{0}(X))$.
\end{defn}

\begin{lem}\label{imaginarypart2}
	Let $(\beta,\alpha) \in U$. For any object $0 \neq E \in \mathrm{Coh}^{\beta,\alpha}_{0}(X)$, we have $v_{2}^{\beta}(E)-(\alpha-\displaystyle\frac{1}{2}\beta^{2})v_{0}^{\beta}(E) \geqslant 0$. Moreover, if $v_{2}^{\beta}(E)-(\alpha-\displaystyle\frac{1}{2}\beta^{2})v_{0}^{\beta}(E)=0$, then $\dim \mathcal{H}^{0}_{\beta}(E) \leqslant n-3$ or $\mathcal{H}^{0}_{\beta}(E)=0$ and $\mathcal{H}^{-1}_{\beta}(E) \in \mathrm{Coh}^{\beta}_{0}(X)$ is a $\nu^{\beta,\alpha}$-semistable object with tilt-slope $\beta$ or $\mathcal{H}^{-1}_{\beta}(E)=0$.
\end{lem}

\begin{pf}
	For any object $E \in \mathrm{Coh}^{\beta,\alpha}_{0}(X)$, we have a short exact sequence in $\mathrm{Coh}^{\beta,\alpha}_{0}(X)$ as follows:
	\begin{align*}
		0 \to F[1] \to E \to T \to 0,
	\end{align*}where $T \in \mathrm{Coh}^{>\beta,\alpha}_{0}(X)$ and $F \in \mathrm{Coh}^{\leqslant \beta,\alpha}_{0}(X)$. Then
	\begin{align*}
		v_{2}^{\beta}(E)-(\alpha-\displaystyle\frac{1}{2}\beta^{2})v_{0}^{\beta}(E)=v_{2}^{\beta}(T)-(\alpha-\displaystyle\frac{1}{2}\beta^{2})v_{0}^{\beta}(T)-(v_{2}^{\beta}(F)-(\alpha-\displaystyle\frac{1}{2}\beta^{2})v_{0}^{\beta}(F)).
	\end{align*} For any nonzero object $F \in \mathrm{Coh}^{\leqslant \beta,\alpha}_{0}(X)$, we have $\nu^{\beta,\alpha}(F) \leqslant \nu^{\beta,\alpha}_{\max}(F) \leqslant \beta$. In other words, we have  $v_{2}^{\beta}(F)-(\alpha-\displaystyle\frac{1}{2}\beta^{2})v_{0}^{\beta}(F) \leqslant 0$ with equality holds if and only if $F$ is a $\nu^{\beta,\alpha}$-semistable object with tilt-slope $\beta$. 
	
	Similarly, if $T \in \mathrm{Coh}^{>\beta,\alpha}_{0}(X)$ and $v_{1}(T)-\beta v_{0}(T)>0$, then we have $\nu^{\beta,\alpha}(T) \geqslant \nu^{\beta,\alpha}_{\min}(T)>\beta$, i.e. $v_{2}^{\beta}(T)-(\alpha-\displaystyle\frac{1}{2}\beta^{2})v_{0}^{\beta}(T)>0$. If $v_{1}(T)-\beta v_{0}(T)=0$, then $v_{2}^{\beta}(T)-(\alpha-\displaystyle\frac{1}{2}\beta^{2})v_{0}^{\beta}(T)=v_{2}(T)-\alpha v_{0}(T) \geqslant 0$ by the definition of weak stability condition. In conclusion, we have $v_{2}^{\beta}(T)-(\alpha-\displaystyle\frac{1}{2}\beta^{2})v_{0}^{\beta}(T) \geqslant 0$ with equality holds if and only if $Z^{\beta,\alpha}(T)=0$ if and only if $\dim T \leqslant n-3$ or $T=0$. Thus we get our conclusion.
\end{pf}

For any $\beta,\alpha,a \in \mathbb{R}$ with $\alpha>\displaystyle\frac{1}{2}\beta^{2}$, we could construct a group homomorphism as follows:
\begin{align}
	\label{centralcharge}Z^{\beta,\alpha,a}=-v^{\beta}_{3}+av^{\beta}_{1}+\sqrt{-1}\left(v^{\beta}_{2}-(\alpha-\frac{\beta^{2}}{2})v^{\beta}_{0}\right) \colon \Lambda \to \mathbb{C}.
\end{align}

We also have a conjecture about the Bogomolov-Gieseker type inequality as in \cite{bayer2014bridgeland}:

\begin{conj}\label{conj}
	For any $(\beta,\alpha) \in U$ and any $\nu^{\beta,\alpha}$-semistable object $E \in \mathrm{Coh}_{0}^{\beta}(X)$ with tilt-slope $\beta$, we have the following Bogomolov-Gieseker type inequality:
	\begin{align*}
		v^{\beta}_{3}(E) \leqslant \frac{2\alpha-\beta^{2}}{6}v^{\beta}_{1}(E).
	\end{align*}
\end{conj}

If the above conjecture is true, then we will have:
\begin{conj}\label{conj2}
	For any $(\beta,\alpha) \in U$ and $a>\displaystyle\frac{\omega^{2}}{6}=\frac{2\alpha-\beta^{2}}{6}$, the central charge $Z^{\beta,\alpha,a}$ is a weak stability function on the bounded heart $\mathrm{Coh}^{\beta,\alpha}_{0}(X)$ and 
	\begin{align*}
		\{E \in \mathrm{Coh}^{\beta,\alpha}_{0}(X) \mid Z^{\beta,\alpha,a}(E)=0\}=\{E \in \mathrm{Coh}_{0}(X) \mid \dim E \leqslant n-4\}.
	\end{align*}Moreover, if $\dim Y=3$, then $Z^{\beta,\alpha,a}$ is a stability function.
\end{conj}

\begin{prop}
	If Conjecture \ref{conj} holds, then the Conjecture \ref{conj2} holds.
\end{prop}

\begin{pf}
	By Lemma \ref{imaginarypart2}, we have $\Im Z^{\beta,\alpha,a}(E) \geqslant 0$ for any $E \in \mathrm{Coh}^{\beta,\alpha}_{0}(X)$. And when $\Im Z^{\beta,\alpha,a}(E)=0$, we have a short exact sequence in $\mathrm{Coh}^{\beta,\alpha}_{0}(X)$ as follows:
	\begin{align*}
		0 \to F[1] \to E \to T \to 0,
	\end{align*}where $T$ is a sheaf with $\dim T \leqslant n-3$ or $T=0$ and $F \in \mathrm{Coh}_{0}^{\beta}(X)$ is a $\nu^{\beta,\alpha}$-semistable object with tilt-slope $\beta$ or $F=0$. By Conjecture \ref{conj}, we have $v^{\beta}_{3}(F) \leqslant \displaystyle\frac{\omega^{2}}{6}v_{1}^{\beta}(F)<av_{1}^{\beta}(F)$ when $F \neq 0$. In other words, we have $\Re Z^{\beta,\alpha,a}(F)>0$ for any $F \neq 0$. Meanwhile, $\Re Z^{\beta,\alpha,a}(T)=-v_{3}(T) \leqslant 0$ with equality holds if and only if $\dim T \leqslant n-4$. Thus, $\Re Z^{\beta,\alpha,a}(E)=\Re Z^{\beta,\alpha,a}(T)-\Re Z^{\beta,\alpha,a}(F) \leqslant 0$ for any nonzero object $E$ with $\Im Z^{\beta,\alpha,a}(E)=0$.
\end{pf}

We will prove Conjecture \ref{conj} when $X$ is the canonical bundle of $\mathbb{P}^{3}$. Before that, we will prove more general results.

\subsection{Reduction of the Bogomolov-Gieseker type inequality}

\begin{defn}
	Let $(\beta,\alpha) \in U$, we say that the point $(\beta,\alpha)$ satisfies the Bogomolov-Gieseker type inequality if for any $\nu^{\beta,\alpha}$-semistable object $E \in \mathrm{Coh}_{0}^{\beta}(X)$ with tilt-slope $\beta$, we have the following inequality:
	\begin{align}
	\label{BGinequality} v^{\beta}_{3}(E) \leqslant \frac{2\alpha-\beta^{2}}{6}v^{\beta}_{1}(E).
	\end{align}
\end{defn}

First, in order to prove Conjecture \ref{conj}, we could reduce it to small $\omega=\sqrt{2\alpha-\beta^{2}}$ by Emanuele Macr\`i in \cite[Proposition 2.7]{macri2014generalized}.

\begin{prop}\label{reduce1}
	Assume that there exists $k \in \mathbb{R}^{+}$, such that all $(\beta,\alpha) \in U$ with $\sqrt{2\alpha-\beta^{2}}<k$ satisfiy the Bogomolov-Gieseker type inequality (\ref{BGinequality}). Then Conjecture \ref{conj} holds.
\end{prop}

The proof is the same as the one in \cite[Proposition 2.7]{macri2014generalized}, we just simply rewrite it in our new notations here.

For any $E \in D^{b}_{0}(X)$, we define a curve $\mathcal{C}_{E}$  in $U$ as follows:
\begin{align*}
	\mathcal{C}_{E}=\{(\beta,\alpha) \in U \mid v_{0}(E)\alpha=v_{0}(E)\beta^{2}-v_{1}(E)\beta+v_{2}(E), v_{1}(E)>\beta v_{0}(E)\}.
\end{align*} Note that:
\begin{itemize}
	\item If $v_{0}(E)=0$ and $v_{1}(E) \leqslant 0$, then $\mathcal{C}_{E}=\varnothing$.
	\item If $v_{0}(E)=0$ and $v_{1}(E)>0$, then $\mathcal{C}_{E}$ is the intersection of the line $\beta=\displaystyle\frac{v_{2}(E)}{v_{1}(E)}$ and $U$.
	\item If $v_{0}(E) \neq 0$ and $\overline{\Delta}_{H}(E)<0$, then $\mathcal{C}_{E}=\varnothing$ because of $\alpha>\displaystyle\frac{1}{2}\beta^{2}$.
	\item If $v_{0}(E) \neq 0$ and $\overline{\Delta}_{H}(E) \geqslant 0$, then $\mathcal{C}_{E}$ is a piece of the parabola $\alpha=\displaystyle\beta^{2}-\frac{v_{1}(E)}{v_{0}(E)}\beta+\frac{v_{2}(E)}{v_{0}(E)}$. Explicitly, when $v_{0}(E)>0$, then $(\beta,\alpha) \in \mathcal{C}_{E}$ implies that $\beta<\displaystyle\frac{v_{1}(E)-\sqrt{\overline{\Delta}_{H}(E)}}{v_{0}(E)}$; when $v_{0}(E)<0$, then $(\beta,\alpha) \in \mathcal{C}_{E}$ implies that $\beta>\displaystyle\frac{v_{1}(E)+\sqrt{\overline{\Delta}_{H}(E)}}{v_{0}(E)}$.
\end{itemize}

The reason why we introduce the curve is that if $\beta \in \mathbb{R}$ and $E \in \mathrm{Coh}_{0}^{\beta}(X)$ with $v_{1}(E)>\beta v_{0}(E)$, then $\nu^{\beta,\alpha}(E)=\beta$ if and only if $(\beta,\alpha) \in \mathcal{C}_{E}$. 
\begin{lem}\label{reducelem}
	Let $E \in D_{0}^{b}(X)$ with $\mathcal{C}_{E}$ non-empty and $(\beta,\alpha) \in \mathcal{C}_{E}$. Then $v_{3}^{\beta}(E) \leqslant \displaystyle\frac{2\alpha-\beta^{2}}{6}v_{1}^{\beta}(E)$ if and only if
	\begin{align*}
		& \frac{\beta\overline{\Delta}_{H}(E)}{v_{0}(E)} \leqslant \frac{v_{2}(E)v_{1}(E)}{v_{0}(E)}-3v_{3}(E), & \mbox{ if } v_{0}(E)\neq 0; \\
		& v_{3}(E)-\frac{v_{2}(E)^{2}}{2v_{1}(E)} \leqslant \frac{2\alpha-\beta^{2}}{6}v_{1}(E), & \mbox{ if } v_{0}(E)=0;
	\end{align*}
\end{lem}

\begin{pf}
	We can verify them directly.
\end{pf}

\begin{pf}(Proposition \ref{reduce1})
	Assume that there exists $(\beta_{0},\alpha_{0}) \in U$, and $\nu^{\beta_{0},\alpha_{0}}$-stable object $E_{0} \in \mathrm{Coh}_{0}^{\beta_{0}}(X)$ with tilt-slope $\beta_{0}$, which does not satisfy the Bogomolov-Gieseker type inequality, in other words, $v_{3}^{\beta_{0}}(E_{0})> \displaystyle\frac{2\alpha_{0}-\beta_{0}^{2}}{6}v_{1}^{\beta_{0}}(E_{0})$. 
	
	Consider the path $\gamma_{0}(t)=(\beta_{0}(t),\alpha_{0}(t))$ from $(\beta_{0},\alpha_{0})$ along the curve $\mathcal{C}_{E_{0}}$ to the end point $(\beta_{E_{0}},\displaystyle\frac{1}{2}\beta_{E_{0}}^{2})$ on $\partial U$, where
	\begin{align*}
		\beta_{E_{0}}=\begin{cases}
			\displaystyle\frac{v_{1}(E_{0})-\sqrt{\overline{\Delta}_{H}(E_{0})}}{v_{0}(E_{0})}, & v_{0}(E_{0})>0;\\
			\displaystyle\frac{v_{1}(E_{0})+\sqrt{\overline{\Delta}_{H}(E_{0})}}{v_{0}(E_{0})}, & v_{0}(E_{0})<0;\\
			\displaystyle\frac{v_{2}(E_{0})}{v_{1}(E_{0})}, & v_{0}(E_{0})=0.
		\end{cases}
	\end{align*}
	
	Claim that: $v_{3}^{\beta_{0}(t)}(E_{0})> \displaystyle\frac{2\alpha_{0}(t)-\beta_{0}(t)^{2}}{6}v_{1}^{\beta_{0}(t)}(E_{0})$ for any $t \in [0,1)$. First we can verify that the continous function $\sqrt{2\alpha_{0}(t)-\beta_{0}(t)^{2}}$ of $t$ decreases directly. When $v_{0}(E_{0})=0$, then by Lemma \ref{reducelem}, we have $\displaystyle v_{3}(E_{0})-\frac{v_{2}(E_{0})^{2}}{2v_{1}(E_{0})}>\frac{2\alpha_{0}-\beta_{0}^{2}}{6}v_{1}(E_{0}) \geqslant \frac{2\alpha_{0}(t)-\beta_{0}(t)^{2}}{6}v_{1}(E_{0})$ for any $t \in [0,1)$, and we get the conclusion. Similarly, we could also prove it when $v_{0}(E_{0}) \neq 0$ by showing that the function $\displaystyle\frac{\beta_{0}(t)}{v_{0}(E_{0})}$ of $t$ increases. Therefore, on the other hand, by our assumption $E_{0}$ cannot be tilt semistable object when $\sqrt{2\alpha_{0}(t)-\beta_{0}(t)^{2}}<k$. Then there exists $t' \in (0,1)$, such that $E_{0}$ is a strictly $\nu^{(\beta_{0}(t'),\alpha_{0}(t'))}$-semistable object by Proporsition \ref{pathandwallcrossing}. We denote $\beta_{0}(t')$ by $\beta_{1}$ and $\alpha_{0}(t')$ by $\alpha_{1}$. Asuume that $E_{0}[l_{0}] \in \mathrm{Coh}_{0}^{\beta_{1}}(X)$ and $\nu^{\beta_{1},\alpha_{1}}(E_{0}[l_{0}])=\beta_{1}$ because $(\beta_{1},\alpha_{1}) \in \mathcal{C}_{E_{0}}$. Take a Jordan-H\"older factor $E_{1}$ of $E_{0}[l_{0}]$ in $\mathrm{Coh}_{0}^{\beta_{1}}(X)$ with $v_{3}^{\beta_{1}}(E_{1})> \displaystyle\frac{2\alpha_{1}-\beta_{1}^{2}}{6}v_{1}^{\beta_{1}}(E_{1})$. 
	 
	We use the same method to get a sequence of pairs $\{(\beta_{n},\alpha_{n})\}_{n=0}^{\infty}$ and a sequence of objects $\{E_{n}\}_{n=0}^{\infty}$ such that $(\beta_{n},\alpha_{n}) \in U$, $E_{n} \in \mathrm{Coh}_{0}^{\beta_{n}}(X)$ which is $\nu^{\beta_{n},\alpha_{n}}$-stable object with tilt-slope $\beta_{n}$ and $v_{3}^{\beta_{n}}(E_{n})> \displaystyle\frac{2\alpha_{n}-\beta_{n}^{2}}{6}v_{1}^{\beta_{n}}(E_{n})$ and $\sqrt{2\alpha_{n-1}-\beta_{n-1}^{2}}>\sqrt{2\alpha_{n}-\beta_{n}^{2}}>k$ for any $n \in \mathbb{N}^{+}$. We assume that $\displaystyle\lim_{n \to \infty}\sqrt{2\alpha_{n}-\beta_{n}^{2}}=K \geqslant k>0$.
	
	Meanwhile, we have $v^{\beta_{n+1}}_{1}(E_{n+1})^{2}< v^{\beta_{n+1}}_{1}(E_{n})^{2}$ because $E_{n+1}$ is a Jordan-H\"older factor of $E_{n}[l_{n}]$ for some $l_{n} \in \mathbb{Z}$ in $\mathrm{Coh}^{\beta_{n+1}}_{0}(X)$. From this inequality, we can get
	\begin{align*}
		&\overline{\Delta}_{H}(E_{n+1})+(2\alpha_{n+1}-\beta_{n+1}^{2})v_{0}(E_{n+1})^{2}=v^{\beta_{n+1}}_{1}(E_{n+1})^{2}\\
		&<v^{\beta_{n+1}}_{1}(E_{n})^{2}=\overline{\Delta}_{H}(E_{n})+(2\alpha_{n+1}-\beta_{n+1}^{2})v_{0}(E_{n})^{2}\\ &\leqslant 	\overline{\Delta}_{H}(E_{n})+(2\alpha_{n}-\beta_{n}^{2})v_{0}(E_{n})^{2}.
	\end{align*}It implies that $\overline{\Delta}_{H}(E_{n})+K^{2}v_{0}(E_{n})^{2} \leqslant \overline{\Delta}_{H}(E_{n})+(2\alpha_{n}-\beta_{n}^{2})v_{0}(E_{n})^{2}< \overline{\Delta}_{H}(E_{0})+(2\alpha_{0}-\beta_{0}^{2})v_{0}(E_{0})^{2}$. It means that $\overline{\Delta}_{H}(E_{n})$ and $v_{0}(E_{n})^{2}$ are all bounded. We will prove $|v_{1}(E_{n})|$ is also bounded in the next step.
	
	We claim that: for any $n \in \mathbb{N}^{+}$, we have:
	\begin{align*}
		|\beta_{n}-\beta_{n+1}| \leqslant \sqrt{2\alpha_{n}-\beta_{n}^{2}}-\sqrt{2\alpha_{n+1}-\beta_{n+1}^{2}}
	\end{align*} Note that $(\beta_{n+1},\alpha_{n+1}) \in \mathcal{C}_{E_{n}}$. If $v_{0}(E_{n})=0$, then $\beta_{n}=\beta_{n+1}$; if $v_{0}(E_{n}) \neq 0$, let $\displaystyle f_{n}(t)=\sqrt{2\alpha_{n}(t)-\beta_{n}(t)^{2}}=\sqrt{\beta_{n}(t)^{2}-2\frac{v_{1}(E_{n})}{v_{0}(E_{n})}\beta_{n}(t)+2\frac{v_{2}(E_{n})}{v_{0}(E_{n})}}$. Then
	\begin{align*}
		\left|\frac{df_{n}}{dt}\right|=\left|\frac{df_{n}}{d\beta_{n}}\right|\left|\frac{d\beta_{n}}{dt}\right|=\left|\frac{\beta_{n}(t)-\displaystyle\frac{v_{1}(E_{n})}{v_{0}(E_{n})}}{\sqrt{\beta_{n}^{2}(t)-2\displaystyle\frac{v_{1}(E_{n})}{v_{0}(E_{n})}\beta_{n}(t)+2\displaystyle\frac{v_{2}(E_{n})}{v_{0}(E_{n})}}}\right|\left|\frac{d\beta_{n}}{dt}\right| \geqslant \left|\frac{d\beta_{n}}{dt}\right|
	\end{align*} because $\overline{\Delta}_{H}(E_{n}) \geqslant 0$.
	Thus, we have:
	\begin{align*}
		\frac{\sqrt{2\alpha_{n}-\beta_{n}^{2}}-\sqrt{2\alpha_{n+1}-\beta_{n+1}^{2}}}{|\beta_{n}-\beta_{n+1}|} \geqslant 1.
	\end{align*}And we prove the claim.
	
	Now, we have:
	\begin{align*}
		& |v_{1}(E_{n})|=|v_{1}^{\beta_{n}}(E_{n})+\beta_{n}v_{0}(E_{n})| \\
		& \leqslant |v_{1}^{\beta_{n}}(E_{n})|+|\beta_{n}||v_{0}(E_{n})|=\sqrt{\overline{\Delta}_{H}(E_{n})+(2\alpha_{n}-\beta_{n}^{2})v_{0}(E_{n})^{2}}+|\beta_{n}||v_{0}(E_{n})| \\
		& \leqslant \sqrt{\overline{\Delta}_{H}(E_{0})+(2\alpha_{0}-\beta_{0}^{2})v_{0}(E_{0})^{2}}+(|\beta_{n}-\beta_{n-1}|+|\beta_{n-1}-\beta_{n-2}|+\cdots+|\beta_{1}-\beta_{0}|+|\beta_{0}|)|v_{0}(E_{n})| \\
		& \leqslant \sqrt{\overline{\Delta}_{H}(E_{0})+(2\alpha_{0}-\beta_{0}^{2})v_{0}(E_{0})^{2}}+(\sqrt{2\alpha_{0}-\beta_{0}^{2}}+|\beta_{0}|)|v_{0}(E_{n})|.
	\end{align*}
	Because $|v_{0}(E_{n})|$ is bounded, $|v_{1}(E_{n})|$ is also bounded.
	
	Finally, we will prove $|v_{2}(E_{n})|$ is also bounded. If $v_{0}(E_{n})=0$, then $v_{2}(E_{n})=\beta_{n}v_{1}(E_{n})$ and 
	\begin{align*}
		|v_{2}(E_{n})| \leqslant (|\beta_{n}-\beta_{n-1}|+|\beta_{n-1}-\beta_{n-2}|+\cdots+|\beta_{1}-\beta_{0}|+|\beta_{0}|)|v_{1}(E_{n})| \leqslant (\sqrt{2\alpha_{0}-\beta_{0}^{2}}+|\beta_{0}|)|v_{1}(E_{n})|.
	\end{align*}
	
	If $v_{0}(E_{n}) \neq 0$, then
	\begin{align*}
		|v_{2}(E_{n})| \leqslant |v_{0}(E_{n})v_{2}(E_{n})| \leqslant |\overline{\Delta}_{H}(E_{n})|+v_{1}(E_{n})^{2}.
	\end{align*}
	
	And because $|v_{1}(E_{n})|$ is bounded, $|v_{2}(E_{n})|$ is also bounded. But it is a contradiction with the sets $\{(v_{0}(E_{n}),v_{1}(E_{n}),v_{2}(E_{n}))\}_{n \in \mathbb{N}}$ is infinite.
	\end{pf}

We can also reduce the general Bogomolov-Gieseker inequality to rational case. Moreover, we have:

\begin{prop}\label{reduce2}
	Let $U'=\{(\beta,\alpha) \in U \mid (\beta,\alpha) \mbox{ does not satisfy the Bogomolov-Gieseker inequality.}\}$. Then for any $(\beta,\alpha) \in U'$, there exists an open neighbourhood $U''$ of $(\beta,\alpha)$ in $U$, such that $U'' \bigcap \mathbb{Q}^{2} \bigcap U' \neq \varnothing$.
\end{prop}

\begin{pf}
	 Assume that $(\beta_{0},\alpha_{0}) \in U'$ and there exists $\nu^{\beta_{0},\alpha_{0}}$-stable object $E \in \mathrm{Coh}_{0}^{\beta_{0}}(X)$ with tilt-slope $\beta_{0}$, which does not satisfy the Bogomolov-Gieseker type inequality, in other words, $v_{3}^{\beta_{0}}(E)> \displaystyle\frac{2\alpha_{0}-\beta_{0}^{2}}{6}v_{1}^{\beta_{0}}(E)$. By Proposition \ref{wallandchamber1}, there is an open subset $U_{1}$ of $U$, such that $E$ belong to $\mathrm{Coh}_{0}^{\beta}(X)$ and is $\nu^{\beta,\alpha}$-stable for any $(\beta,\alpha) \in U_{1}$. Meanwhile, functions $(\beta,\alpha) \mapsto v_{3}^{\beta}(E)-\displaystyle\frac{2\alpha-\beta^{2}}{6}v_{1}^{\beta}(E)$ and $(\beta,\alpha) \to v_{1}^{\beta}(E)$ are continuous. Thus we can find a smaller open neighbourhood $U'' \subseteq U_{1}$ of $(\beta_{0},\alpha_{0})$, such that for any $(\beta,\alpha) \in U_{2}$, we have $v_{3}^{\beta}(E)> \displaystyle\frac{2\alpha-\beta^{2}}{6}v_{1}^{\beta}(E)$ and $v_{1}^{\beta}(E)>0$. Fianlly, any point $(\beta,\alpha) \in \mathcal{C}_{E} \bigcap U''$ makes $\nu^{\beta,\alpha}(E)=\beta$. In particular, there are rational points on $\mathcal{C}_{E} \bigcap U''$ and we get our conclusion.
\end{pf}

We can also reduce Bogomolov-Gieseker type inequality to the case of $-\displaystyle\frac{1}{2} \leqslant \beta < \displaystyle\frac{1}{2}$ by acting the tensor functor $-\otimes_{X}\pi^{*}\mathcal{O}_{Y}(H)$. In general, we have the following proposition: 
	
\begin{prop}\label{reduce3}
	Let $(\beta,\alpha) \in U$. Then $(\beta,\alpha) \in U$ satisfies the Bogomolov-Gieseker type inequality if and only if $(\beta+1,\alpha+\beta+\displaystyle\frac{1}{2}) \in U$  satisfies the Bogomolov-Gieseker type inequality.
\end{prop}

\begin{pf}
	First, we can verify that the functor $\Phi= -\otimes_{X}\pi^{*}\mathcal{O}_{Y}(H)$ is an exact autoequivalence of $\mathcal{D}$, and it induces an isomorphism
	between $\Lambda$ as follows:
	\begin{align*}
		\Phi_{*} \colon \Lambda \to \Lambda, (v_{0}(E),\ldots,v_{n}(E)) \mapsto (v_{0}^{-1}(E),\ldots,v_{n}^{-1}(E)).
	\end{align*}Here $v^{-1}_{i}(E)$ means that	the intersection number of $H^{n-i}$ with $i$-th component of $\mathrm{ch}(\pi_{*}E)e^{H}$.
	
	It is not hard to see that $\Phi(\mathcal{Z},\mathrm{Coh}_{0}(X))=(\mathcal{Z}^{1},\mathrm{Coh}_{0}(X))$, where $\mathcal{Z}=-v_{1}+\sqrt{-1}v_{0}$ and $\mathcal{Z}^{1}=-v_{1}^{1}+\sqrt{-1}v_{0}^{1}=-v_{1}+v_{0}+\sqrt{-1}v_{0}$. Then if $E \in \mathrm{Coh}_{0}(X)$ is slope semistable sheaf with slope $\mu$, then $E \otimes_{X} \pi^{*}\mathcal{O}_{Y}(H)$ is also semistable but with slope $\mu+1$. Thus we have 
	\begin{align*}
		& \mathrm{Coh}^{>\beta}_{0}(X) \otimes_{X} \pi^{*}\mathcal{O}_{Y}(H)=\mathrm{Coh}^{>\beta+1}_{0}(X);\\
		& \mathrm{Coh}^{\leqslant \beta}_{0}(X) \otimes_{X} \pi^{*}\mathcal{O}_{Y}(H)=\mathrm{Coh}^{\leqslant \beta+1}_{0}(X).
	\end{align*} and
	\begin{align*}
		\mathrm{Coh}_{0}^{\beta}(X)\otimes_{X} \pi^{*}\mathcal{O}_{Y}(H)=\mathrm{Coh}_{0}^{\beta+1}(X).
	\end{align*}
	
	Meanwhile, we can verify that if $(\beta,\alpha) \in U$, then $(\beta+1,\alpha+\beta+\displaystyle\frac{1}{2}) \in U$ and $\Phi(Z^{\beta,\alpha},\mathrm{Coh}_{0}^{\beta}(X))=(Z',\mathrm{Coh}_{0}^{\beta+1}(X))$, where 
	\begin{align*}
		Z'=-v_{2}+v_{1}+\left(\alpha-\frac{1}{2}\right)v_{0}+\sqrt{-1}(v_{1}-(\beta+1)v_{0}).
	\end{align*}Note that for any $E' \in \mathrm{Coh}_{0}^{\beta+1}(X)$, $E'$ is semistable with respect to the above weak stability condition $(Z',\mathrm{Coh}_{0}^{\beta+1}(X))$ with slope $\nu$ if and only if $E'$ is $\nu^{\beta+1,\alpha+\beta+\frac{1}{2}}$-semistable with tilt-slope $\nu+1$. In particular, if $E \in \mathrm{Coh}_{0}^{\beta}(X)$ is $\nu^{\beta,\alpha}$-semistable with tilt slope $\nu$, then $E \otimes_{X} \pi^{*}\mathcal{O}_{Y}(H) \in \mathrm{Coh}_{0}^{\beta+1}(X)$ is $\nu^{\beta+1,\alpha+\beta+\frac{1}{2}}$-semistable with tilt-slope $\nu+1$.
	
	Now, suppose that $(\beta+1,\alpha+\beta+\displaystyle\frac{1}{2})$ satisfy the Bogomolov-Gieseker type inequality and $(\beta,\alpha) \in U$ doesn't satisfy the Bogomolov-Gieseker type inequality, then there exists a $\nu^{\beta,\alpha}$-semistable object $E$ with tilt-slope $\beta$, such that $v_{3}^{\beta}(E)>\displaystyle\frac{2\alpha-\beta^{2}}{6}v_{1}^{\beta}(E)$. However, $E \otimes_{X} \pi^{*}\mathcal{O}_{Y}(H) \in \mathrm{Coh}_{0}^{\beta+1}(X)$ is $\nu^{\beta+1,\alpha+\beta+\frac{1}{2}}$-semistable with tilt-slope $\beta+1$, then by our assumption, we have
	\begin{align*}
		v_{3}^{\beta+1}(E \otimes_{X} \pi^{*}\mathcal{O}_{Y}(H)) \leqslant \displaystyle\frac{2(\alpha+\beta+\frac{1}{2})-(\beta+1)^{2}}{6}v_{1}^{\beta+1}(E \otimes_{X} \pi^{*}\mathcal{O}_{Y}(H)).
	\end{align*}It is equivalent to the inequality $v_{3}^{\beta}(E) \leqslant \displaystyle\frac{2\alpha-\beta^{2}}{6}v_{1}^{\beta}(E)$ and it is a contradiction. Conversely, we can prove the remain part.
\end{pf}

When $\dim Y=3$ we can also use relative derived dual functor $R\mathcal{H}om(-,\pi^{*}\mathrm{det}\mathcal{E}_{0}^{\lor}[r]) \colon D^{b}_{0}(X) \to D^{b}_{0}(X)$ to reduce the Bogomolov-Gieseker type inequality similarly. We just mention our final consequence here. The basic proposition of relative derived dual functor and some proofs can refer to Appendix \ref{appendix2}. 

\begin{prop}\label{reduce4}
	Suppose that $\dim Y=3$. Let $(\beta,\alpha) \in U$. Then $(\beta,\alpha) \in U$ satisfies the Bogomolov-Gieseker type inequality if and only if $(-\beta,\alpha) \in U$ satisfies the Bogomolov-Gieseker type inequality.
\end{prop}

\begin{pf}
	See Proposition \ref{reduce4B}.
\end{pf}

Combine the above reduction propositions, we obtain the following conjecture which is equivalent to Conjecture \ref{conj}:

\begin{conj}\label{conj3}
	Suppose that $\dim Y=3$. Then for any $(\beta,\alpha) \in U \bigcap \mathbb{Q}^{2}$ with $-\displaystyle\frac{1}{2}\leqslant \beta \leqslant 0$, $\sqrt{2\alpha-\beta^{2}}<\displaystyle\frac{1}{2}$ and for any $\nu^{\beta,\alpha}$-semistable object $E \in \mathrm{Coh}_{0}^{\beta}(X)$ with tilt-slope $\beta$, we have the following Bogomolov-Gieseker type inequality:
	\begin{align*}
		v_{3}^{\beta}(E) \leqslant \displaystyle\frac{2\alpha-\beta^{2}}{6}v_{1}^{\beta}(E).
	\end{align*}
\end{conj}

\begin{prop}
	Conjecture \ref{conj3} is equivalent to Conjecture \ref{conj}.
\end{prop}

\begin{pf}
	By Proposition \ref{reduce1}, \ref{reduce2}, \ref{reduce3} and \ref{reduce4}.
\end{pf}

In the end of this subsection, we use the wall and chamber structure on $U$ to generalize the Bogomolov-Gieseker type inequality (\ref{BGinequality}) in a quadratic form as follows:

\begin{prop}\cite[Theorem 4.2]{bayer2016space}
	Conjecture \ref{conj} holds if and only if for any $(\beta,\alpha) \in U$ and any $\nu^{\beta,\alpha}$-semistable object $E$, we have:
	\begin{align}
	\label{BGinequality2} Q^{\beta,\alpha}(E)=\omega^{2}\overline{\Delta}_{H}(E)+4(v_{2}^{\beta}(E))^{2}-6v^{\beta}_{3}(E)v^{\beta}_{1}(E) \geqslant 0.
	\end{align}Here we denote $\sqrt{2\alpha-\beta^{2}}$ by $\omega$.
\end{prop}

\begin{pf}
	The proof is the same as the one in \cite[Theorem 4.2]{bayer2016space}, we just simply rewrite it in our new notations here.
	
	Suppose that Conjecture \ref{conj} holds. We show that for any $(\beta_{0},\alpha_{0}) \in U$ and a $\nu^{\beta_0, \alpha_0}$-semistable object $E$, the inequality (3) holds. Let $E$ be a $\nu^{\beta_{0},\alpha_{0}}$-semistable object. If $v_{1}^{\beta_{0}}(E)=0$, then it is obviously that $Q^{\beta_{0},\alpha_{0}}(E)=\omega_{0}^{2}\overline{\Delta}_{H}(E)+4(v_{2}^{\beta_{0}}(E))^{2} \geqslant 0$. If $\nu^{\beta_{0},\alpha_{0}}(E)=\beta'$, then there is a line segment $\alpha-\alpha_{0}=\beta'(\beta-\beta_{0})$ in $U$, such that $E$ is in $\mathrm{Coh}_{0}^{\beta}(X)$ and it is still $\nu^{\beta,\alpha}$-semistable object with tilt-slope $\beta'$ for any $(\beta,\alpha)$ on this line segment by Proposition \ref{wallandchamberlem1}. Let $\alpha'=\beta'(\beta'-\beta_{0})+\alpha_{0}$, we could verify that:
	\begin{align*}
		2\alpha'-\beta'^{2}=2\alpha_{0}-\beta_{0}^{2}+\beta'^{2}-2\beta'\beta_{0}+\beta_{0}^{2}>0.
	\end{align*}It means that $E \in \mathrm{Coh}_{0}^{\beta'}(X)$ is $\nu^{\beta',\alpha'}$-semistable object with tilt-slope $\beta'$, then by Conjecture \ref{conj}, we have:
	\begin{align*}
		v^{\beta'}_{3}(E) \leqslant \frac{2\alpha'-\beta'^{2}}{6}v^{\beta'}_{1}(E).
	\end{align*}
It yields:
\begin{align}\label{computation1}
	\nonumber
	&v_{3}^{\beta_{0}}(E)+(\beta_{0}-\beta')v_{2}^{\beta_{0}}(E)+\frac{1}{2}(\beta_{0}-\beta')^{2}v_{1}^{\beta_{0}}(E)+\frac{1}{6}(\beta_{0}-\beta')^{3}v_{0}^{\beta_{0}}(E)  \\ 
	&\geqslant \frac{2\alpha_{0}-\beta_{0}^{2}+(\beta_{0}-\beta')^{2}}{6}(v_{1}^{\beta_{0}}(E)+(\beta_{0}-\beta')v_{0}^{\beta_{0}}(E)).
\end{align}

And we know $\beta'=\displaystyle\frac{v_{2}(E)-\alpha_{0}v_{0}(E)}{v_{1}^{\beta_{0}}(E)}=\frac{v_{2}^{\beta_{0}}(E)+\beta_{0}v_{1}^{\beta_{0}}(E)-(\alpha_{0}^{2}-\frac{1}{2}\beta_{0}^{2})v_{0}^{\beta_{0}}(E)}{v_{1}^{\beta_{0}}(E)}$. In other words,
\begin{align}\label{computation2}
	\beta_{0}-\beta'=\frac{-v_{2}^{\beta_{0}}(E)+(\alpha_{0}^{2}-\frac{1}{2}\beta_{0}^{2})v_{0}^{\beta_{0}}(E)}{v_{1}^{\beta_{0}}(E)}.
\end{align} Combine (\ref{computation1}) and (\ref{computation2}), we have
\begin{align}\label{computation3}
	v_{3}^{\beta_{0}}(E)+(\beta_{0}-\beta')v_{2}^{\beta_{0}}(E)+\frac{1}{3}(\beta_{0}-\beta')^{2}v_{1}^{\beta_{0}}(E) \geqslant \frac{2\alpha_{0}-\beta_{0}^{2}}{6}(v_{1}^{\beta_{0}}(E)+(\beta_{0}-\beta')v_{0}^{\beta_{0}}(E)).
\end{align} After simply the inequality (\ref{computation3}), we can finally get $Q^{\beta_{0},\alpha_{0}}(E) \geqslant 0$. Conversely, it is obvious that if the inequality (\ref{BGinequality2}) holds for any $(\beta,\alpha) \in U$, then Conjecture \ref{conj} holds.
\end{pf}

\subsection{Noetherian property of the bounded heart}

Besides discussing Bogomolov-Gieseker type inequality, it is also improtant to prove that $\mathrm{Coh}_{0}^{\beta,\alpha}(X)$ is a Noetherian abelian category when $\beta,\alpha$ are rational numbers and $\dim Y=3$. And the proof is similar to the proof in \cite{bayer2014bridgeland},\cite{toda2013bogomolov} and \cite{toda2013stability}.  

\begin{prop}
	For any $(\beta,\alpha) \in U \cap \mathbb{Q}^2$, the abelian category $\mathrm{Coh}^{\beta,\alpha}_{0}(X)$ is Noetherian.
\end{prop}

\begin{pf}
	See Proposition \ref{Noetherian2}.
\end{pf}

\begin{prop}\label{finitelength}
	For any $(\beta,\alpha) \in U \cap \mathbb{Q}^2$, the subcategory $\mathcal{I}^{\beta,\alpha}_{0}(X)=\{E \in \mathrm{Coh}^{\beta,\alpha}_{0}(X) \mid v^{\beta}_{2}(E)=(\alpha-\displaystyle\frac{\beta^{2}}{2})v^{\beta}_{0}(E)\}$ is an abelian category of finite length.
\end{prop}

\begin{pf}
	By proposition \ref{Noetherian2}, we just need to prove it is Artinian. Suppose we have an infinite chain in $\mathcal{I}^{\beta,\alpha}_{0}(X)$:
	\begin{align*}
		E_{0} \supseteq E_{1} \supseteq E_{2} \supseteq \cdots \supseteq E_{n} \supseteq E_{n+1} \supseteq \cdots.
	\end{align*}
	
	Now, we take the cohomology functor $\mathcal{H}^{l}_{\beta}$ with respect to the heart $\mathrm{Coh}_{0}^{\beta}(X)$ on the following short exact sequences in $\mathcal{I}^{\beta,\alpha}_{0}(X) \subseteq \mathrm{Coh}_{0}^{\alpha,\beta}(X)$:
	\begin{align*}
		0 \to E_{n+1} \to E_{n} \to E_{n}/E_{n+1} \to 0.
	\end{align*}Then we have an infinite chain in $\mathrm{Coh}_{0}^{\beta}(X)$:
	\begin{align*}
		\mathcal{H}^{-1}_{\beta}(E_{0}) \supseteq \mathcal{H}^{-1}_{\beta}(E_{1}) \supseteq \mathcal{H}^{-1}_{\beta}(E_{2}) \supseteq \cdots \supseteq \mathcal{H}^{-1}_{\beta}(E_{n}) \supseteq \mathcal{H}^{-1}_{\beta}(E_{n+1}) \supseteq \cdots
	\end{align*} and exact sequence in $\mathrm{Coh}_{0}^{\beta}(X)$:
	\begin{align*}
		0 \to \mathcal{H}^{-1}_{\beta}(E_{n+1}) \to \mathcal{H}^{-1}_{\beta}(E_{n}) \to \mathcal{H}^{-1}_{\beta}(E_{n}/E_{n+1}) \to \mathcal{H}^{0}_{\beta}(E_{n+1}) \to \mathcal{H}^{0}_{\beta}(E_{n}) \to \mathcal{H}^{0}_{\beta}(E_{n}/E_{n+1}) \to 0.
	\end{align*}
	
	Note that if $E \in \mathcal{I}^{\beta,\alpha}_{0}(X)$, then $\mathcal{H}^{-1}_{\beta}(E)=0$ or it is $\nu^{\beta,\alpha}$-semistable object with tilt-slope $\beta$ and $\mathcal{H}^{0}_{\beta}(E)$ is a sheaf with dimension zero. But we know the quasi-abelian category 
	\begin{align*}
		\{E \in \mathrm{Coh}^{\beta}_{0}(X) \mid E \mbox{ is}\mbox{ } \nu^{\beta,\alpha}\mbox{-semistable with slope } \beta\}
	\end{align*} is of finite length. And the cokernel of $\mathcal{H}^{-1}_{\beta}(E_{n+1}) \to \mathcal{H}^{-1}_{\beta}(E_{n})$ is a subobject of $\mathcal{H}^{-1}_{\beta}(E_{n}/E_{n+1})$, which is still a $\nu^{\beta,\alpha}$-semistable object with tilt-slope $\beta$. Thus, there exists $N_{1} \in \mathbb{N}$, such that when $n>N_{1}$, we have $\mathcal{H}^{-1}_{\beta}(E_{n+1}) \cong \mathcal{H}^{-1}_{\beta}(E_{n})$. In this case, we have an exact sequence in $\mathrm{Coh}_{0}^{\beta}(X)$:
	\begin{align*}
		0 \to \mathcal{H}^{-1}_{\beta}(E_{n}/E_{n+1}) \to \mathcal{H}^{0}_{\beta}(E_{n+1}) \to \mathcal{H}^{0}_{\beta}(E_{n}) \to \mathcal{H}^{0}_{\beta}(E_{n}/E_{n+1}) \to 0.
	\end{align*} Now notice that $v_{1}^{\beta}(\mathcal{H}^{-1}_{\beta}(E_{n}/E_{n+1}))=v_{1}^{\beta}(\mathcal{H}^{0}_{\beta}(E_{n+1}))+v_{1}^{\beta}(\mathcal{H}^{0}_{\beta}(E_{n}/E_{n+1}))-v_{1}^{\beta}(\mathcal{H}^{0}_{\beta}(E_{n}))=0$, thus $\mathcal{H}^{-1}_{\beta}(E_{n}/E_{n+1})=0$. Therefore, we have an exact sequence in $\mathrm{Coh}_{0}^{\beta}(X)$:
	\begin{align*}
		0 \to \mathcal{H}^{0}_{\beta}(E_{n+1}) \to \mathcal{H}^{0}_{\beta}(E_{n}) \to \mathcal{H}^{0}_{\beta}(E_{n}/E_{n+1}) \to 0.
	\end{align*}It is also an exact sequence in $\mathrm{Coh}_{0}(X)$.
	
	Thus, we have an infinite chain of zero dimension sheaf in $\mathrm{Coh}_{0}(X)$:
	\begin{align*}
		\mathcal{H}^{0}_{\beta}(E_{0}) \supseteq \mathcal{H}^{0}_{\beta}(E_{1}) \supseteq \mathcal{H}^{0}_{\beta}(E_{2}) \supseteq \cdots \supseteq \mathcal{H}^{0}_{\beta}(E_{n}) \supseteq \mathcal{H}^{0}_{\beta}(E_{n+1}) \supseteq \cdots.
	\end{align*}And it must be stable. In other words, there exists $N_{2}>N_{1} \in \mathbb{N}$, such that when $n>N_{2}$, $\mathcal{H}^{0}_{\beta}(E_{n}/E_{n+1})=0$. Thus, when $n>N_{2}$, $E_{n}/E_{n+1}=0$ and $E_{n}=E_{n+1}$.
\end{pf}

Note that the simple object in $\mathcal{I}^{\beta,\alpha}_{0}(X)$ must be zero dimensional sheaf or $F[1]$, where $F[1]$ is a $\nu^{\beta,\alpha}$-stable sheaf with slope $\beta$. In fact, we don't need to prove the Bogomolov-Gieseker type inequality for all $F$ which is $\nu^{\beta,\alpha}$-semistable object with tilt-slope $\beta$. We only need to prove it for those $F[1]$ which is a simple object in $\mathcal{I}^{\beta,\alpha}_{0}(X)$.

We have a useful lemma to prove the Bogomolov-Gieseker type inequality:

\begin{lem}\label{BGlem}
	Suppose $\mathcal{D}$ is a triangulated category and $\mathcal{A}$ is a bounded heart of $\mathcal{D}$. Let $Z \colon K(\mathcal{D}) \to \mathbb{C}$ be a group homomorphism such that for any $E \in \mathcal{A}$, $\Im Z(E) \geqslant 0$. Let $\mathcal{I}=\{E \in \mathcal{A} \mid \Im Z(E)=0\}$ be the full abelian subcategory of $\mathcal{D}$ which is of finite length, and $F[1]$ be a simple object of $\mathcal{I}$.
	
	Assume that there exists another bounded heart $\mathcal{B}$ of $\mathcal{D}$ with the following properties:
	
	(1) $\mathcal{B} \subseteq \langle \mathcal{A},\mathcal{A}[1]\rangle$,
	
	(2) there exists $\phi_{0} \in (0,1)$ such that:
	\begin{align*}
		Z(\mathcal{B}) \subseteq \{r \mathrm{exp}(\sqrt{-1}\pi\phi) \mid r \geqslant 0, \phi_{0} \leqslant \phi \leqslant \phi_{0}+1\}.
	\end{align*}
	
	(3) $F[2] \notin \mathcal{B}$. 
	
	Then $\Re Z(F) \geqslant 0$.
\end{lem}

\begin{pf}
	Because $\mathcal{B} \subseteq \langle \mathcal{A},\mathcal{A}[1]\rangle$, then we have a torsion pair $(\mathcal{T},\mathcal{F})$ of $\mathcal{A}$ by Proposition \ref{tiltingheart}:
	\begin{align*}
		& \mathcal{T}=\mathcal{A} \bigcap \mathcal{B}; \\
		& \mathcal{F}=\mathcal{A} \bigcap \mathcal{B}[-1].
	\end{align*} such that $\mathcal{B}$ is the tilting heart of $\mathcal{A}$ with respect to $(\mathcal{T},\mathcal{F})$.
	
	Now, if $F[2] \notin \mathcal{B}$, then $F[1] \notin \mathcal{F}$. Because $F[1]$ is a simple object of $\mathcal{I}$, thus $F[1] \in \mathcal{T}$ and then $F[1] \in \mathcal{B}$. By the second condition, $\Re Z(F) \geqslant 0$. 
\end{pf}

	By Taking $\mathcal{D}=D_{0}^{b}(X)$, $\mathcal{A}=\mathrm{Coh}^{\beta,\alpha}_{0}(X)$ and $Z=Z^{\beta,\alpha,a_{0}}$ and Lemma \ref{BGlem}, we have the following lemma:

\begin{lem}\label{BGlem1}
	Let $(\beta,\alpha) \in U \cap \mathbb{Q}^{2}$. Suppose that there exists another bounded heart $\mathcal{B}$ of $D^{b}_{0}(X)$ with the following properties:
	
	(1) $\mathcal{B} \subseteq \langle \mathrm{Coh}^{\beta,\alpha}_{0}(X),\mathrm{Coh}^{\beta,\alpha}_{0}(X)[1]\rangle$,
	
	(2) there exists $\phi_{0} \in (0,1)$ such that:
	\begin{align*}
		Z^{\beta,\alpha,a_{0}}(\mathcal{B}) \subseteq \{r \mathrm{exp}(\sqrt{-1}\pi\phi) \colon r \geqslant 0, \phi_{0} \leqslant \phi \leqslant \phi_{0}+1\},
	\end{align*}where $a_{0}=\displaystyle\frac{2\alpha-\beta^{2}}{6}$.
	
	(3) for any $\nu^{\beta,\alpha}$-semistable object $F \in \mathrm{Coh}^{\beta}_{0}(X)$ with tilt-slope $\beta$ and $F[1]$ is a simple object in $\mathcal{I}^{\beta,\alpha}_{0}(X)$, we have $F[2] \notin \mathcal{B}$.
	
	Then $(\beta,\alpha)$ satisfies the Bogomolov-Gieseker type ineqaulity.
\end{lem}

\section{Geometric stability conditions on the canonical bundle $\mathrm{Tot}_{\mathbb{P}^{3}}(\omega_{\mathbb{P}^{3}})$ of $\mathbb{P}^{3}$}\label{section4}

Now we will put our attention on the main case when $X=\mathrm{Tot}_{\mathbb{P}^{3}}(\omega_{\mathbb{P}^{3}})$ and choose the polarization $H$ to be the hyperplane class of $\mathbb{P}^{3}$. First, we need to check Assumption \ref{crucialassumption} holds in this case.

\begin{prop}
	A sheaf $E \in \mathrm{Coh}_{0}(X)$ is a slope stable sheaf if and only if it is the pushforward $E=i_{*}E_{0}$ of some slope stable sheaf $E_{0} \in \mathrm{Coh}(\mathbb{P}^{3})$.
\end{prop}

\begin{pf}
	We claim that: if $E \in \mathrm{Coh}_{0}(X)$ with $\mathrm{Hom}_{\mathrm{Coh}_{0}(X)}(E,E) \cong \mathbb{C}$ as $\mathbb{C}$-algebra, then $E=i_{*}E_{0}$ of some slope stable sheaf $E_{0} \in \mathrm{Coh}(\mathbb{P}^{3})$. Let $Z$ be the scheme-theoretic support of $E$. By definition, its global sections act faithfully on $E$, so $H^{0}(X,\mathcal{O}_{Z}) \cong \mathbb{C}$. Hence $Z$ must be contained scheme-theoretically in the fiber of the origin under the contraction 
	\begin{align*}
		X \twoheadrightarrow \mathrm{Spec}H^{0}(X,\mathcal{O}_{X}) \cong \mathbb{C}^4/\mathbb{Z}_{4},
	\end{align*}as otherwise the image of $H^{0}(X,\mathcal{O}_{X}) \to H^{0}(X,\mathcal{O}_{Z})$ would be non-trivial. But the scheme-theoretic fiber of the origin is exactly $\mathbb{P}^{3}$, and so $E$ is the pushforward $i_{*}E_{0}$ of some sheaf $E_{0} \in \mathrm{Coh}(\mathbb{P}^{3})$.
	
	Finally, we can check $E_{0}$ is a slope stable sheaf in $\mathbb{P}^{3}$ and we finish the proof.
\end{pf}

Now, we consider the exceptional collection
\begin{align*}
	(\mathcal{O}(-1),\mathcal{T}(-2),\mathcal{O},\mathcal{O}(1))
\end{align*} of $\mathbb{P}^{3}$, where $\mathcal{T}$ is the tangent sheaf of $\mathbb{P}^{3}$. Bridgeland proved that there is a bounded heart $\mathcal{B}$ of $D_{0}^{b}(X)$ which is of finite length and with simple objects $S_{0}=i_{*}\mathcal{O}(1)$, $S_{1}=i_{*}\mathcal{O}[1]$, $S_{2}=i_{*}\mathcal{T}(-2)[2]$, $S_{3}=i_{*}\mathcal{O}(-1)[3]$ in \cite[Lemma 4.4]{bridgeland2005t}.

\begin{prop}
	These four sheaves $i_{*}\mathcal{O}(1),i_{*}\mathcal{O},i_{*}\mathcal{T}(-2),i_{*}\mathcal{O}(-1) \in \mathrm{Coh}_{0}(X)$ are all slope stable sheaves. 
\end{prop}

\begin{pf}
	In fact, we know $\mathcal{O}(1),\mathcal{O},\mathcal{T}(-2),\mathcal{O}(-1)$ are all slope stable sheaves in $\mathbb{P}^{3}$ and by Proposition \ref{ipushforwardofslopestablesheaf} we get the conclusion.
\end{pf}

\begin{prop}\label{tiltstablebobject}
	If $(\beta,\alpha) \in U$ satisfies $-\displaystyle\frac{1}{2}\leqslant \beta \leqslant 0$ and $\sqrt{2\alpha-\beta^{2}}<\displaystyle\frac{1}{2}$, then:
	
	(1) The object $i_{*}\mathcal{O}(1) \in \mathrm{Coh}^{\beta}_{0}(X)$ is a $\nu^{\beta,\alpha}$-stable object with tilt-slope $\displaystyle\frac{1-2\alpha}{2-2\beta}>\beta$ and $S_{0}=i_{*}\mathcal{O}(1) \in \mathrm{Coh}^{\beta,\alpha}_{0}(X)$;
	
	(2) When $-\displaystyle\frac{1}{2} \leqslant \beta<0$, the object $i_{*}\mathcal{O} \in \mathrm{Coh}^{\beta}_{0}(X)$ is a $\nu^{\beta,\alpha}$-stable object with tilt-slope $\displaystyle\frac{\alpha}{\beta}$. When $\beta=0$, the object $i_{*}\mathcal{O}[1] \in \mathrm{Coh}^{0}_{0}(X)$ is a $\nu^{0,\alpha}$-stable object with tilt-slope $+\infty$. In both cases, we have $S_{1}[-1]=i_{*}\mathcal{O} \in \mathrm{Coh}^{\beta,\alpha}_{0}(X)$ or $S_{1}=i_{*}\mathcal{O}[1] \in \mathrm{Coh}^{\beta,\alpha}_{0}(X)$;
	
	(3) The object $i_{*}\mathcal{T}(-2)[1] \in \mathrm{Coh}^{\beta}_{0}(X)$ is a $\nu^{\beta,\alpha}$-stable object with tilt-slope $\displaystyle\frac{3\alpha}{2+3\beta}>\beta$ and $S_{2}[-1]=i_{*}\mathcal{T}(-2)[1] \in \mathrm{Coh}^{\beta,\alpha}_{0}(X)$;
	
	(4) The object $i_{*}\mathcal{O}(-1)[1] \in \mathrm{Coh}^{\beta}_{0}(X)$ is a $\nu^{\beta,\alpha}$-stable object with tilt-slope $\displaystyle\frac{-1+2\alpha}{2+2\beta}<\beta$ and $S_{3}[-1]=i_{*}\mathcal{O}(-1)[2] \in \mathrm{Coh}^{\beta,\alpha}_{0}(X)$.
\end{prop}

\begin{pf}
	First, we know that for any line bundle $\mathcal{O}(d)$ on $\mathbb{P}^{3}$, it is $\nu^{\beta,\alpha}$-stable, for any $(\beta,\alpha) \in U$.(see also \cite[Corollary 3.11]{bayer2016space}) And $\mathcal{T}(-2)[1]$ is $\nu^{0,\alpha}$-stable obejct for any $\alpha>0$.(see also \cite[Corollary 3.11]{macri2014generalized}). 
	If $\mathcal{T}(-2)[1]$ is strictly $\nu^{\beta_{0},\alpha_{0}}$-semistable for some $(\beta_{0},\alpha_{0}) \in U$ with $-\displaystyle\frac{1}{2} \leqslant \beta_{0}<0$, then there must be a wall $l$ passes through $(\beta_{0},\alpha_{0})$ and $\Pi(\mathcal{T}(-2)[1])=\left(-\displaystyle\frac{2}{3},0\right)$ by Proposition \ref{wallandchamber}. Note that the slope of this line must bigger than $0$ and it must intersect with the ray $\{(0,\alpha) \in U\mid \alpha>0\}$. It is a contradiction. Thus their pushforward by $i_{*}$ is also $\nu^{\beta,\alpha}$-stable by Proposition \ref{ipushforwardoftiltstableobject}.
	
	Now, we only need to verify that when $-\displaystyle\frac{1}{2}\leqslant \beta \leqslant 0$ and $0<2\alpha-\beta^{2}<\displaystyle\frac{1}{4}$, we have
	\begin{align*}
		& 1-2\alpha>2\beta-2\beta^{2}; \\ 
		& 3\alpha>2\beta+3\beta^{2}; \\
		& -1+2\alpha<2\beta+2\beta^{2}.
	\end{align*} All of them are obvious.
\end{pf}

\begin{prop}\label{finalpartproposition}
	If $(\beta,\alpha) \in U$ satisfies $-\displaystyle\frac{1}{2}\leqslant \beta \leqslant 0$ and $\sqrt{2\alpha-\beta^{2}}<\displaystyle\frac{1}{2}$ and $\beta,\alpha$ are rational numbers, then $(\beta,\alpha)$ satisfies Bogomolov-Gieseker type inequality. 
\end{prop}

\begin{pf}
	Denote $\displaystyle\frac{2\alpha-\beta^{2}}{6}$ by $a_{0}$. By Proposition \ref{tiltstablebobject}, we have 
	\begin{align*}
		\mathcal{B} \subseteq \langle \mathrm{Coh}^{\beta,\alpha}_{0}(X),\mathrm{Coh}^{\beta,\alpha}_{0}(X)[1]\rangle
	\end{align*}
	
	After some calculation, we have: 
	\begin{align*}
		& z_{0}=Z^{\beta,\alpha,a_{0}}(S_{0})=\frac{1}{2}(2\beta^{2}-2\beta+1-2\alpha)(\frac{1}{3}(\beta-1)+\sqrt{-1});\\
		&
		z_{1}=Z^{\beta,\alpha,a_{0}}(S_{1})=(\alpha-\beta^{2})(\frac{1}{3}\beta+\sqrt{-1});\\
		&
		z_{2}=Z^{\beta,\alpha,a_{0}}(S_{2})=\frac{1}{2}\beta^{3}+\beta^{2}-\frac{2}{3}-a_{0}(3\beta+2)+\sqrt{-1}(3\beta^{2}+2\beta-3\alpha);\\
		&
		z_{3}=Z^{\beta,\alpha,a_{0}}(S_{3})=-\frac{1}{2}(2\beta^{2}+2\beta+1-2\alpha)(\frac{1}{3}(\beta+1)+\sqrt{-1}).
	\end{align*}
	
	Note that $z_{0}$ stay in the second quadrant with slope $-\displaystyle\frac{1}{3}(\beta-1)$ and $z_{2},z_{3}$ stay in the third quadrant because $\displaystyle\frac{1}{2}\beta^{3}+\beta^{2}-\frac{2}{3}<0$ for $-\displaystyle\frac{1}{2}\leqslant \beta \leqslant 0$. 
	
	When $\alpha \geqslant \beta^{2}$, we have:
	\begin{align*}
		Z^{\beta,\alpha,a_{0}}(\mathcal{B}) \subseteq \{r \mathrm{exp}(\sqrt{-1}\pi\phi) \mid r \geqslant 0, \frac{1}{2}\leqslant \phi \leqslant \frac{3}{2}\}.
	\end{align*}And when $\displaystyle\frac{1}{2}\beta^{2}<\alpha<\beta^{2}$, $z_{1}$ is in the fourth quadrant but the slope of $z_{0}$ is bigger than the slope of $-z_{1}$:
	\begin{align*}
		-\frac{1}{3}(\beta-1)<-\frac{1}{3}\beta.
	\end{align*}If we denote the phase of $-z_{1}$ by $\phi_{0} \in (0,1)$, thus:
	\begin{align*}
		Z^{\beta,\alpha,a_{0}}(\mathcal{B}) \subseteq \{r \mathrm{exp}(\sqrt{-1}\pi\phi) \mid r \geqslant 0, \phi_{0}\leqslant \phi \leqslant \phi_{0}+1\}.
	\end{align*}
	
	Now, we just need to prove that for any $\nu^{\beta,\alpha}$-semistable object $F \in \mathrm{Coh}^{\beta}_{0}(X)$ with tilt-slope $\beta$ and $F[1]$ is a simple object of $\mathcal{I}^{\beta,\alpha}_{0}(X)$, we always have $F[2] \notin \mathcal{B}$. In fact, we claim that $\mathrm{Hom}_{D^{b}_{0}(X)}(F[2],S_{i})=0$ for any $i \in \{0,1,2,3\}$.
	
	Because $F,S_{0} \in \mathrm{Coh}^{\beta}_{0}(X)$, then $\mathrm{Hom}_{D^{b}_{0}(X)}(F[2],S_{0})=0$. For $S_{1}$, we have $S_{1} \in \mathrm{Coh}^{\beta}_{0}(X)$ or $S_{1}[-1] \in \mathrm{Coh}^{\beta}_{0}(X)$, thus $\mathrm{Hom}_{D^{b}_{0}(X)}(F[2],S_{1})=0$. For $S_{2}$, we have $S_{2}[-1] \in \mathrm{Coh}^{\beta}_{0}(X)$ and $\mathrm{Hom}_{D^{b}_{0}(X)}(F[2],S_{2})=\mathrm{Hom}_{D^{b}_{0}(X)}(F[1],S_{2}[-1])=0$. Finally, for $S_{3}$, we have $\mathrm{Hom}_{D^{b}_{0}(X)}(F[2],S_{3})=\mathrm{Hom}_{D^{b}_{0}(X)}(F,i_{*}\mathcal{O}(-1)[1])=0$ because $i_{*}\mathcal{O}(-1)[1] \in \mathrm{Coh}^{\beta}_{0}(X)$ is $\nu^{\beta,\alpha}$-stable with tilt-slope $<\beta$. 
	
	By Lemma \ref{BGlem1}, we get our conclusion
\end{pf}

\begin{thm}\label{thm1}
	Let $X=\mathrm{Tot}(\omega_{\mathbb{P}^{3}})$. Then for any $(\beta,\alpha) \in U$ and $a>\displaystyle\frac{2\alpha-\beta^{2}}{6}$, the central charge $Z^{\beta,\alpha,a}$ is a stability function on the bounded heart $\mathrm{Coh}_{0}^{\beta}(X)$.
\end{thm}

\begin{pf}
	 By Proposition \ref{finalpartproposition}, \ref{reduce1}, \ref{reduce2} and \ref{reduce3}, we can get the conclusion.
\end{pf}

\begin{rmk}
	(1) In fact, by using the method of deformation in \cite[Section 8]{bayer2016space}, we could prove that there is a continous embedding:
	\begin{align*}
		\{(\beta,\alpha,a) \in \mathbb{R}^{3} \mid 2\alpha>\beta^{2},6a>2\alpha-\beta^{2}\}\to \mathrm{Stab}_{H}(D^{b}_{0}(X)), (\beta,\alpha,a) \mapsto \sigma^{\beta,\alpha,a}=(Z^{\beta,\alpha,a},\mathrm{Coh}^{\beta,\alpha}_{0}(X)).
	\end{align*}Moreover, we could construct a simply connected open subset of $\mathrm{Stab}_{H}(D^{b}_{0}(X))$. But we will not use this result, so we don't write detail here.
	
	(2) Our method also works for other locally free sheaf on $\mathbb{P}^{3}$ with two assumptions. Explicitly, if $X$ is the total space of a locally free sheaf $\mathcal{E}_{0}$ on $\mathbb{P}^{3}$, which satisfies:
	\begin{itemize}
		\item (Assumption \ref{crucialassumption})A sheaf $E \in \mathrm{Coh}_{0}(X)$ is a slope stable sheaf if and only if it is the pushforward $E=i_{*}E_{0}$ of some slope stable sheaf $E_{0} \in \mathrm{Coh}(\mathbb{P}^{3})$.
		\item The extension-closed subcategory $\langle i_{*}\mathcal{O}(1),i_{*}\mathcal{O}[1],i_{*}\mathcal{T}(-2)[2],i_{*}\mathcal{O}(-1)[3]\rangle$ is a bounded heart of $D_{0}^{b}(X)$.
	\end{itemize}
	Then there will be a continous embedding:
	\begin{align*}
		\{(\beta,\alpha,a) \in \mathbb{R}^{3} \mid 2\alpha>\beta^{2},6a>2\alpha-\beta^{2}\}\to \mathrm{Stab}_{H}(D^{b}_{0}(X)), (\beta,\alpha,a) \mapsto \sigma^{\beta,\alpha,a}=(Z^{\beta,\alpha,a},\mathrm{Coh}^{\beta,\alpha}_{0}(X)).
	\end{align*}
\end{rmk}

\section{Part of the boundary of the subset of geometric stability conditions on $D_{0}^{b}(\mathrm{Tot}_{\mathbb{P}^{3}}(\omega_{\mathbb{P}^{3}}))$}\label{section5}

In this section, we will construct some stability condition on the boundary of the subset $\mathrm{Stab}_{H}^{\mathrm{geo}}(D_{0}^{b}(X))$ of geometric stability conditions, where $X=\mathrm{Tot}_{\mathbb{P}^{3}}(\omega_{\mathbb{P}^{3}})$. Firstly, we need a lemma to construct new tilting heart:

\begin{lem}\label{boundaryheart}
	Suppose $\mathcal{D}$ is a $\mathbb{C}$-triangulated category of finite type and the pair $\sigma=(Z,\mathcal{A})$ is a weak stability condition on $\mathcal{D}$. Let $\mu$ be the corresponding slope function. Let $E$ be a $\sigma$-stable object in $\mathcal{A}$ with $\mu(E)=\beta<+\infty$ and we denote
	\begin{align*}
		&\mathcal{T}^{E}=\langle T\in \mathcal{A} \mid T \mbox{ is } \sigma\mbox{-semistable with }\mu(T)>\beta, E \rangle, \\
		&\mathcal{F}^{E}=\langle F\in \mathcal{A} \mid F \mbox{ is } \sigma\mbox{-semistable with }\mu(F) \leqslant \beta \mbox{ with }\mathrm{Hom}_{\mathcal{A}}(E,F)=0 \rangle, \\
		& \mathcal{T}^{\beta}=\langle T\in \mathcal{A} \mid T \mbox{ is } \sigma\mbox{-semistable with }\mu(T)>\beta \rangle, \\
		& \mathcal{F}^{\beta}=\langle F\in \mathcal{A} \mid F \mbox{ is } \sigma\mbox{-semistable with }\mu(F) \leqslant \beta \rangle.
	\end{align*} Moreover, we assume that $\mathrm{Hom}_{\mathcal{D}}(E,E[1])=0$ and 
	$\mathrm{Hom}_{\mathcal{D}}(C,E[1])=0$ for any object $C \in \mathcal{A}$ with $Z(C)=0$. Then we have:
	
	(1) The pair $(\mathcal{T}^{E},\mathcal{F}^{E})$ is a torsion pair of $\mathcal{A}$.
	
	(2) Denote the tilting heart of $\mathcal{A}$ with respect to $(\mathcal{T}^{\beta},\mathcal{F}^{\beta})$ by $\mathcal{A}^{\beta}$ and denote the tilting heart of $\mathcal{A}$ with respect to $(\mathcal{T}^{E},\mathcal{F}^{E})$ by $\mathcal{A}^{E}$, then $\mathcal{A}^{E}[1]$ is a tilting heart of $\mathcal{A}^{\beta}$ with respect to a torsion pair $(\mathcal{T},\mathcal{F})$, where $\mathcal{T}=\{E[1]^{\oplus n} \mid n \in \mathbb{N}^{+}\}$ and $\mathcal{F}$ is the right orthogonal subcategory of $\mathcal{T}$ in $\mathcal{A}^{\beta}$.
	
	(3) If $\mathcal{A}^{\beta}$ is a Noetherian abelian category, then $\mathcal{A}^{E}$ is also Noetherian.
\end{lem}

\begin{pf}
	(1) It is obvious that $\mathrm{Hom}_{\mathcal{A}}(\mathcal{T}^{E},\mathcal{F}^{E})=0$ from the definition. Suppose $E_{0} \in \mathcal{A}$, then we have a short exact sequence in $\mathcal{A}$:
	\begin{align*}
		0 \to T_{0} \to E_{0} \to F_{0} \to 0,
	\end{align*}where $T_{0} \in \mathcal{T}^{\beta} \subseteq \mathcal{T}^{E}$ and $F_{0} \in \mathcal{F}^{\beta}$ by Harder-Narasimhan filtration of $\sigma$. Let $V=\mathrm{Hom}_{\mathcal{A}}(E,F_{0})$, it is a finite-dimensional $\mathbb{C}$-vector space. 
	
	If $V=0$, then we can show that $F_{0} \in \mathcal{F}^{E}$ and the above short exact sequence is our desired decomposition of $E_{0}$. In fact, we will claim that $\mathcal{F}^{E}=\mathcal{F}^{\beta} \bigcap \{F \in \mathcal{A} \mid \mathrm{Hom}_{\mathcal{A}}(E,F)=0\}$. We have a short exact sequence in $\mathcal{A}$ because of the Harder-Narasimhan filtration: 
	\begin{align*}
		0 \to F_{0}^{\beta} \to F_{0} \to F_{0}^{<\beta} \to 0,
	\end{align*} where $F_{0}^{<\beta} \in \langle F\in \mathcal{A} \mid F \mbox{ is } \sigma\mbox{-semistable with }\mu(F)<\beta\rangle$ and it imples that $\mathrm{Hom}_{\mathcal{A}}(E,F_{0}^{<\beta})=0$. Therefore, $\mathrm{Hom}_{\mathcal{A}}(E,F_{0}^{\beta}) \cong \mathrm{Hom}_{\mathcal{A}}(E,F_{0})=0$, and our claim is ture.
	
	If $V \neq 0$, then we have a nonzero canonical morphism in $\mathcal{A}$:
	\begin{align*}
		f \colon V \otimes E \to F_{0}.
	\end{align*} Moreover, it induces a short exact sequence in $\mathcal{A}$:
	\begin{align*}
		0 \to \ker f \to V \otimes E \to \Im f \to 0.
	\end{align*}We assume that $\ker f \neq 0$. Note that $\Im f$ is a nonzero subobject of $F_{0}$, then $\mu(\Im f) \leqslant \mu_{\max}(F_{0}) \leqslant \beta=\mu (V \otimes E)$. But $V \otimes E$ is $\sigma$-semistable since $E$ is $\sigma$-stable, thus $\beta=\mu (V \otimes E)=\mu(\Im f)=\mu(\ker f)$. It implies that $\ker f$ is a $\sigma$-semistable object with slope $\beta$ and all of its Jordan-H\"older factors are $E$. By taking the functor $\mathrm{Hom}_{\mathcal{D}}(E,-)$ on the above short exact sequence, we have $\mathbb{C}$-vector space exact sequences:
	\begin{align*}
		0 \to \mathrm{Hom}_{\mathcal{A}}(E,\ker f) \to \mathrm{Hom}_{\mathcal{A}}(E,V \otimes E) \to \mathrm{Hom}_{\mathcal{A}}(E,\Im  f).
	\end{align*}
	From the definition of $f$ and $\mathrm{Hom}_{\mathcal{A}}(E,E) \cong \mathbb{C}$, we know that the composition of morphism:
	\begin{align*}
		\mathrm{Hom}_{\mathcal{A}}(E,V \otimes E) \to \mathrm{Hom}_{\mathcal{A}}(E,\Im  f) \to \mathrm{Hom}_{\mathcal{A}}(E,F_{0})=V
	\end{align*} is an isomorphism of $\mathbb{C}$-vector space. Meanwhile, we have $\mathrm{Hom}_{\mathcal{A}}(E,\Im  f) \subseteq \mathrm{Hom}_{\mathcal{A}}(E,F_{0})$. Therefore, $\mathrm{Hom}_{\mathcal{A}}(E,V \otimes E) \to \mathrm{Hom}_{\mathcal{A}}(E,\Im  f)$ is an isomorphism of $\mathbb{C}$-vector space and $\mathrm{Hom}_{\mathcal{A}}(E,\ker f)=0$, which is a contradiction. Thus, $\ker f=0$.
	
	Now we have a short exact sequence in $\mathcal{A}$:
	\begin{align*}
		0  \to V \otimes E \to F_{0} \to \mathrm{coker} f \to 0.
	\end{align*}
	Consider the following commutative diagram induced by a pullback square:
	\begin{align*}
		\xymatrix{
			0 \ar[r] & T_{0} \ar@{=}[d] \ar[r] & T \ar[r] \ar[d] & V \otimes E \ar[r] \ar[d] & 0\\
			0 \ar[r] & T_{0} \ar[r] & E_{0} \ar[r] & F_{0} \ar[r] & 0.
		}
	\end{align*} By the snake lemma, we have a short exact sequence in $\mathcal{A}$:
	\begin{align*}
		0 \to T \to E_{0} \to \mathrm{coker} f \to 0.
	\end{align*}Finally, we only need to prove that $\mathrm{coker} f \in  \mathcal{F}^{E}$.
	
	First, we have a long exact sequence of $\mathbb{C}$-vector spaces:
	\begin{align*}
		0 \to \mathrm{Hom}_{\mathcal{A}}(E,V \otimes E) \to \mathrm{Hom}_{\mathcal{A}}(E,F_{0}) \to \mathrm{Hom}_{\mathcal{A}}(E,\mathrm{coker} f) \to \mathrm{Hom}_{\mathcal{A}}(E,V \otimes E[1])=0.
	\end{align*}And the first morphism is an isomorphism, thus $\mathrm{Hom}_{\mathcal{A}}(E,\mathrm{coker} f)=0$.
	
	Meanwhile, we have a short exact sequence in $\mathcal{A}$ because of the Harder-Narasimhan filtration:
	\begin{align*}
		0 \to F_{0}^{\beta} \to F_{0} \to F_{0}^{<\beta} \to 0,
	\end{align*}where $F_{0}^{<\beta} \in \langle F\in \mathcal{A} \mid F \mbox{ is } \sigma\mbox{-semistable with }\mu(F)<\beta\rangle$. Because $V \neq 0$, then $F_{0}^{\beta} \neq 0$. Since $\mathrm{Hom}_{\mathcal{A}}(V \otimes E, F_{0}^{<\beta})=0$, we have a commutative diagram:
	\begin{align*}
		\xymatrix{
			0 \ar[r] & V \otimes E \ar[d] \ar[r] & F_{0} \ar[r] \ar@{=}[d] & \mathrm{coker} f \ar[r] \ar[d] & 0\\
			0 \ar[r] & F_{0}^{\beta} \ar[r] & F_{0} \ar[r] & F_{0}^{<\beta} \ar[r] & 0.
		}
	\end{align*} 
	
	Let $G$ be the cokernel of $V \otimes E \to F_{0}^{\beta}$, it is also isomorphic to the kernel of $\mathrm{coker} f \to F^{<\beta}_{0}$ by the snake lemma. Thus, $\mathrm{Hom}_{\mathcal{A}}(E,G) \subseteq \mathrm{Hom}_{\mathcal{A}}(E,\mathrm{coker} f)=0$. If $G \neq 0$, and we know $\beta=\mu(F_{0}^{\beta}) \leqslant \mu_{\min}(G)$, then there is a short exact sequence in $\mathcal{A}$ because of the Harder-Narasimhan filtration:
	\begin{align*}
		0 \to G^{>\beta} \to G \to G^{\beta} \to 0,
	\end{align*}where $G^{>\beta} \in \langle G\in \mathcal{A} \mid G \mbox{ is } \sigma\mbox{-semistable with }\mu(G)>\beta\rangle$. Note that $Z(F_{0}^{\beta})=Z(V \otimes E)+Z(G)=Z(V \otimes E)+Z(G^{\beta})+Z(G^{>\beta})$ and we conclude that $Z(G^{>\beta})=0$ since other terms are on the same line. Note that $\mathrm{Hom}_{\mathcal{A}}(G^{>\beta},V \otimes E[1])=0$. Then there will be an injection $G^{>\beta} \to F^{\beta}_{0}$, which is a contradiction unless $G^{>\beta}=0$. It means that $G=0$ or $G$ is a $\sigma$-semistable object in $\mathcal{A}$ with slope $\beta$ and $\mathrm{Hom}_{\mathcal{A}}(E,G)=0$. Thus $G \in \mathcal{F}^{E}$ and $\mathrm{coker} f \in \mathcal{F}^{E}$.
	
	(2) Because $\mathcal{T}^{\beta} \subseteq \mathcal{T}^{E}$, then $\mathcal{F}^{E} \subseteq \mathcal{F}^{\beta}$ and $\mathcal{F}^{E}[1] \subseteq \mathcal{F}^{\beta}[1] \subseteq \mathcal{A}^{\beta}$. Meanwhile, $\mathcal{T}^{E} \subseteq \mathcal{A} \subseteq \langle\mathcal{A}^{\beta},\mathcal{A}^{\beta}[-1]\rangle$. Thus, $\mathcal{A}^{E} \subseteq \langle\mathcal{A}^{\beta},\mathcal{A}^{\beta}[-1]\rangle$. By  Proposition \ref{tiltingheart}, we have $\mathcal{A}^{E}[1]$ is a tilting heart of $\mathcal{A}^{\beta}$ with respect to $(\mathcal{T},\mathcal{F})$, where $\mathcal{T}=\mathcal{A}^{E}[1] \bigcap \mathcal{A}^{\beta}$. It is obviously that $E[1] \in \mathcal{A}^{E}[1] \bigcap \mathcal{A}^{\beta}$. If $E' \in \mathcal{A}^{E}[1] \bigcap \mathcal{A}^{\beta}$, then $\mathcal{H}^{-1}(E')= \mathcal{H}^{0}(E'[-1]) \in \mathcal{F}^{\beta} \bigcap \mathcal{T}^{E}$ and $\mathcal{H}^{i}(E')=0$ for any $i \neq -1$, where $\mathcal{H}^{i}$ is the cohomology functor with respect to the bounded heart $\mathcal{A}$. It means that $E'[-1] \in \mathcal{F}^{\beta} \bigcap \mathcal{T}^{E}$.
	Let $V=\mathrm{Hom}_{\mathcal{A}}(E,E'[-1])$, then we have a short exact sequence in $\mathcal{A}$ as in (1):
	\begin{align*}
		0 \to V \otimes E \to E'[-1] \to C \to 0 
	\end{align*} with $C \in \mathcal{F}^{E}$. But $E'[-1] \in \mathcal{T}^{E}$, then $C=0$ and $E'[-1] \cong V \otimes E$. Therefore, We conclude that $\mathcal{T}=\{E[1]^{\oplus n} \mid n \in \mathbb{N}^{+}\}$.
	
	(3) We only need to prove that $\mathcal{A}^{E}[1]$ is Noetherian. Suppose that there is a chain of epimorphisms in $\mathcal{A}^{E}[1]$:
	\begin{align*}
		E_{0} \twoheadrightarrow E_{1} \twoheadrightarrow E_{2} \twoheadrightarrow \cdots
	\end{align*} and for any $i \in \mathbb{N}^{+}$, we denote $\ker(E_{0} \to E_{i}) \in \mathcal{A}^{E}[1]$ by $F_{i}$, thus we have a filtration in $\mathcal{A}^{E}[1]$:
	\begin{align*}
		F_{1} \subseteq F_{2} \subseteq \cdots \subseteq E_{0}.
	\end{align*} Moreover, we have $\ker(E_{i} \to E_{i+1}) \cong F_{i+1}/F_{i}$ for any $i \in \mathbb{N}^{+}$ in $\mathcal{A}^{E}[1]$ by the snake lemma. Take the cohomology functor $\mathcal{H}^{i}_{\mathcal{A}^{\beta}}$ with respect to $\mathcal{A}^{\beta}$ on the following short exact sequences in $\mathcal{A}^{E}[1]$:
	\begin{align*}
		& 0 \to F_{i} \to F_{i+1} \to F_{i+1}/F_{i} \to 0, \\
		& 0 \to F_{i+1}/F_{i} \to E_{i} \to E_{i+1} \to 0.
	\end{align*}Then we have:
	\begin{align*}
		\mathcal{H}^{-1}_{\mathcal{A}^{\beta}}(F_{1}) \subseteq \mathcal{H}^{-1}_{\mathcal{A}^{\beta}}(F_{2}) \subseteq \cdots \subseteq \mathcal{H}^{-1}_{\mathcal{A}^{\beta}}(F_{j}) \subseteq \mathcal{H}^{-1}_{\mathcal{A}^{\beta}}(F_{j+1})  \subseteq \cdots \subseteq \mathcal{H}^{-1}_{\mathcal{A}^{\beta}}(E_{0}). 
	\end{align*} and
	\begin{align*}
		\mathcal{H}^{0}_{\mathcal{A}^{\beta}}(E_{0}) \twoheadrightarrow \mathcal{H}^{0}_{\mathcal{A}^{\beta}}(E_{1}) \twoheadrightarrow \mathcal{H}^{0}_{\mathcal{A}^{\beta}}(E_{2}) \twoheadrightarrow \cdots \twoheadrightarrow \mathcal{H}^{0}_{\mathcal{A}^{\beta}}(E_{j}) \twoheadrightarrow \mathcal{H}^{0}_{\mathcal{A}^{\beta}}(E_{j+1}) \twoheadrightarrow \cdots.
	\end{align*}Because $\mathcal{A}^{\beta}$ is Noetherian, we could assume that for any $i \in \mathbb{N}^{+}$, we have $\mathcal{H}^{-1}_{\mathcal{A}^{\beta}}(F_{i}) \cong \mathcal{H}^{-1}_{\mathcal{A}^{\beta}}(F_{i+1}) \subseteq \mathcal{H}^{-1}_{\mathcal{A}^{\beta}}(E_{0})$ and $\mathcal{H}^{0}_{\mathcal{A}^{\beta}}(E_{i}) \cong \mathcal{H}^{0}_{\mathcal{A}^{\beta}}(E_{i+1})$. We denote the cokernel of $\mathcal{H}^{-1}_{\mathcal{A}^{\beta}}(F_{i}) \subseteq \mathcal{H}^{-1}_{\mathcal{A}^{\beta}}(E_{0})$ by $Q \in \mathcal{A}^{\beta}$.
	
	Then we also have an exact sequence in $\mathcal{A}^{\beta}$ as follows:
	\begin{align*}
		0 \to \mathcal{H}^{-1}_{\mathcal{A}^{\beta}}(F_{i+1}/F_{i}) \to \mathcal{H}^{0}_{\mathcal{A}^{\beta}}(F_{i}) \to \mathcal{H}^{0}_{\mathcal{A}^{\beta}}(F_{i+1}) \to \mathcal{H}^{0}_{\mathcal{A}^{\beta}}(F_{i+1}/F_{i}) \to 0.
	\end{align*} We can divide it into two short exact sequences in $\mathcal{A}^{\beta}$ as follows:
	\begin{align*}
		& 0 \to \mathcal{H}^{-1}_{\mathcal{A}^{\beta}}(F_{i+1}/F_{i}) \to \mathcal{H}^{0}_{\mathcal{A}^{\beta}}(F_{i}) \to I_{i} \to 0,\\
		& 0 \to I_{i} \to \mathcal{H}^{0}_{\mathcal{A}^{\beta}}(F_{i+1}) \to \mathcal{H}^{0}_{\mathcal{A}^{\beta}}(F_{i+1}/F_{i}) \to 0.
	\end{align*}Take the functor $\mathrm{Hom}_{\mathcal{A}^{\beta}}(E[1],-)$ on them, we will have $\mathrm{Hom}_{\mathcal{A}^{\beta}}(E[1],\mathcal{H}^{0}_{\mathcal{A}^{\beta}}(F_{i})) \subseteq \mathrm{Hom}_{\mathcal{A}^{\beta}}(E[1],I_{i})$, and $\mathrm{Hom}_{\mathcal{A}^{\beta}}(E[1],I_{i}) \subseteq \mathrm{Hom}_{\mathcal{A}^{\beta}}(E[1],\mathcal{H}^{0}_{\mathcal{A}^{\beta}}(F_{i+1}))$. If we denote $\dim_{\mathbb{C}} \mathrm{Hom}_{\mathcal{A}^{\beta}}(E[1],\mathcal{H}^{0}_{\mathcal{A}^{\beta}}(F_{i}))$ by $n_{i}$, then we have $n_{i} \leqslant n_{i+1}$.
	
	Meanwhile, consider the short exact sequence in $\mathcal{A}^{\beta}$:
	\begin{align*}
		0 \to Q \to \mathcal{H}_{\mathcal{A}^{\beta}}^{-1}(E_{i}) \to \mathcal{H}_{\mathcal{A}^{\beta}}^{0}(F_{i}) \to 0.
	\end{align*}
	Take the functor $\mathrm{Hom}_{\mathcal{A}^{\beta}}(E[1],-)$ on it, then we have $n_{i} \leqslant \mathrm{Hom}_{\mathcal{D}}(E,Q)$, it means $n_{i}$ is bounded.
	
	Therefore, for $i \gg 0$,  $\mathrm{Hom}_{\mathcal{A}^{\beta}}(E[1],I_{i}) =\mathrm{Hom}_{\mathcal{A}^{\beta}}(E[1],\mathcal{H}^{0}_{\mathcal{A}^{\beta}}(F_{i+1}))$ and $\mathrm{Hom}_{\mathcal{D}}(E[1],I_{i}[1])=0$ because of $I_{i} \in \mathcal{T}$ and $\mathrm{Hom}_{\mathcal{D}}(E,E[1])=0$. Thus, $\mathrm{Hom}_{\mathcal{A}^{\beta}}(E[1],\mathcal{H}^{0}_{\mathcal{A}^{\beta}}(F_{i+1}/F_{i}))=0$ and it means $\mathcal{H}^{0}_{\mathcal{A}^{\beta}}(F_{i+1}/F_{i})=0$ because $\mathcal{H}^{0}_{\mathcal{A}^{\beta}}(F_{i+1}/F_{i}) \in \mathcal{T}$.
	
	Now, $\mathcal{H}^{0}_{\mathcal{A}^{\beta}}(F_{i}) \to \mathcal{H}^{0}_{\mathcal{A}^{\beta}}(F_{i+1})$ is an epimorphism with $n_{i}=n_{i+1}$. Note that every epimorphism of $\mathrm{Hom}_{\mathcal{D}}(E^{\oplus n_{i}},E^{\oplus n_{i}})$ is also a monomorphism because of $\mathrm{Hom}_{\mathcal{D}}(E,E) \cong \mathbb{C}$. Therefore $\mathcal{H}^{-1}_{\mathcal{A}^{\beta}}(F_{i+1}/F_{i})=0$ and $F_{i+1}/F_{i}=0$ for $i \gg 0$. 
\end{pf}

Now, suppose that $\mathcal{E} \in \mathrm{Coh}(\mathbb{P}^{3})$ is an exceptional locally free sheaf of finite rank, then by the theorem \cite[Corollary]{zube199011} proved by Zube, we conclude that $\mathcal{E}$ is a slope stable sheaf. Therefore, if $\beta<\mu(\mathcal{E})$, then $\mathcal{E} \in \mathrm{Coh}_{0}^{\beta}(X)$; if $\beta \geqslant \mu(\mathcal{E})$, then $\mathcal{E}[1] \in \mathrm{Coh}_{0}^{\beta}(X)$.

For any $\beta \in \mathbb{R}$, let \begin{align*}
	& \alpha_{\mathcal{E}}^{\beta}=\beta^{2}-\displaystyle\frac{v_{1}(\mathcal{E})}{v_{0}(\mathcal{E})}\beta+\frac{v_{2}(\mathcal{E})}{v_{0}(\mathcal{E})};\\
	& \mu_{1}(\mathcal{E})=\mu(\mathcal{E})-\displaystyle\frac{\sqrt{\overline{\Delta}_{H}(\mathcal{E})}}{v_{0}(\mathcal{E})}; \\
	& \mu_{2}(\mathcal{E})=\mu(\mathcal{E})+\displaystyle\frac{\sqrt{\overline{\Delta}_{H}(\mathcal{E})}}{v_{0}(\mathcal{E})}.
\end{align*} Note that if $i_{*}\mathcal{E}$ is $\nu^{\beta,\alpha}$-stable for some $(\beta,\alpha) \in U$ with tilt-slope $\beta$, then $\alpha=\alpha_{\mathcal{E}}^{\beta}$ and $\alpha_{\mathcal{E}}^{\beta}>\displaystyle\frac{1}{2}\beta^{2}$. It implies that $\beta<\mu_{1}(\mathcal{E})$ or $\beta>\mu_{2}(\mathcal{E})$.

\begin{lem}
	(1) The object $i_{*}\mathcal{E}$ is a spherical object of $D_{0}^{b}(X)$.
	
	(2) For any $0 \neq C \in \mathrm{Coh}_{0}(X)$ with $\dim C=0$, $\mathrm{Hom}_{D^{b}_{0}(X)}(C,i_{*}\mathcal{E}[l])=0$ for $l=0,1,2$.
\end{lem}

\begin{pf}
	(1) By Proposition \ref{Homsetofinclusion}, we have:
	\begin{align*}
		\mathrm{Hom}_{D^{b}_{0}(X)}(i_{*}\mathcal{E},i_{*}\mathcal{E}[l]) \cong \mathrm{Hom}_{D^{b}(\mathbb{P}^{3})}(\mathcal{E},\mathcal{E}[l]) \oplus \mathrm{Hom}_{D^{b}(\mathbb{P}^{3})}(\omega_{\mathbb{P}^{3}}^{\lor} \otimes_{\mathbb{P}^{3}} \mathcal{E},\mathcal{E}[l-1])
	\end{align*}And by Serre duality, finally, we get:
	\begin{align*}
		\mathrm{Hom}_{D^{b}_{0}(X)}(i_{*}\mathcal{E},i_{*}\mathcal{E}[l]) \cong \mathrm{Hom}_{D^{b}(\mathbb{P}^{3})}( \mathcal{E},\mathcal{E}[l]) \oplus \mathrm{Hom}_{D^{b}(\mathbb{P}^{3})}( \mathcal{E},\mathcal{E}[4-l]) \cong
		\begin{cases}
			\mathbb{C} & l=0,4; \\
			0 & l \neq 0,4.
		\end{cases}
	\end{align*}
	
	(2) Note that we have:
	\begin{align*}						
		\mathrm{Hom}_{D^{b}_{0}(X)}(C,i_{*}\mathcal{E}[l]) \cong \mathrm{Hom}_{D^{b}(\mathbb{P}^{3})}(C,\mathcal{E}[l]) \oplus \mathrm{Hom}_{D^{b}(\mathbb{P}^{3})}( C,\mathcal{E}[l-1])=0
	\end{align*}by Proposition \ref{Homsetofinclusion}. 
\end{pf}

Now, let $\beta$ be a real number, if $\beta<\mu_{1}(\mathcal{E})$ and $i_{*}\mathcal{E} \in \mathrm{Coh}_{0}^{\beta}(X)$ is  $\nu^{\beta,\alpha_{\mathcal{E}}^{\beta}}$-stable. Then there is a torsion pair $(\mathrm{Coh}_{0}^{\beta,i_{*}\mathcal{E},>\beta}(X), \mathrm{Coh}_{0}^{\beta,i_{*}\mathcal{E},\leqslant \beta}(X))$ of $\mathrm{Coh}_{0}^{\beta}(X)$ by Lemma \ref{boundaryheart}, where
\begin{align*}
	&\mathrm{Coh}_{0}^{\beta,i_{*}\mathcal{E},>\beta}(X)=\langle T\in \mathrm{Coh}_{0}^{\beta}(X) \mid T \mbox{ is } \nu^{\beta,\alpha_{\mathcal{E}}^{\beta}}\mbox{-semistable with }\nu^{\beta,\alpha_{\mathcal{E}}^{\beta}}(T)>\beta, i_{*}\mathcal{E} \rangle, \\
	&\mathrm{Coh}_{0}^{\beta,i_{*}\mathcal{E},\leqslant \beta}(X)=\langle F\in \mathrm{Coh}_{0}^{\beta}(X) \mid F \mbox{ is } \nu^{\beta,\alpha_{\mathcal{E}}^{\beta}}\mbox{-semistable with }\nu^{\beta,\alpha_{\mathcal{E}}^{\beta}}(F) \leqslant \beta \mbox{ with }\mathrm{Hom}_{\mathrm{Coh}_{0}^{\beta}(X)}(i_{*}\mathcal{E},F)=0 \rangle.
\end{align*}

Similarly, if $\beta>\mu_{2}(\mathcal{E})$ and $i_{*}\mathcal{E}[1] \in \mathrm{Coh}_{0}^{\beta}(X)$ is  $\nu^{\beta,\alpha_{\mathcal{E}}^{\beta}}$-stable. Then there is a torsion pair $(\mathrm{Coh}_{0}^{\beta,i_{*}\mathcal{E}[1],>\beta}(X), \mathrm{Coh}_{0}^{\beta,i_{*}\mathcal{E}[1],\leqslant \beta}(X))$ of $\mathrm{Coh}_{0}^{\beta}(X)$ by Lemma \ref{boundaryheart}, where
\begin{align*}
	&\mathrm{Coh}_{0}^{\beta,i_{*}\mathcal{E}[1],>\beta}(X)=\langle T\in \mathrm{Coh}_{0}^{\beta}(X) \mid T \mbox{ is } \nu^{\beta,\alpha_{\mathcal{E}}^{\beta}}\mbox{-semistable with }\nu^{\beta,\alpha_{\mathcal{E}}^{\beta}}(T)>\beta, i_{*}\mathcal{E}[1] \rangle, \\
	&\mathrm{Coh}_{0}^{\beta,i_{*}\mathcal{E}[1],\leqslant \beta}(X)=\langle F\in \mathrm{Coh}_{0}^{\beta}(X) \mid F \mbox{ is } \nu^{\beta,\alpha_{\mathcal{E}}^{\beta}}\mbox{-semistable with }\nu^{\beta,\alpha_{\mathcal{E}}^{\beta}}(F) \leqslant \beta \mbox{ with }\mathrm{Hom}_{\mathrm{Coh}_{0}^{\beta}(X)}(i_{*}\mathcal{E}[1],F)=0 \rangle.
\end{align*}

Let $\mathrm{Coh}^{\beta,i_{*}\mathcal{E}}_{0}(X)=\langle\mathrm{Coh}_{0}^{\beta,i_{*}\mathcal{E},\leqslant \beta}(X)[1],\mathrm{Coh}_{0}^{\beta,i_{*}\mathcal{E},>\beta}(X)\rangle$ when $\beta<\mu_{1}(\mathcal{E})$ and
$\mathrm{Coh}^{\beta,i_{*}\mathcal{E}[1]}_{0}(X)=\langle\mathrm{Coh}_{0}^{\beta,i_{*}\mathcal{E}[1],\leqslant \beta}(X)[1],\mathrm{Coh}_{0}^{\beta,i_{*}\mathcal{E}[1],>\beta}(X)\rangle$ when $\beta>\mu_{2}(\mathcal{E})$. 
  
Our goal is to find some suitable real numbers $a$, such that $Z^{\beta,\alpha_{\mathcal{E}}^{\beta},a}$ is a stability function of $\mathrm{Coh}^{\beta,i_{*}\mathcal{E}}_{0}(X)$ or $\mathrm{Coh}^{\beta,i_{*}\mathcal{E}[1]}_{0}(X)$, where $Z^{\beta,\alpha_{\mathcal{E}}^{\beta},a}$ has the same form as (\ref{centralcharge}) in Section \ref{section3}:
\begin{align*}
 	Z^{\beta,\alpha_{\mathcal{E}}^{\beta},a}=-v^{\beta}_{3}+av^{\beta}_{1}+\sqrt{-1}\left(v^{\beta}_{2}-(\alpha_{\mathcal{E}}^{\beta}-\frac{\beta^{2}}{2})v^{\beta}_{0}\right) \colon \Lambda \to \mathbb{C}.
\end{align*}

\begin{prop}\label{imaginarypart3}
	Let $\beta$ be a real number and $\mathcal{E}$ be an exceptional locally free sheaf of finite rank on $\mathbb{P}^{3}$.
	
	(1) If $\beta<\mu_{1}(\mathcal{E})$ and $i_{*}\mathcal{E} \in \mathrm{Coh}_{0}^{\beta}(X)$ is  $\nu^{\beta,\alpha_{\mathcal{E}}^{\beta}}$-stable, then for any $a \in \mathbb{R}$ and $E \in \mathrm{Coh}^{\beta,i_{*}\mathcal{E}}_{0}(X)$, we have $\Im Z^{\beta,\alpha_{\mathcal{E}}^{\beta},a}(E) \geqslant 0$. Moreover, if $\Im Z^{\beta,\alpha_{\mathcal{E}}^{\beta},a}(E)=0$, then there is a filtration of $E$ in $\mathrm{Coh}^{\beta,i_{*}\mathcal{E}}_{0}(X)$:
	\begin{align*}
		0=E_{0} \subseteq E_{1}[1] \subseteq E_{2} \subseteq E_{3}=E,
	\end{align*}where $E_{1} \in \mathrm{Coh}_{0}^{\beta}(X)$ is $\nu^{\beta,\alpha_{\mathcal{E}}^{\beta}}$-semistable with tilt-slope $\beta$ and $\mathrm{Hom}_{\mathrm{Coh}_{0}^{\beta}(X)}(i_{*}\mathcal{E},E_{1})=0$ or $0$, $E_{2}/(E_{1}[1])$ is a zero dimensional sheaf or $0$ and $E_{3}/E_{2} \cong i_{*}\mathcal{E}^{\oplus n}$ for some $n \in \mathbb{N}$.
	
	(2) If $\beta>\mu_{2}(\mathcal{E})$ and $i_{*}\mathcal{E}[1] \in \mathrm{Coh}_{0}^{\beta}(X)$ is  $\nu^{\beta,\alpha_{\mathcal{E}}^{\beta}}$-stable, then for any $a \in \mathbb{R}$ and $E \in \mathrm{Coh}^{\beta,i_{*}\mathcal{E}[1]}_{0}(X)$, we have $\Im Z^{\beta,\alpha_{\mathcal{E}}^{\beta},a}(E) \geqslant 0$. Moreover, if $\Im Z^{\beta,\alpha_{\mathcal{E}}^{\beta},a}(E)=0$, then there is a filtration of $E$ in $\mathrm{Coh}^{\beta,i_{*}\mathcal{E}[1]}_{0}(X)$:
	\begin{align*}
		0=E_{0} \subseteq E_{1}[1] \subseteq E_{2} \subseteq E_{3}=E,
	\end{align*}where $E_{1} \in \mathrm{Coh}_{0}^{\beta}(X)$ is $\nu^{\beta,\alpha_{\mathcal{E}}^{\beta}}$-semistable with tilt-slope $\beta$ and $\mathrm{Hom}_{\mathrm{Coh}_{0}^{\beta}(X)}(i_{*}\mathcal{E}[1],E_{1})=0$ or $0$, $E_{2}/(E_{1}[1])$ is a zero dimensional sheaf or $0$ and $E_{3}/E_{2} \cong i_{*}\mathcal{E}[1]^{\oplus n}$ for some $n \in \mathbb{N}$.
\end{prop}

\begin{pf}
	(1) If $\beta<\mu_{1}(\mathcal{E})$, by Lemma \ref{boundaryheart}, for any $E \in \mathrm{Coh}^{\beta,i_{*}\mathcal{E}}_{0}(X)$, there is a short exact sequence in $\mathrm{Coh}^{\beta,i_{*}\mathcal{E}}_{0}(X)$:
	\begin{align*}
		0 \to E_{2} \to E \to i_{*}\mathcal{E}^{\oplus n} \to 0,
	\end{align*}where $E_{2} \in \mathrm{Coh}^{\beta,\alpha_{\mathcal{E}}^{\beta}}_{0}(X)$ with $\mathrm{Hom}_{D_{0}^{b}(X)}(i_{*}\mathcal{E}[1],E_{2})=0$ and
	\begin{align*}
		\Im Z^{\beta,\alpha_{\mathcal{E}}^{\beta},a}(i_{*}\mathcal{E})=v_{2}^{\beta}(i_{*}\mathcal{E})-(\alpha_{\mathcal{E}}^{\beta}-\displaystyle\frac{1}{2}\beta^{2})v_{0}^{\beta}(i_{*}\mathcal{E})=0,
	\end{align*}thus $\Im Z^{\beta,\alpha_{\mathcal{E}}^{\beta},a}(E)=\Im Z^{\beta,\alpha_{\mathcal{E}}^{\beta},a}(E_{2}) \geqslant 0$ by Lemma \ref{imaginarypart2}. And if $\Im Z^{\beta,\alpha_{\mathcal{E}}^{\beta},a}(E)=0$, then $\Im Z^{\beta,\alpha_{\mathcal{E}}^{\beta},a}(E_{2})=0$ and by Lemma \ref{imaginarypart2}, we have a short exact sequence in $\mathrm{Coh}^{\beta,\alpha_{\mathcal{E}}^{\beta}}_{0}(X)$:
	\begin{align*}
		0 \to E_{1}[1] \to E_{2} \to E_{2}/(E_{1}[1]) \to 0
	\end{align*}where $E_{2}/(E_{1}[1])$ is a zero dimensional sheaf or $0$ and $E_{1} \in \mathrm{Coh}_{0}^{\beta}(X)$ is a $\nu^{\beta,\alpha_{\mathcal{E}}^{\beta}}$-semistable object with tilt-slope $\beta$ or $0$. And we have $\mathrm{Hom}_{D_{0}^{b}(X)}(i_{*}\mathcal{E}[1],E_{1}[1]) \subseteq \mathrm{Hom}_{D_{0}^{b}(X)}(i_{*}\mathcal{E}[1],E_{2})=0$. And then the above short exact sequence is in $\mathrm{Coh}^{\beta,i_{*}\mathcal{E}}_{0}(X)$.
	
	(2) The proof follows the same argument as in (1).
\end{pf}

As an analogy of Conjecture \ref{conj}, we propose the following conjecture:

\begin{conj}\label{conj4}
	Let $\beta$ be a real number and $\mathcal{E}$ be an exceptional locally free sheaf with finite rank on $\mathbb{P}^{3}$.
	
	(1) If $\beta<\mu_{1}(\mathcal{E})$ and $i_{*}\mathcal{E} \in \mathrm{Coh}_{0}^{\beta}(X)$ is  $\nu^{\beta,\alpha_{\mathcal{E}}^{\beta}}$-stable, then there exists a real number $a_{0}<\displaystyle\frac{v_{3}^{\beta}(i_{*}\mathcal{E})}{v_{1}^{\beta}(i_{*}\mathcal{E})}$, such that for any $\nu^{\beta,\alpha_{\mathcal{E}}^{\beta}}$-semistable object $E \in \mathrm{Coh}_{0}^{\beta}(X)$ with tilt-slope $\beta$ and $\mathrm{Hom}_{\mathrm{Coh}_{0}^{\beta}(X)}(i_{*}\mathcal{E},E)=0$, we always have the following Bogomolov-Gieseker type inequality:
	\begin{align*}
		v_{3}^{\beta}(E) \leqslant a_{0}v_{1}^{\beta}(E).
	\end{align*}
	
	(2) If $\beta>\mu_{2}(\mathcal{E})$ and $i_{*}\mathcal{E}[1] \in \mathrm{Coh}_{0}^{\beta}(X)$ is  $\nu^{\beta,\alpha_{\mathcal{E}}^{\beta}}$-stable, then there exists a real number $a_{0}<\displaystyle\frac{v_{3}^{\beta}(i_{*}\mathcal{E})}{v_{1}^{\beta}(i_{*}\mathcal{E})}$, such that for any $\nu^{\beta,\alpha_{\mathcal{E}}^{\beta}}$-semistable object $E \in \mathrm{Coh}_{0}^{\beta}(X)$ with tilt-slope $\beta$ and $\mathrm{Hom}_{\mathrm{Coh}_{0}^{\beta}(X)}(i_{*}\mathcal{E}[1],E)=0$, we always have the following Bogomolov-Gieseker type inequality:
	\begin{align*}
		v_{3}^{\beta}(E) \leqslant a_{0}v_{1}^{\beta}(E).
	\end{align*}
\end{conj}

\begin{prop}\label{boundarystabilityfunction}
	If Conjecture \ref{conj4} holds, then $Z^{\beta,\alpha_{\mathcal{E}}^{\beta},a}$ is a stabilty function on $\mathrm{Coh}^{\beta,i_{*}\mathcal{E}}_{0}(X)$ or $\mathrm{Coh}^{\beta,i_{*}\mathcal{E}[1]}_{0}(X)$ for any $a_{0}<a<\displaystyle\frac{v_{3}^{\beta}(i_{*}\mathcal{E})}{v_{1}^{\beta}(i_{*}\mathcal{E})}$.
\end{prop}

\begin{pf}
	If $\beta<\mu_{1}(\mathcal{E})$, then by Proposition \ref{imaginarypart3}, we only need to prove that $\Re Z^{\beta,\alpha_{\mathcal{E}}^{\beta},a}(i_{*}\mathcal{E})<0$ and $\Re Z^{\beta,\alpha_{\mathcal{E}}^{\beta},a}(E)>0$ for any $a_{0}<a<\displaystyle\frac{v_{3}^{\beta}(i_{*}\mathcal{E})}{v_{1}^{\beta}(i_{*}\mathcal{E})}$ and $E \in \mathrm{Coh}_{0}^{\beta}(X)$ is a $\nu^{\beta,\alpha_{\mathcal{E}}^{\beta}}$-semistable object with tilt-slope $\beta$ and $\mathrm{Hom}_{\mathrm{Coh}_{0}^{\beta}(X)}(i_{*}\mathcal{E},E)=0$. The above two inequalities are equivalent to $v_{3}^{\beta}(i_{*}\mathcal{E})> av_{1}^{\beta}(i_{*}\mathcal{E})$  and $v_{3}^{\beta}(E) <av_{1}^{\beta}(E)$. And they holds immediately. When $\beta>\mu_{2}(\mathcal{E})$, we can prove it similarly.
\end{pf}

\begin{rmk}
	(1) By Theorem \ref{thm1}, we know that for any $\nu^{\beta,\alpha_{\mathcal{E}}^{\beta}}$-semistable object $E$ with tilt-slope $\beta$, we have:
	\begin{align*}
		\frac{v_{3}^{\beta}(E)}{v_{1}^{\beta}(E)} \leqslant \frac{2\alpha_{\mathcal{E}}^{\beta}-\beta^{2}}{6}.
	\end{align*} Especially, $i_{*}\mathcal{E}$ or $i_{*}\mathcal{E}[1]$ satisfies this inequality. Conjecture \ref{conj4} claims that there exists $a_{0}<\displaystyle\frac{v_{3}^{\beta}(i_{*}\mathcal{E})}{v_{1}^{\beta}(i_{*}\mathcal{E})} \leqslant \frac{2\alpha_{\mathcal{E}}^{\beta}-\beta^{2}}{6}$, such that $\displaystyle\frac{v_{3}^{\beta}(E)}{v_{1}^{\beta}(E)} \leqslant a_{0}$ for  any $\nu^{\beta,\alpha_{\mathcal{E}}^{\beta}}$-semistable object $E$ with tilt-slope $\beta$ and $\mathrm{Hom}_{\mathrm{Coh}_{0}^{\beta}(X)}(i_{*}\mathcal{E},E)=0$ or $\mathrm{Hom}_{\mathrm{Coh}_{0}^{\beta}(X)}(i_{*}\mathcal{E}[1],E)=0$ in different cases. In fact, Conjecture \ref{conj4} is a stronger version of Bogomolov-Gieseker type inequality.
	
	(2) If $\beta<\mu_{1}(\mathcal{E})$ and $i_{*}\mathcal{E} \in \mathrm{Coh}_{0}^{\beta}(X)$ is  $\nu^{\beta,\alpha_{\mathcal{E}}^{\beta}}$-stable, then let
	\begin{align*}
		a_{\mathcal{E}}^{\beta}=\sup \left\{ \frac{v_{3}^{\beta}(E)}{v_{1}^{\beta}(E)} \mid E \mbox{ is }\nu^{\beta,\alpha_{\mathcal{E}}^{\beta}}\mbox{-semistable with tilt-slope } \beta \mbox{ and } \mathrm{Hom}_{\mathrm{Coh}_{0}^{\beta}(X)}(i_{*}\mathcal{E},E)=0\right\}.
	\end{align*}
	 If we can prove
	\begin{align*}
		a_{\mathcal{E}}^{\beta}<\frac{v_{3}^{\beta}(i_{*}\mathcal{E})}{v_{1}^{\beta}(i_{*}\mathcal{E})}.
	\end{align*}then Conjecture \ref{conj4} will hold if we take $a_{0}=a_{\mathcal{E}}^{\beta}$. But it is not easy to find $a_{\mathcal{E}}^{\beta}$. A similar remark applies to the case $\beta>\mu_{2}(\mathcal{E})$.
\end{rmk}

\begin{cor}\label{Noetherian3}
	Let $\beta$ be a rational number and $\mathcal{E}$ be an exceptional locally free sheaf with finite rank on $\mathbb{P}^{3}$.
	
	(1) If $\beta<\mu_{1}(\mathcal{E})$ and $i_{*}\mathcal{E} \in \mathrm{Coh}_{0}^{\beta}(X)$ is  $\nu^{\beta,\alpha_{\mathcal{E}}^{\beta}}$-stable, then the heart $\mathrm{Coh}^{\beta,i_{*}\mathcal{E}}_{0}(X)$ is Noetherian.
	
	(2) If $\beta>\mu_{2}(\mathcal{E})$ and $i_{*}\mathcal{E}[1] \in \mathrm{Coh}_{0}^{\beta}(X)$ is  $\nu^{\beta,\alpha_{\mathcal{E}}^{\beta}}$-stable, then the heart  $\mathrm{Coh}^{\beta,i_{*}\mathcal{E}[1]}_{0}(X)$ is Noetherian.
\end{cor}

\begin{pf}
	The corollary directly follows from Lemma \ref{boundaryheart} and Proposition \ref{Noetherian2}.
\end{pf}

\begin{lem}
	Let $\beta$ be a rational number and $\mathcal{E}$ be an exceptional locally free sheaf with finite rank on $\mathbb{P}^{3}$.
	
	(1) If $\beta<\mu_{1}(\mathcal{E})$ and $i_{*}\mathcal{E} \in \mathrm{Coh}_{0}^{\beta}(X)$ is  $\nu^{\beta,\alpha_{\mathcal{E}}^{\beta}}$-stable, then the subcategory
	\begin{align*}
		\mathcal{I}^{\beta,i_{*}\mathcal{E}}_{0}(X)=\{E \in \mathrm{Coh}^{\beta,i_{*}\mathcal{E}}_{0}(X) \mid v^{\beta}_{2}(E)=\displaystyle(\alpha_{\mathcal{E}}^{\beta}-\displaystyle\frac{1}{2}\beta^{2})v_{0}^{\beta}(E)\}
	\end{align*} is an abelian category of finite length. 
	
	(2) If $\beta>\mu_{2}(\mathcal{E})$ and $i_{*}\mathcal{E}[1] \in \mathrm{Coh}_{0}^{\beta}(X)$ is  $\nu^{\beta,\alpha_{\mathcal{E}}^{\beta}}$-stable, then the subcategory
	\begin{align*}
		\mathcal{I}^{\beta,i_{*}\mathcal{E}[1]}_{0}(X)=\{E \in \mathrm{Coh}^{\beta,i_{*}\mathcal{E}[1]}_{0}(X) \mid v^{\beta}_{2}(E)=\displaystyle(\alpha_{\mathcal{E}}^{\beta}-\displaystyle\frac{1}{2}\beta^{2})v_{0}^{\beta}(E)\}
	\end{align*} is an abelian category of finite length. 
\end{lem}

\begin{pf}
	(1) We just need to prove $\mathcal{I}^{\beta,i_{*}\mathcal{E}}_{0}(X)[1]$ is Artinian. Suppose we have an infinite chain in $\mathcal{I}^{\beta,i_{*}\mathcal{E}}_{0}(X)[1]$:
	\begin{align*}
		E_{0} \supseteq E_{1} \supseteq E_{2} \supseteq \cdots \supseteq E_{n} \supseteq E_{n+1} \supseteq \cdots.
	\end{align*}
	
	Now, we take the cohomology functor $\mathcal{H}^{i}$ with respect to the heart $\mathrm{Coh}_{0}^{\beta,\alpha_{\mathcal{E}}^{\beta}}(X)$ on the following short exact sequences in $\mathcal{I}^{\beta,i_{*}\mathcal{E}}_{0}(X)[1] \subseteq \mathrm{Coh}^{\beta,i_{*}\mathcal{E}}_{0}(X)[1]$:
	\begin{align*}
		0 \to E_{n+1} \to E_{n} \to E_{n}/E_{n+1} \to 0,
	\end{align*}then we have an infinite chain in $\mathrm{Coh}_{0}^{\beta,\alpha_{\mathcal{E}}^{\beta}}(X)$:
	\begin{align*}
		\mathcal{H}^{-1}(E_{0}) \supseteq \mathcal{H}^{-1}(E_{1}) \supseteq \mathcal{H}^{-1}(E_{2}) \supseteq \cdots \supseteq \mathcal{H}^{-1}(E_{n}) \supseteq \mathcal{H}^{-1}(E_{n+1}) \supseteq \cdots
	\end{align*} and exact sequence in $\mathrm{Coh}_{0}^{\beta,\alpha_{\mathcal{E}}^{\beta}}(X)$:
	\begin{align*}
		0 \to \mathcal{H}^{-1}(E_{n+1}) \to \mathcal{H}^{-1}(E_{n}) \to \mathcal{H}^{-1}(E_{n}/E_{n+1}) \to \mathcal{H}^{0}(E_{n+1}) \to \mathcal{H}^{0}(E_{n}) \to \mathcal{H}^{0}(E_{n}/E_{n+1}) \to 0.
	\end{align*}
	
	Note that if $E \in \mathcal{I}^{\beta,i_{*}\mathcal{E}}_{0}(X)[1]$, then $\mathcal{H}^{-1}(E) \in \mathcal{I}^{\beta,\alpha_{\mathcal{E}}^{\beta}}_{0}(X)$ and $\mathcal{H}^{0}(E) \cong i_{*}\mathcal{E}[1]^{\oplus n}$ for some $n$ by Lemma \ref{boundaryheart} and Proposition \ref{imaginarypart3}.
	
	But we know the abelian category $\mathcal{I}^{\beta,\alpha_{\mathcal{E}}^{\beta}}_{0}(X)$ is of finite length by Lemma \ref{finitelength}. And the cokernel of $\mathcal{H}^{-1}(E_{n+1}) \to \mathcal{H}^{-1}(E_{n})$ is a subobject of $\mathcal{H}^{-1}(E_{n}/E_{n+1})$, and it is still in $\mathcal{I}^{\beta,\alpha_{\mathcal{E}}^{\beta}}_{0}(X)$. Thus, there exists $N_{1} \in \mathbb{N}$, such that when $n>N_{1}$, we have $\mathcal{H}^{-1}(E_{n+1}) \cong \mathcal{H}^{-1}(E_{n})$. In this case, we have an exact sequence in $\mathrm{Coh}_{0}^{\beta,\alpha_{\mathcal{E}}^{\beta}}(X)$:
	\begin{align*}
		0 \to \mathcal{H}^{-1}(E_{n}/E_{n+1}) \to \mathcal{H}^{0}(E_{n+1}) \to \mathcal{H}^{0}(E_{n}) \to \mathcal{H}^{0}(E_{n}/E_{n+1}) \to 0.
	\end{align*} As in the proof of Lemma \ref{boundaryheart}, we can divide it into two short exact sequences in $\mathrm{Coh}_{0}^{\beta,\alpha_{\mathcal{E}}^{\beta}}(X)$ and take the functor $\mathrm{Hom}_{D_{0}^{b}(X)}(i_{*}\mathcal{E}[1],-)$ on both of them and we have: 
	\begin{align*}
		\dim_{\mathbb{C}}\mathrm{Hom}_{D_{0}^{b}(X)}(i_{*}\mathcal{E}[1],\mathcal{H}^{0}(E_{n+1})) \leqslant \dim_{\mathbb{C}}\mathrm{Hom}_{D_{0}^{b}(X)}(i_{*}\mathcal{E}[1],\mathcal{H}^{0}(E_{n})).
	\end{align*}And for $n \gg 0$, it will be an equality. Then $\dim_{\mathbb{C}}\mathrm{Hom}_{D_{0}^{b}(X)}(i_{*}\mathcal{E}[1],\mathcal{H}^{0}(E_{n}/E_{n+1}))=0$ in the same way with the proof of Lemma \ref{boundaryheart}. And it means $E_{n}=E_{n+1}$. 
	
	(2) The proof follows the same argument as in (1).
\end{pf}

We will use the same method to prove the stronger Bogomolov-Gieseker type inequality again, in fact, we will have the following proposition:

\begin{prop}\label{BGlem2}
	Let $\beta$ be a rational number and $\mathcal{E}$ be an exceptional locally free sheaf with finite rank on $\mathbb{P}^{3}$. Assume that $\beta<\mu_{1}(\mathcal{E})$ and $i_{*}\mathcal{E} \in \mathrm{Coh}_{0}^{\beta}(X)$ is  $\nu^{\beta,\alpha_{\mathcal{E}}^{\beta}}$-stable. If there exists another heart $\mathcal{B}$ of a bounded t-structure of $D^{b}_{0}(X)$ with the following properties:
	
	(1) $\mathcal{B} \subseteq \langle \mathrm{Coh}^{\beta,i_{*}\mathcal{E}}_{0}(X),\mathrm{Coh}^{\beta,i_{*}\mathcal{E}}_{0}(X)[1]\rangle$,
	
	(2) there exists $\phi_{0} \in (0,1)$ and $a_{0}<\displaystyle\frac{v_{3}^{\beta}(i_{*}\mathcal{E})}{v_{1}^{\beta}(i_{*}\mathcal{E})}$ such that:
	\begin{align*}
		Z^{\beta,\alpha_{\mathcal{E}}^{\beta},a_{0}}(\mathcal{B}) \subseteq \{r \mathrm{exp}(\sqrt{-1}\pi\phi) \colon r \geqslant 0, \phi_{0} \leqslant \phi \leqslant \phi_{0}+1\},
	\end{align*}
	
	(3) for any $\nu^{\beta,\alpha_{\mathcal{E}}^{\beta}}$-semistable object $F \in \mathrm{Coh}^{\beta}_{0}(X)$ with tilt-slope $\beta$ and $F[1]$ is a simple object in $\mathcal{I}^{\beta,i_{*}\mathcal{E}}_{0}(X)$ with $\mathrm{Hom}_{\mathrm{Coh}^{\beta}_{0}(X)}(i_{*}\mathcal{E},F)=0$, we have $F[2] \notin \mathcal{B}$.

	Then Conjecture \ref{conj4} holds. In other words, for any $\nu^{\beta,\alpha_{\mathcal{E}}^{\beta}}$-semistable object $F \in \mathrm{Coh}_{0}^{\beta}(X)$ with tilt-slope $\beta$ and $\mathrm{Hom}_{\mathrm{Coh}_{0}^{\beta}(X)}(i_{*}\mathcal{E},F)=0$, we always have the following Bogomolov-Gieseker type inequality:
	\begin{align*}
		v_{3}^{\beta}(F) \leqslant a_{0}v_{1}^{\beta}(F).
	\end{align*}
\end{prop}

\begin{pf}
	Take $\mathcal{D}=D_{0}^{b}(X)$, $\mathcal{A}=\mathrm{Coh}^{\beta,i_{*}\mathcal{E}}_{0}(X)$ and $Z=Z^{\beta,\alpha_{\mathcal{E}}^{\beta},a_{0}}$, then it holds by Lemma \ref{BGlem}.
\end{pf}

\begin{rmk}
	There is a similar consequence when $\beta>\mu_{2}(\mathcal{E})$ and 
\end{rmk}

\subsection{One simple case}We want to give stability functions on the Noetherian heart $\mathrm{Coh}^{\beta,i_{*}\mathcal{O}}_{0}(X)$  of $D_{0}^{b}(X)$ in this subsection. Note that
\begin{align*}
	& \alpha_{\mathcal{O}}^{\beta}=\beta^{2}; & \mu_{1}(\mathcal{O})=\mu_{2}(\mathcal{O})=0.
\end{align*} in this case.

\begin{prop}\label{simplecase}
	If $-\displaystyle\frac{1}{2}<\beta<0$ is a rational number and $\displaystyle\frac{3\beta^{3}+6\beta^{2}-4}{6(3\beta+2)}<a<\displaystyle\frac{1}{6}\beta^{2}$ is a real number, then $Z^{\beta,\beta^{2},a}$ is a stability function on $\mathrm{Coh}^{\beta,i_{*}\mathcal{O}}_{0}(X)$ with the Harder-Narasimhan property.
\end{prop}

\begin{pf}
	Consider the following full exceptional collection:
	\begin{align*}
		(\mathcal{O}(-1),\Omega^{2}(2),\Omega(1),\mathcal{O})
	\end{align*} of $D^{b}(\mathbb{P}^{3})$. Then there is a bounded heart of $\mathcal{B}$ of $D_{0}^{b}(X)$ which is of finite length with simple objects $S_{0}=i_{*}\mathcal{O}$, $S_{1}=i_{*}\Omega(1)[1]$, $S_{2}=i_{*}\Omega^{2}(2)[2]$, $S_{3}=i_{*}\mathcal{O}(-1)[3]$ by \cite[Proposition 4.1, Lemma 4.4]{bridgeland2005t}
	
	Note that $\mathcal{O}(-1),\Omega^{2}(2),\Omega(1),\mathcal{O}$ are all slope stable sheaves on $\mathbb{P}^{3}$ and we will have similar results as in Proposition \ref{tiltstablebobject}:

	(1) The object $i_{*}\mathcal{O} \in \mathrm{Coh}^{\beta}_{0}(X)$ is $\nu^{\beta,\beta^{2}}$-stable with tilt-slope $\beta$ and $S_{0}=i_{*}\mathcal{O} \in \mathrm{Coh}^{\beta,i_{*}\mathcal{O}}_{0}(X)$;

	(2) When $-\displaystyle\frac{1}{2} \leqslant \beta<-\displaystyle\frac{1}{3}$, then $i_{*}\Omega(1) \in \mathrm{Coh}^{\beta}_{0}(X)$. Note that 
	\begin{align*}
		v^{-\frac{1}{2}}_{1}(i_{*}\Omega(1))=-1+3\cdot\frac{1}{2}=\frac{1}{2}.
	\end{align*}Then by Lemma \ref{smalltiltslope}, $i_{*}\Omega(1)$ is a $\nu^{-\frac{1}{2},\alpha}$-stable object for all $\alpha>\displaystyle\frac{1}{4}$. If $i_{*}\Omega(1)$ is strictly $\nu^{\beta_{0},\beta_{0}^{2}}$-semistable for some $-\displaystyle\frac{1}{2}<\beta_{0}<-\frac{1}{3}$, then there must be a wall $l$ passing through $(\beta_{0},\beta_{0}^{2})$ and $\Pi(i_{*}\Omega(1))=(-\displaystyle\frac{1}{3},-\frac{1}{6})$ by Proposition \ref{wallandchamberlem1}. We can verify that $l$ must pass through $(-\displaystyle\frac{1}{2},\alpha')$ with some $\alpha'>\displaystyle\frac{1}{4}$. However, it is a contradiction. Thus, $i_{*}\Omega(1)$ is $\nu^{\beta,\beta^{2}}$-stable with tilt-slope $\displaystyle\frac{1+6\beta^{2}}{2+6\beta}<\beta$ with 
	\begin{align*}
		\mathrm{Hom}_{\mathrm{Coh}^{\beta}_{0}(X)}(i_{*}\mathcal{O},i_{*}\Omega(1)) \cong \mathrm{Hom}_{D^{b}(\mathbb{P}^{3})}(\mathcal{O},\Omega(1)) \oplus \mathrm{Hom}_{D^{b}(\mathbb{P}^{3})}(\mathcal{O}(4),\Omega(1)[1])=0.
	\end{align*}
	
	When $\beta=-\displaystyle\frac{1}{3}$, $i_{*}\Omega(1)[1] \in \mathrm{Coh}^{\beta}_{0}(X)$ is $\nu^{\beta,\beta^{2}}$-stable with tilt-slope $+\infty$.
	
	When $-\displaystyle\frac{1}{3}<\beta<0$, then $i_{*}\Omega(1)[1] \in \mathrm{Coh}^{\beta}_{0}(X)$. Note that
	\begin{align*}
		v^{0}_{1}(i_{*}\Omega(1)[1])=1.
	\end{align*}Then by Lemma \ref{smalltiltslope}, $i_{*}\Omega(1)[1]$ is a $\nu^{0,\alpha}$-stable object for any $\alpha>0$. If $i_{*}\Omega(1)[1]$ is strictly $\nu^{\beta_{0},\beta_{0}^{2}}$-semistable for some $-\displaystyle\frac{1}{3}<\beta_{0}<0$, then there must be a wall $l$ passes through $(\beta_{0},\beta_{0}^{2})$ and $\Pi(i_{*}\Omega(1)[1])=(-\displaystyle\frac{1}{3},-\frac{1}{6})$ by Proposition \ref{wallandchamberlem1}. We can verify that $l$ must passes through $(-\displaystyle\frac{1}{2},\alpha')$ with some $\alpha'>\displaystyle\frac{1}{4}$. However, it is a contradiction. Thus, $i_{*}\Omega(1)[1]$ is $\nu^{\beta,\beta^{2}}$-stable with tilt-slope $\displaystyle\frac{1+6\beta^{2}}{2+6\beta}>\beta$. In all cases, $S_{1}=i_{*}\Omega(1)[1] \in \mathrm{Coh}^{\beta,i_{*}\mathcal{O}}_{0}(X)$;

	(3) The object $i_{*}\Omega^{2}(2)[1] \in \mathrm{Coh}^{\beta}_{0}(X)$ is $\nu^{\beta,\beta^{2}}$-stable with tilt-slope $\displaystyle\frac{3\beta^{2}}{2+3\beta}>\beta$. In fact $\mathcal{T}(-2) \cong \Omega^{2}(2)$ and the fact which has just been mentioned holds from the proof of Proposition \ref{tiltstablebobject}. In conclusion, we have $S_{2}[-1]=i_{*}\Omega^{2}(2)[1] \in \mathrm{Coh}^{\beta,i_{*}\mathcal{O}}_{0}(X)$.
	
	(4) The object $i_{*}\mathcal{O}(-1)[1] \in \mathrm{Coh}^{\beta}_{0}(X)$ is $\nu^{\beta,\beta^{2}}$-stable with tilt-slope $\displaystyle\frac{-1+2\beta^{2}}{2+2\beta} <\beta$. Besides, we have
	\begin{align*}
		\mathrm{Hom}_{\mathrm{Coh}^{\beta}_{0}(X)}(i_{*}\mathcal{O},i_{*}\mathcal{O}(-1)[1]) \cong \mathrm{Hom}_{D^{b}(\mathbb{P}^{3})}(\mathcal{O},\mathcal{O}(-1)[1]) \oplus \mathrm{Hom}_{D^{b}(\mathbb{P}^{3})}(\mathcal{O}(4),\mathcal{O}(-1))=0.
	\end{align*}Thus, $S_{3}[-1]=i_{*}\mathcal{O}(-1)[2] \in \mathrm{Coh}^{\beta,i_{*}\mathcal{O}}_{0}(X)$. Thus, we have $\mathcal{B} \subseteq \langle \mathrm{Coh}^{\beta,i_{*}\mathcal{O}}_{0}(X),\mathrm{Coh}^{\beta,i_{*}\mathcal{O}}_{0}(X)[1] \rangle$.
	
	After some calculation, we have: 
	\begin{align*}
		& 	z_{0}=Z^{\beta,\beta^{2},a}(S_{0})=\frac{1}{6}\beta^{3}-a\beta;\\
		&
		z_{1}=Z^{\beta,\beta^{2},a}(S_{1})=-\frac{1}{2}\beta^{3}-\frac{1}{2}\beta^{2}+\frac{1}{2}\beta-\frac{1}{6}+a(3\beta+1)+\sqrt{-1}(-\beta+\frac{1}{2});\\
		&
		z_{2}=Z^{\beta,\beta^{2},a}(S_{2})=\frac{1}{2}\beta^{3}+\beta^{2}-\frac{2}{3}-a(3\beta+2)+\sqrt{-1}(2\beta);\\
		&
		z_{3}=Z^{\beta,\beta^{2},a}(S_{3})=-\frac{1}{6}\beta^{3}-\frac{1}{2}\beta^{2}-\frac{1}{2}\beta-\frac{1}{6}+a(\beta+1)+\sqrt{-1}(-\beta-\frac{1}{2});
	\end{align*} Let $a_{0}=\displaystyle\frac{3\beta^{3}+6\beta^{2}-4}{6(3\beta+2)}$, then we can check that $a_{0}<\displaystyle\frac{v_{3}^{\beta}(i_{*}\mathcal{O})}{v_{1}^{\beta}(i_{*}\mathcal{O})}=\displaystyle\frac{1}{6}\beta^{2}$ for any $-\displaystyle\frac{1}{2}<\beta<0$. And we also can check that 
	\begin{align*}
		Z^{\beta,\beta^{2},a_{0}}(S_{i}) \subseteq \{r \mathrm{exp}(\sqrt{-1}\pi\phi) \colon r>0, \frac{1}{2}<\phi \leqslant \frac{3}{2}\}
	\end{align*} for any $i \in \{0,1,2,3\}$. In fact,
	\begin{align*}
		& \Re z_{0}=\beta(\frac{1}{6}\beta^{2}-a_{0})<0;\\
		& \Re z_{1}=-\frac{1}{2}\beta^{3}-\frac{1}{2}\beta^{2}+\frac{1}{2}\beta-\frac{1}{6}+\frac{3\beta^{3}+6\beta^{2}-4}{6(3\beta+2)}(3\beta+1)=\frac{(\beta+2)(\beta-1)(2\beta+1)}{2(3\beta+2)}<0; \\
		&\Re z_{2}=\frac{1}{2}\beta^{3}+\beta^{2}-\frac{2}{3}-\frac{3\beta^{3}+6\beta^{2}-4}{6(3\beta+2)}(3\beta+2)=0; \\
		& \Re z_{3}=-\frac{1}{6}\beta^{3}-\frac{1}{2}\beta^{2}-\frac{1}{2}\beta-\frac{1}{6}+a_{0}(\beta+1)=(\beta+1)(a_{0}-\frac{1}{6}(\beta+1)^{2})<0;
	\end{align*}The last inequality holds because $\beta^{2} \leqslant (\beta+1)^{2}$ when $\beta> -\displaystyle\frac{1}{2}$.
	
	Finally, suppose that $F \in \mathrm{Coh}_{0}^{\beta}(X)$ is $\nu^{\beta,\beta^{2}}$-stable with tilt-slope $\beta$ and $F[1]$ is a simple object in $\mathcal{I}_{0}^{\beta,i_{*}\mathcal{O}}(X)$ and $\mathrm{Hom}_{\mathrm{Coh}_{0}^{\beta}(X)}(i_{*}\mathcal{O},F)=0$. Because $F,S_{0} \in \mathrm{Coh}^{\beta}_{0}(X)$, then $\mathrm{Hom}_{D^{b}_{0}(X)}(F[2],S_{0})=0$. As for $S_{1}$, we have $S_{1} \in \mathrm{Coh}^{\beta}_{0}(X)$ or $S_{1}[-1] \in \mathrm{Coh}^{\beta}_{0}(X)$, thus $\mathrm{Hom}_{D^{b}_{0}(X)}(F[2],S_{1})=0$. For $S_{2}$, we have $S_{2}[-1] \in \mathrm{Coh}^{\beta}_{0}(X)$ and $\mathrm{Hom}_{D^{b}_{0}(X)}(F[2],S_{2})=\mathrm{Hom}_{D^{b}_{0}(X)}(F[1],S_{2}[-1])=0$. Finally, for $S_{3}$, we have $\mathrm{Hom}_{D^{b}_{0}(X)}(F[2],S_{3})=\mathrm{Hom}_{\mathrm{Coh}_{0}^{\beta}(X)}(F,i_{*}\mathcal{O}(-1)[1])=0$ because $i_{*}\mathcal{O}(-1)[1] \in \mathrm{Coh}^{\beta}_{0}(X)$ is $\nu^{\beta,\beta^{2}}$-stable with tilt-slope $<\beta$.
	
	Therefore, by Proposition \ref{BGlem2} and \ref{boundarystabilityfunction}, $Z^{\beta,\beta^{2},a}$ is a stability function on $\mathrm{Coh}^{\beta,i_{*}\mathcal{O}}_{0}(X)$ for any $a_{0}<a<\displaystyle\frac{1}{6}\beta^{2}$. Moreover, because the image of $\mathrm{Im}Z^{\beta,\beta^{2},a}$ is discrete in $\mathbb{R}$ and $\mathrm{Coh}_{0}^{\beta,i_{*}\mathcal{O}}(X)$ is Noetherian by Corollary \ref{Noetherian3}, then $Z^{\beta,\beta^{2},a}$ has the Harder-Narasimhan filtration by proposition \ref{HNproperty}.
\end{pf}

\begin{rmk}
	(1) It implies that $a_{\mathcal{O}}^{\beta} \leqslant \displaystyle\frac{3\beta^{3}+6\beta^{2}-4}{6(3\beta+2)}$ when $-\displaystyle\frac{1}{2}<\beta<0$ and $\beta$ is rational.
	
	(2) In fact, the pair $(Z^{\beta,\beta^{2},a},\mathrm{Coh}_{0}^{\beta,i_{*}\mathcal{O}}(X))$ is a stability condition on $D_{0}^{b}(X)$ when $-\displaystyle\frac{1}{2}<\beta<0$ is rational and $\displaystyle\frac{3\beta^{3}+6\beta^{2}-4}{6(3\beta+2)}<a<\displaystyle\frac{1}{6}\beta^{2}$. We will prove this fact in the next subsection.
\end{rmk}

\subsection{General Case} 

In this subsection, we want to give stability functions on the Noetherian bounded heart $\mathrm{Coh}^{\beta,i_{*}\mathcal{E}}_{0}(X)$ of $D_{0}^{b}(X)$.  It is not easy in general, but we can give one method to make it possible with some assumption. To make things easier, we always assume that $\beta<\mu_{1}(\mathcal{E}) \leqslant \mu(\mathcal{E})$ is a rational number and $i_{*}\mathcal{E} \in \mathrm{Coh}_{0}^{\beta}(X)$ is $\nu^{\beta,\alpha_{\mathcal{E}}^{\beta}}$-stable. 

In order to apply Lemma \ref{BGlem2}, we first need to find a suitable bounded heart $\mathcal{B}$. In fact, we want to assume that it is induced by some full exceptional collection $(F_{0},F_{1},F_{2},F_{3})$ of $D^{b}(\mathbb{P}^{3})$. Explicitly, suppose that there exists a full exceptional collection $(F_{0},F_{1},F_{2},F_{3})$ of $D^{b}(\mathbb{P}^{3})$ with $F_{i}$ are locally free sheaves of finite rank and $F_{3}=\mathcal{E}$. Then there is a bounded heart $\mathcal{B}$ of $D_{0}^{b}(X)$ which is of finite length and with simple objects $S_{0}=i_{*}\mathcal{E}$, $S_{1}=i_{*}F_{2}[1]$, $S_{2}=i_{*}F_{1}[2]$, $S_{3}=i_{*}F_{0}[3]$ by \cite[Proposition4.1, Lemma 4.4]{bridgeland2005t}. We also assume that $i_{*}F_{j} \in D_{0}^{b}(X)$ are $\nu^{\beta,\alpha_{\mathcal{E}}^{\beta}}$-semistable for every $j \in \{0,1,2,3\}$. One of our aim is to establish suitable conditions relating $\beta$ and $(F_{0},F_{1},F_{2},F_{3})$ such that  $\mathcal{B} \subseteq \langle \mathrm{Coh}^{\beta,i_{*}\mathcal{E}}_{0}(X),\mathrm{Coh}^{\beta,i_{*}\mathcal{E}}_{0}(X)[1]\rangle$. Note that we always have $\mu(F_{0})<\mu(F_{1})<\mu(F_{2})<\mu(F_{3})$ by \cite[Lemma 2.8]{gould2022constructive} and we will clarify the inequality of $\beta$ through a case-by-case analysis.

(1) It is not difficult to verify that $S_{0}=i_{*}\mathcal{E} \in \mathrm{Coh}^{\beta,i_{*}\mathcal{E}}_{0}(X)$ when $\beta<\mu_{1}(\mathcal{E})$.

(2) When $\beta<\mu(F_{0})$, then $S_{3}[-3]=i_{*}F_{0} \in \mathrm{Coh}_{0}^{\beta}(X)$ and it implies that $S_{3}[-2]$ or $S_{3}[-3]$ in $\mathrm{Coh}^{\beta,i_{*}\mathcal{E}}_{0}(X)$, it is a contradiction. When $\mu(F_{0}) \leqslant \beta<\mu(F_{3})$, then $S_{3}[-2]=i_{*}F_{0}[1] \in \mathrm{Coh}_{0}^{\beta}(X)$. If $\nu^{\beta,\alpha_{\mathcal{E}}^{\beta}}(i_{*}F_{0}[1])>\beta$, then $S_{3}[-2] \in \mathrm{Coh}^{\beta,i_{*}\mathcal{E}}_{0}(X)$, it is a contradiction. Thus we must have $\nu^{\beta,\alpha_{\mathcal{E}}^{\beta}}(i_{*}F_{0}[1]) \leqslant \beta$ and thus $\mu(F_{0}) \neq \beta$. Moreover, we expect that $\nu^{\beta,\alpha_{\mathcal{E}}^{\beta}}(i_{*}F_{0}[1])<\beta$ for our convenience. 

(3) When $\mu(F_{0})<\beta<\mu(F_{1})$, then  $S_{2}[-2]=i_{*}F_{1} \in \mathrm{Coh}_{0}^{\beta}(X)$. Similarly, if we want to have $S_{2} \in \langle \mathrm{Coh}^{\beta,i_{*}\mathcal{E}}_{0}(X),\mathrm{Coh}^{\beta,i_{*}\mathcal{E}}_{0}(X)[1]\rangle$, then we must have $\nu^{\beta,\alpha_{\mathcal{E}}^{\beta}}(i_{*}F_{1}) \leqslant \beta$. We also expect that $\nu^{\beta,\alpha_{\mathcal{E}}^{\beta}}(i_{*}F_{1})<\beta$ for our convenience. When $\mu(F_{1}) \leqslant \beta<\mu(F_{3})$, then $S_{2}[-1]=i_{*}F_{1}[1] \in \mathrm{Coh}_{0}^{\beta}(X)$, and whether $\nu^{\beta,\alpha_{\mathcal{E}}^{\beta}}(i_{*}F_{1}[1]) \leqslant \beta$ or $\nu^{\beta,\alpha_{\mathcal{E}}^{\beta}}(i_{*}F_{1}[1]) >\beta$, we have $S_{2}=i_{*}F_{1}[2] \in \langle \mathrm{Coh}^{\beta,i_{*}\mathcal{E}}_{0}(X),\mathrm{Coh}^{\beta,i_{*}\mathcal{E}}_{0}(X)[1]\rangle$.

(4) When $\mu(F_{0})<\beta<\mu(F_{2})$, then $S_{1}[-1]=i_{*}F_{2} \in \mathrm{Coh}_{0}^{\beta}(X)$ and whether $\nu^{\beta,\alpha_{\mathcal{E}}^{\beta}}(i_{*}F_{2}) \leqslant \beta$ or $\nu^{\beta,\alpha_{\mathcal{E}}^{\beta}}(i_{*}F_{2}) >\beta$, we always have $S_{1}=i_{*}F_{2}[1] \in \langle \mathrm{Coh}^{\beta,i_{*}\mathcal{E}}_{0}(X),\mathrm{Coh}^{\beta,i_{*}\mathcal{E}}_{0}(X)[1]\rangle$. When $\mu(F_{2}) \leqslant \beta<\mu(F_{3})$, then $S_{1}=i_{*}F_{2}[1] \in \mathrm{Coh}_{0}^{\beta}(X)$. Similarly, if we want that $S_{1} \in \langle \mathrm{Coh}^{\beta,i_{*}\mathcal{E}}_{0}(X),\mathrm{Coh}^{\beta,i_{*}\mathcal{E}}_{0}(X)[1]\rangle$, then we must have $\nu^{\beta,\alpha_{\mathcal{E}}^{\beta}}(i_{*}F_{2}[1]) >\beta$ in this case.

In conclusion, we have the following proposition:

\begin{prop}\label{generalproposition}
	Assume that $\mathcal{E} \in \mathrm{Coh}(\mathbb{P}^{3})$ is an exceptional locally free sheaf of finite rank. Suppose that there exist a rational number $\beta$, a real number $a_{0}$ and a full exceptional collection $(F_{0},F_{1},F_{2},F_{3}=\mathcal{E})$ consisting of finite rank locally free sheaves on $\mathbb{P}^{3}$ such that the objects $i_{*}F_{j} \in D_{0}^{b}(X)$ are $\nu^{\beta,\alpha_{\mathcal{E}}^{\beta}}$-stable for every $j \in \{0,1,2,3\}$. If the numbers $\beta$ and $a_{0}$ satisfies:
	
	(1) $\beta<\mu_{1}(\mathcal{E})$, $\beta>\mu(F_{0})$ and $\displaystyle\frac{v_{2}(F_{0})-\alpha_{\mathcal{E}}^{\beta}v_{0}(F_{0})}{v_{1}^{\beta}(F_{0})}<\beta$.
	
	(2) One of the following conditions holds:
	\begin{itemize}
		\item $\mu(F_{0})<\beta<\mu(F_{1})$, $\displaystyle\frac{v_{2}(F_{1})-\alpha_{\mathcal{E}}^{\beta}v_{0}(F_{1})}{v_{1}^{\beta}(F_{1})}<\beta$;
		\item $\mu(F_{1})\leqslant\beta\leqslant\mu(F_{2})$; 
		\item $\mu(F_{2})<\beta<\mu(F_{3})$,   $\displaystyle\frac{v_{2}(F_{2})-\alpha_{\mathcal{E}}^{\beta}v_{0}(F_{2})}{v_{1}^{\beta}(F_{2})}>\beta$.
	\end{itemize}
	
	(3)
	\begin{align*}
		& v_{3}^{\beta}(F_{0})<\displaystyle\frac{v_{3}^{\beta}(\mathcal{E})}{v_{1}^{\beta}(\mathcal{E})}v_{1}^{\beta}(F_{0});\\
		& v_{3}^{\beta}(F_{1})>\displaystyle\frac{v_{3}^{\beta}(\mathcal{E})}{v_{1}^{\beta}(\mathcal{E})}v_{1}^{\beta}(F_{1});\\
		&
		v_{3}^{\beta}(F_{2})<\displaystyle\frac{v_{3}^{\beta}(\mathcal{E})}{v_{1}^{\beta}(\mathcal{E})}v_{1}^{\beta}(F_{2});
	\end{align*} 

	(4) For any $j \in \{0,1,2,3\}$, we have
	\begin{align*}
		Z^{\beta,\alpha_{\mathcal{E}}^{\beta},a_{0}}(i_{*}F_{j}[3-j]) \in \left\{r \mathrm{exp}(\sqrt{-1}\pi\phi) \mid r>0, \frac{1}{2}<\phi \leqslant \frac{3}{2}\right\}.
	\end{align*}
	
	Then $Z^{\beta,\alpha_{\mathcal{E}}^{\beta},a}$ will be a stability function on $\mathrm{Coh}_{0}^{\beta,i_{*}\mathcal{E}}(X)$ for any $a_{0}<a<\displaystyle\frac{v_{3}^{\beta}(\mathcal{E})}{v_{1}^{\beta}(\mathcal{E})}$. Moreover, $
	\sigma^{\beta,i_{*}\mathcal{E},a}=(Z^{\beta,\alpha_{\mathcal{E}}^{\beta},a},\mathrm{Coh}_{0}^{\beta,i_{*}\mathcal{E}}(X)) \in \mathrm{Stab}_{H}(D_{0}^{b}(X))$ for any $a_{0}<a<\displaystyle\frac{v_{3}^{\beta}(\mathcal{E})}{v_{1}^{\beta}(\mathcal{E})}$.
\end{prop}

\begin{pf}
	Let $\mathcal{B}$ be the full extension-closed subcategory of $D_{0}^{b}(X)$ generated by $S_{0}=i_{*}\mathcal{E}$, $S_{1}=i_{*}F_{2}[1]$, $S_{2}=i_{*}F_{1}[2]$, $S_{3}=i_{*}F_{0}[3]$. By \cite[Proposition4.1, Lemma 4.4]{bridgeland2005t}, it is a bounded heart of $D_{0}^{b}(X)$, with four simple objects $S_{0},S_{1},S_{2},S_{3}$. From our discussion above, we have $\mathcal{B} \subseteq \langle \mathrm{Coh}^{\beta,i_{*}\mathcal{E}}_{0}(X),\mathrm{Coh}^{\beta,i_{*}\mathcal{E}}_{0}(X)[1]\rangle$.
		
	By Proposition \ref{BGlem2}, we only need to prove the following fact: for any $\nu^{\beta,\alpha_{\mathcal{E}}^{\beta}}$-semistable object $F \in \mathrm{Coh}^{\beta}_{0}(X)$ with tilt-slope $\beta$ and $F[1]$ is a simple object in $\mathcal{I}^{\beta,i_{*}\mathcal{E}}_{0}(X)$ with $\mathrm{Hom}_{\mathrm{Coh}^{\beta}_{0}(X)}(i_{*}\mathcal{E},F)=0$, we have $\mathrm{Hom}_{\mathrm{Coh}^{\beta}_{0}(X)}(F[2],S_{i})=0$. It is obvious when $i=0,1$. As for $S_{3}$, we have $\mathrm{Hom}_{D^{b}_{0}(X)}(F[2],S_{3})=\mathrm{Hom}_{\mathrm{Coh}_{0}^{\beta}(X)}(F,i_{*}F_{0}[1])=0$ because $i_{*}F_{0}[1] \in \mathrm{Coh}^{\beta}_{0}(X)$ is $\nu^{\beta,\alpha_{\mathcal{E}}^{\beta}}$-semistable with tilt-slope $<\beta$. Similarly, when $\beta \geqslant \mu(F_{1})$, then $\mathrm{Hom}_{D^{b}_{0}(X)}(F[2],S_{2})=\mathrm{Hom}_{\mathrm{Coh}_{0}^{\beta}(X)}(F[1],i_{*}F_{1}[1])=0$ because $i_{*}F_{1}[1] \in \mathrm{Coh}^{\beta}_{0}(X)$ and when $\beta<\mu(F_{1})$, then $\mathrm{Hom}_{D^{b}_{0}(X)}(F[2],S_{2})=\mathrm{Hom}_{\mathrm{Coh}_{0}^{\beta}(X)}(F,i_{*}F_{1})=0$ because $i_{*}F_{1} \in \mathrm{Coh}^{\beta}_{0}(X)$ is $\nu^{\beta,\alpha_{\mathcal{E}}^{\beta}}$-semistable with tilt-slope $<\beta$. 
	
	Now, because $\beta$ is rational, then the image of $\mathrm{Im}Z^{\beta,\alpha_{\mathcal{E}}^{\beta},a}=v_{2}^{\beta}-(\alpha_{\mathcal{E}}^{\beta}-\displaystyle\frac{\beta^{2}}{2})v_{0}^{\beta}$ is discrete and the heart $\mathrm{Coh}_{0}^{\beta,i_{*}\mathcal{E}}(X)$ is Noetherian, then by Proposition \ref{HNproperty}, we know that the stability function $Z^{\beta,\alpha_{\mathcal{E}}^{\beta},a}$ on $\mathrm{Coh}^{\beta,i_{*}\mathcal{E}}_{0}(X)$ has the Harder-Narasimhan property.
	
	Finally, for any $a_{0}<a<\displaystyle\frac{v_{3}^{\beta}(\mathcal{E})}{v_{1}^{\beta}(\mathcal{E})}$, we have:
	\begin{align*}
		\mathrm{Re}Z^{\beta,\alpha_{\mathcal{E}}^{\beta},a}(S_{3})=-v_{3}^{\beta}(S_{3})+av_{1}^{\beta}(S_{3})=v_{3}^{\beta}(F_{0})-av_{1}^{\beta}(F_{0})<v_{3}^{\beta}(F_{0})-\frac{v_{3}^{\beta}(\mathcal{E})}{v_{1}^{\beta}(\mathcal{E})}v_{1}^{\beta}(F_{0})<0.
	\end{align*}since $v_{1}^{\beta}(F_{0})<0$. When $\mu(F_{0})<\beta\leqslant \mu(F_{1})$, we have:
	\begin{align*}
		\mathrm{Re}Z^{\beta,\alpha_{\mathcal{E}}^{\beta},a}(S_{2})=-v_{3}^{\beta}(S_{2})+av_{1}^{\beta}(S_{2})=-v_{3}^{\beta}(F_{1})+av_{1}^{\beta}(F_{1})\leqslant-v_{3}^{\beta}(F_{1})+\frac{v_{3}^{\beta}(\mathcal{E})}{v_{1}^{\beta}(\mathcal{E})}v_{1}^{\beta}(F_{1})<0.
	\end{align*}since $v_{1}^{\beta}(F_{1}) \geqslant 0$. 
	When $\mu(F_{1})<\beta$, we have:
	\begin{align*}
		\mathrm{Re}Z^{\beta,\alpha_{\mathcal{E}}^{\beta},a}(S_{2})=\mathrm{Re}Z^{\beta,\alpha_{\mathcal{E}}^{\beta},a_{0}}(F_{1})+(a-a_{0})v_{1}^{\beta}(F_{1})<\mathrm{Re}Z^{\beta,\alpha_{\mathcal{E}}^{\beta},a_{0}}(F_{1})<0
	\end{align*} since $v_{1}^{\beta}(F_{1})<0$ and $\mathrm{Re}Z^{\beta,\alpha_{\mathcal{E}}^{\beta},a_{0}}(F_{1})<0$ by condition (4). By the same reason, we can prove that $\mathrm{Re}Z^{\beta,\alpha_{\mathcal{E}}^{\beta},a}(S_{1})\leqslant 0$ for any $a_{0}<a<\displaystyle\frac{v_{3}^{\beta}(\mathcal{E})}{v_{1}^{\beta}(\mathcal{E})}$.
	 
	In conclusion, for any $a_{0}<a<\displaystyle\frac{v_{3}^{\beta}(\mathcal{E})}{v_{1}^{\beta}(\mathcal{E})}$ and $j \in \{0,1,2,3\}$, we have $\mathrm{Re}Z^{\beta,\alpha_{\mathcal{E}}^{\beta},a}(S_{j})<0$. This implies the existence of algebraic stability conditions $\tau^{\beta,\alpha_{\mathcal{E}}^{\beta},a}=(-\sqrt{-1}Z^{\beta,\alpha_{\mathcal{E}}^{\beta},a},\mathcal{B})$ for any $a_{0}<a<\displaystyle\frac{v_{3}^{\beta}(\mathcal{E})}{v_{1}^{\beta}(\mathcal{E})}$. Under the $\widetilde{\mathrm{GL}}_{2}(\mathbb{R})$-action, there exist algebraic stability conditions $\widetilde{\sigma}^{\beta,\alpha_{\mathcal{E}}^{\beta},a}=(Z^{\beta,\alpha_{\mathcal{E}}^{\beta},a},\mathcal{B}^{\beta,\alpha_{\mathcal{E}}^{\beta},a})$, where $\mathcal{B}^{\beta,\alpha_{\mathcal{E}}^{\beta},a} \subseteq \langle \mathcal{B},\mathcal{B}[-1]\rangle$. Then $\sigma^{\beta,i_{*}\mathcal{E},a}=\widetilde{\sigma}^{\beta,\alpha_{\mathcal{E}}^{\beta},a}$ by the following lemma. The latter is an algebraic stability condition and satisfies support property automatically.
\end{pf}

\begin{lem}\cite[Lemma 8.11]{bayer2016space}
	Assume that $\sigma_{1}=(Z,\mathcal{A}_{1})$ and $\sigma_{2}=(Z,\mathcal{A}_{2})$ are two pre-stability conditions on the triangulated category $\mathcal{D}$ with
	the following properties:
	
	(1) Their central charges agree.
	
	(2) There exists a bounded heart $\mathcal{B}$ of $\mathcal{D}$ such that:
	\begin{align*}
		\mathcal{A}_{1},\mathcal{A}_{2} \subseteq \langle \mathcal{B},\mathcal{B}[-1]\rangle
	\end{align*}
	then $\sigma_{1}=\sigma_{2}$.
\end{lem}

\begin{ex}
	Now, let $\mathcal{E}=\mathcal{O}$ and $(F_{0},F_{1},F_{2},F_{3})=(\mathcal{O}(-1),\Omega^{2}(2),\Omega(1),\mathcal{O})$. Then $\alpha_{\mathcal{O}}^{\beta}=\beta^{2}$, $\mu_{1}(\mathcal{O})=0$ and $\displaystyle\frac{v_{3}^{\beta}(\mathcal{O})}{v_{1}^{\beta}(\mathcal{O})}=\frac{1}{6}\beta^{2}$. The condition in Proposition \ref{generalproposition} is as follows:

	(1) $\beta<0$, $\beta>-1$ and $\beta>-\displaystyle\frac{1}{2}$;

	(2) This condition is satisfied automatically;
	
	
	(3) $\beta>-\displaystyle\frac{1}{2}$, $-1<\beta<1$, $\beta<\displaystyle\frac{1}{2}$;
	
	(4) As for this condition, first notice that $\mathrm{Re}Z^{\beta,\beta^{2},a_{0}}(F_{3})<0$ and $\mathrm{Re}Z^{\beta,\beta^{2},a_{0}}(F_{0}[3])<0$ by our assumptions and the above three conditions. Now we need:
	\begin{align*}
	 & v_{3}^{\beta}(\Omega^{2}(2)) \geqslant a_{0}v_{1}^{\beta}(\Omega^{2}(2)).
	\end{align*}Note that the above three conditions tell us $\beta>-\displaystyle\frac{1}{2}$. Thus $v_{3}^{\beta}(\Omega^{2}(2)) \geqslant a_{0}v_{1}^{\beta}(\Omega^{2}(2))$ implies $a_{0} \geqslant \displaystyle\frac{v_{3}^{\beta}(\Omega^{2}(2))}{v_{1}^{\beta}(\Omega^{2}(2))}=\frac{3\beta^{3}+6\beta^2-4}{6(3\beta+2)}$. Meahwhile, $\mathrm{Im}Z^{\beta,\beta^{2},a_{0}}(F_{1}[2])=\displaystyle v_{2}^{\beta}(F_{1}[2])-(\alpha_{\mathcal{O}}^{\beta}-\frac{\beta^{2}}{2})v_{0}^{\beta}(F_{1}[2])=2\beta<0$ for any $-\displaystyle\frac{1}{2}<\beta<0$. In other words, when $a_{0} \geqslant \displaystyle\frac{3\beta^{3}+6\beta^2-4}{6(3\beta+2)}$, then $Z^{\beta,\beta^{2},a_{0}}(F_{1}[2]) \in \left\{r \mathrm{exp}(\sqrt{-1}\pi\phi) \mid r>0, \displaystyle\frac{1}{2}<\phi \leqslant \frac{3}{2}\right\}$.
	
	Next, we could verify that the following facts:
	
	When $-\displaystyle\frac{1}{2}<\beta<-\displaystyle\frac{1}{3}$, then 
	\begin{align*}
		\frac{v_{3}^{\beta}(\Omega(1))}{v_{1}^{\beta}(\Omega(1))}=\frac{3\beta^{3}+3\beta^2-3\beta+1}{6(3\beta+1)}< \displaystyle\frac{3\beta^{3}+6\beta^2-4}{6(3\beta+2)}=\frac{v_{3}^{\beta}(\Omega^{2}(2))}{v_{1}^{\beta}(\Omega^{2}(2))} \leqslant a_{0}.
	\end{align*}When $\beta=-\displaystyle\frac{1}{3}$, then $-6v_{3}^{\beta}(\Omega(1))=3\beta^{3}+3\beta^2-3\beta+1|_{\beta=-\frac{1}{3}}=\displaystyle\frac{20}{9}>0$. When $-\displaystyle\frac{1}{3}<\beta<0$, then  \begin{align*}
	\frac{v_{3}^{\beta}(\Omega(1))}{v_{1}^{\beta}(\Omega(1))}=\frac{3\beta^{3}+3\beta^2-3\beta+1}{6(3\beta+1)}> \displaystyle\frac{1}{6}\beta^{2}=\frac{v_{3}^{\beta}(\mathcal{O})}{v_{1}^{\beta}(\mathcal{O})}>a_{0}.
	\end{align*}In conclusion, $\mathrm{Re}Z^{\beta,\beta^{2},a_{0}}(F_{2}[1])<0$. Therefore, by Proposition \ref{generalproposition}, we get Proposition \ref{simplecase} again.
\end{ex}

\begin{ex}
	If $\mathcal{E}=\mathcal{O}$, and $(F_{0},F_{1},F_{2},F_{3})=(\mathcal{O}(-3),\mathcal{O}(-2),\mathcal{O}(-1),\mathcal{O})$. Then $\alpha_{\mathcal{O}}^{\beta}=\beta^{2}$, $\mu_{1}(\mathcal{O})=0$ and $\displaystyle\frac{v_{3}^{\beta}(\mathcal{O})}{v_{1}^{\beta}(\mathcal{O})}=\frac{1}{6}\beta^{2}$. The condition in Proposition \ref{generalproposition} is as follows:
	 
	(1) $\beta<0$, $\beta>-3$ and $\beta>-\displaystyle\frac{3}{2}$;
	
	(2) This condition is satidfied automatically;
	
	(3) $\beta>-\displaystyle\frac{3}{2}$, $-2<\beta<-1$, $\beta<-\displaystyle\frac{1}{2}$;
	
	(4) As for this condition, first notice that $\mathrm{Re}Z^{\beta,\beta^{2},a_{0}}(F_{3})<0$ and $\mathrm{Re}Z^{\beta,\beta^{2},a_{0}}(F_{0}[3])<0$ by our assumptions and the above three conditions. Now we need:
	\begin{align*}
		& v_{3}^{\beta}(\mathcal{O}(-2)) \geqslant a_{0}v_{1}^{\beta}(\mathcal{O}(-2)).
	\end{align*}Note that the above three conditions tell us $\beta>-\displaystyle\frac{3}{2}$. Thus $v_{3}^{\beta}(\mathcal{O}(-2)) \geqslant a_{0}v_{1}^{\beta}(\mathcal{O}(-2))$ implies $a_{0} \geqslant \displaystyle\frac{v_{3}^{\beta}(\mathcal{O}(-2))}{v_{1}^{\beta}(\mathcal{O}(-2))}=\frac{1}{6}(\beta+2)^{2}$. Meahwhile, $\mathrm{Im}Z^{\beta,\beta^{2},a_{0}}(F_{1}[2])=\displaystyle v_{2}^{\beta}(F_{1}[2])-(\alpha_{\mathcal{O}}^{\beta}-\frac{\beta^{2}}{2})v_{0}^{\beta}(F_{1}[2])=2+2\beta<0$ for any $-\displaystyle\frac{3}{2}<\beta<-1$. In other words, when $a_{0} \geqslant \displaystyle\frac{1}{6}(\beta+2)^{2}$, then $Z^{\beta,\beta^{2},a_{0}}(F_{1}[2]) \in \left\{r \mathrm{exp}(\sqrt{-1}\pi\phi) \mid r>0, \displaystyle\frac{1}{2}<\phi \leqslant \frac{3}{2}\right\}$.
	
	We could verify that  
	\begin{align*}
		\frac{v_{3}^{\beta}(\mathcal{O}(-1))}{v_{1}^{\beta}(\mathcal{O}(-1))}=\frac{1}{6}(\beta+1)^{2}< \frac{1}{6}(\beta+2)^{2}=\frac{v_{3}^{\beta}(\mathcal{O}(-2))}{v_{1}^{\beta}(\mathcal{O}(-2))} \leqslant a_{0}.
	\end{align*} when $-\displaystyle\frac{3}{2}<\beta<-1$. In conclusion, $\mathrm{Re}Z^{\beta,\beta^{2},a_{0}}(F_{2}[1])<0$. Therefore, by Proposition \ref{generalproposition}, we conclude that $(Z^{\beta,\alpha_{\mathcal{O}}^{\beta},a},\mathrm{Coh}_{0}^{\beta,i_{*}\mathcal{O}}(X))$ will be a stability condition for any rational number $-\displaystyle\frac{3}{2}<\beta<-1$ and  $\displaystyle\frac{1}{6}(\beta+2)^{2}<a<\frac{1}{6}\beta^{2}$.
\end{ex}

\begin{rmk}
	Some of conditions in Proposition \ref{generalproposition} are somewhat redundant, but it is currently difficult to remove them because we do not yet fully understand full exceptional collections of $D^{b}(\mathbb{P}^{3})$.
\end{rmk}

\subsection{Boundary point}

In this subsection, we will prove that the stability conditions we have constructed above are on the boundary of the subset $\mathrm{Stab}_{H}^{\mathrm{geo}}(D_{0}^{b}(X))$ of geometric stability conditions.

\begin{prop}\label{exceptionalissimple1}
	Let $\beta \in \mathbb{R}$ and $\mathcal{E}$ be an exceptional locally free sheaf with finite rank on $\mathbb{P}^{3}$. Assume that $\beta<\mu_{1}(\mathcal{E})$ and $i_{*}\mathcal{E}$ is $\nu^{\beta,\alpha_{\mathcal{E}}^{\beta}}$-stable.
	Then the object $i_{*}\mathcal{E}[1] \in \mathrm{Coh}^{\beta,\alpha_{\mathcal{E}}^{\beta}}_{0}(X)$ is a $\sigma^{\beta,\alpha_{\mathcal{E}}^{\beta},a}$-stable object for any $a>\displaystyle\frac{2\alpha_{\mathcal{E}}^{\beta}-\beta^{2}}{6}$.
\end{prop}

\begin{pf}
	 Notice that $i_{*}\mathcal{E}[1] \in \mathcal{I}_{0}^{\beta,\alpha_{\mathcal{E}}^{\beta}}(X)$ and we just need to prove there are no non-trivial short exact sequences $0 \to F \to i_{*}\mathcal{E}[1] \to G \to 0$ in the abelian category $\mathcal{I}_{0}^{\beta,\alpha_{\mathcal{E}}^{\beta}}(X)$. Otherwise, we can take the cohomology functor $\mathcal{H}^{i}_{\beta}$ with respect to the heart $\mathrm{Coh}^{\beta}_{0}(X)$. Then we have a long exact sequence in $\mathrm{Coh}^{\beta}_{0}(X)$:
	\begin{align*}
		0 \to \mathcal{H}^{-1}_{\beta}(F) \to i_{*}\mathcal{E} \to \mathcal{H}^{-1}_{\beta}(G) \to \mathcal{H}^{0}_{\beta}(F) \to 0 \to \mathcal{H}^{0}_{\beta}(G) \to 0.
	\end{align*}
	
	Divide the exact sequence into two short exact sequences in $\mathrm{Coh}^{\beta}_{0}(X)$:
	\begin{align*}
		& 0 \to \mathcal{H}^{-1}_{\beta}(F) \to i_{*}\mathcal{E} \to C \to 0,\\
		& 0 \to C \to \mathcal{H}^{-1}_{\beta}(G) \to \mathcal{H}^{0}_{\beta}(F) \to 0.
	\end{align*}
	Because $C$ is a subobject of $\mathcal{H}^{-1}_{\beta}(G) \in \mathcal{F}^{\beta,\alpha_{\mathcal{E}}^{\beta}}_{0}(X) \subseteq \mathrm{Coh}^{\beta}_{0}(X)$, we have $C \in \mathcal{F}^{\beta,\alpha_{\mathcal{E}}^{\beta}}_{0}(X)$ and either $C=0$ or $\nu^{\beta,\alpha_{\mathcal{E}}^{\beta}}(C)<+\infty$.
	
	Meanwhile, note that $i_{*}\mathcal{E}$ is a $\nu^{\beta,\alpha_{\mathcal{E}}^{\beta}}$-stable object in $\mathrm{Coh}^{\beta}_{0}(X)$ with slope $\beta$, and $\mathcal{H}^{-1}_{\beta}(F)$ is also a $\nu^{\beta,\alpha_{\mathcal{E}}^{\beta}}$-semistable object in $\mathrm{Coh}^{\beta}_{0}(X)$ with slope $\beta$. Then $\mathcal{H}^{-1}_{\beta}(F)=0$ or $\mathcal{H}^{-1}_{\beta}(F) \cong i_{*}\mathcal{E}$. If $\mathcal{H}^{-1}_{\beta}(F) \cong i_{*}\mathcal{E}$, then $\mathcal{H}^{-1}_{\beta}(G) \cong \mathcal{H}^{0}_{\beta}(F)$ and both of them are $0$ and $G=0$.

	 If $\mathcal{H}^{-1}_{\beta}(F)=0$, we have $\mathrm{Hom}_{D_{0}^{b}(X)}(\mathcal{H}^{0}_{\beta}(F),i_{*}\mathcal{E}[1])=0$ since $\mathcal{H}^{0}_{\beta}(F)$ is a zero-dimensional sheaf. Therefore, $\mathcal{H}^{-1}_{\beta}(G)$ splits into the direct sum $ i_{*}\mathcal{E} \oplus \mathcal{H}^{0}_{\beta}(F)$, it contradicts to the fact $\mathcal{H}^{-1}_{\beta}(G)$ is a $\nu^{\beta,\alpha_{\mathcal{E}}^{\beta}}$-semistable object in $\mathrm{Coh}^{\beta}_{0}(X)$ if $\mathcal{H}^{0}_{\beta}(F) \neq 0$. Therefore $\mathcal{H}_{\beta}^{0}(F)=0$ and we have $F=0$.
\end{pf}

\begin{prop}\label{exceptionalissimple2}
	Let $\beta \in \mathbb{R}$ and $\mathcal{E}$ be an exceptional locally free sheaf with finite rank on $\mathbb{P}^{3}$. Assume that $\beta<\mu_{1}(\mathcal{E})$ and $i_{*}\mathcal{E}$ is $\nu^{\beta,\alpha_{\mathcal{E}}^{\beta}}$-stable. Then the object $i_{*}\mathcal{E} \in \mathcal{I}_{0}^{\beta,i_{*}\mathcal{E}}(X) \subset \mathrm{Coh}^{\beta,i_{*}\mathcal{E}}_{0}(X)$ is a simple object.
\end{prop}

\begin{pf}
	We just need to prove there are no non-trivial short exact sequences $0 \to F \to i_{*}\mathcal{E}[1] \to G \to 0$ in the abelian category $\mathcal{I}_{0}^{\beta,i_{*}\mathcal{E}}(X)[1]$. Otherwise, we can take the cohomology functor $\mathcal{H}^{i}$ with respect to the heart $\mathrm{Coh}^{\beta,\alpha}_{0}(X)$. Then we have a long exact sequence in $\mathrm{Coh}^{\beta,\alpha_{\mathcal{E}}^{\beta}}_{0}(X)$:
	\begin{align*}
		0 \to \mathcal{H}^{-1}(F) \to 0 \to \mathcal{H}^{-1}(G) \to \mathcal{H}^{0}(F) \to i_{*}\mathcal{E}[1] \to \mathcal{H}^{0}(G) \to 0.
	\end{align*}Then $\mathcal{H}^{0}(G)=0$ or $i_{*}\mathcal{E}[1] \cong \mathcal{H}^{0}(G)$ by Proposition \ref{exceptionalissimple1}. If $i_{*}\mathcal{E}[1] \cong \mathcal{H}^{0}(G)$, then $\mathcal{H}^{-1}(G) \cong \mathcal{H}^{0}(F)$ and both of them are $0$. Otherwise, there is a short exact sequence in $\mathrm{Coh}^{\beta,\alpha_{\mathcal{E}}^{\beta}}_{0}(X)$
	\begin{align*}
		0 \to \mathcal{H}^{-1}(G) \to \mathcal{H}^{0}(F) \to i_{*}\mathcal{E}[1] \to 0
	\end{align*} We can take the $\mathrm{Hom}(i_{*}\mathcal{E}[1],-)$ functor on it and we will have $\mathcal{H}^{0}(F)=0$ or $\mathcal{H}^{0}(F) \cong i_{*}\mathcal{E}[1]$, it means that $F=0$ or $G=0$.
\end{pf}

Now, suppose $\mathcal{E}$ is an exceptional vector bundle on $\mathbb{P}^{3}$ with rank $r$ and $x \in \mathbb{P}^{3}$. We also write $\mathcal{E}^{x}$ for the kernel of the natural map $i_{*}\mathcal{E}^{\oplus r} \to \mathbb{C}(x)$.

\begin{prop}\label{exceptionalissimple3}
	Let $\beta \in \mathbb{R}$ and $\mathcal{E}$ be an exceptional locally free sheaf with finite rank on $\mathbb{P}^{3}$. Assume that $\beta<\mu_{1}(\mathcal{E})$ and $i_{*}\mathcal{E}$ is $\nu^{\beta,\alpha_{\mathcal{E}}^{\beta}}$-stable. Denote the kernel of the natural map $i_{*}\mathcal{E}^{\oplus r} \to \mathbb{C}(x)$ by $\mathcal{E}^{x}$. Then the object $\mathcal{E}^{x}[1] \in \mathcal{I}_{0}^{\beta,i_{*}\mathcal{E}}(X) \subset \mathrm{Coh}^{\beta,i_{*}\mathcal{E}}_{0}(X)$ is a simple object.
\end{prop}

\begin{pf}
	First, $\mathcal{E}^{x}$ is a slope semistable sheaf with the slope $\mu(\mathcal{E})$, thus $\mathcal{E}^{x} \in \mathrm{Coh}^{\beta}_{0}(X)$ for $\beta<\mu(\mathcal{E})$. Then 
	\begin{align*}
		0 \to \mathcal{E}^{x} \to i_{*}\mathcal{E}^{\oplus r} \to \mathbb{C}(x) \to 0
	\end{align*} is also an exact sequence in $\mathrm{Coh}^{\beta}_{0}(X)$. Thus, $\mathcal{E}^{x}$ is also a $\nu^{\beta,\alpha_{\mathcal{E}}^{\beta}}$-semistable object of the same tilt-slope as $i_{*}\mathcal{E}$, and it is $\beta$.
	
	Take the $\mathrm{Hom}$ functor $\mathrm{Hom}(i_{*}\mathcal{E},-)$ on the short exact sequence above, and we have 
	\begin{align*}
		\mathrm{Hom}(i_{*}\mathcal{E},i_{*}\mathcal{E}^{\oplus r}) \cong \mathrm{Hom}(i_{*}\mathcal{E},\mathbb{C}(x))
	\end{align*} from the definition of the natural map, then we know $\mathrm{Hom}(i_{*}\mathcal{E},\mathcal{E}^{x})=0$. It means $\mathcal{E}^{x} \in \mathcal{F}^{\beta,i_{*}\mathcal{E}}_{0}(X)$ and $\mathcal{E}^{x}[1] \in \mathcal{I}_{0}^{\beta,i_{*}\mathcal{E}}(X) \subseteq \mathrm{Coh}^{\beta,i_{*}\mathcal{E}}_{0}(X)$. We claim that  there are no non-trivial short exact sequences $0 \to F \to \mathcal{E}^{x}[1] \to G \to 0$ in the abelian category $\mathcal{I}_{0}^{\beta,i_{*}\mathcal{E}}(X)$.
	
	Notice that we have a short exact sequence in $\mathcal{I}_{0}^{\beta,i_{*}\mathcal{E}}(X)$:
	\begin{align*}
		0 \to i_{*}\mathcal{E}^{\oplus r} \to \mathbb{C}(x) \to \mathcal{E}^{x}[1] \to 0.
	\end{align*}
	
	Let $K$ be the kernel of the composition $\mathbb{C}(x) \to \mathcal{E}^{x}[1] \to G$ in $\mathcal{I}_{0}^{\beta,i_{*}\mathcal{E}}(X)$ and we take the cohomology functor $\mathcal{H}^{i}_{\beta}$ with respect to the heart $\mathrm{Coh}_{0}^{\beta}(X)$:
	\begin{align*}
		0 \to \mathcal{H}^{-1}_{\beta}(K) \to 0 \to \mathcal{H}^{-1}_{\beta}(G) \to \mathcal{H}^{0}_{\beta}(K) \to \mathbb{C}(x) \to \mathcal{H}^{0}_{\beta}(G) \to 0
	\end{align*}
	
	We can divide it into two short exact sequences in $\mathrm{Coh}_{0}^{\beta}(X)$:
	\begin{align*}
		& 0 \to \mathcal{H}^{-1}_{\beta}(G) \to \mathcal{H}^{0}_{\beta}(K) \to T \to 0;\\
		& 0 \to T \to \mathbb{C}(x) \to \mathcal{H}^{0}_{\beta}(G) \to 0.
	\end{align*}
	
	In fact, we always have a decomposition of $T$ in $\mathrm{Coh}_{0}^{\beta}(X)$ as follows:
	\begin{align*}
		0 \to T' \to T \to i_{*}\mathcal{E}^{\oplus r'} \to 0,
	\end{align*}where $r'$ is a nonnegative integer and $T' \in \mathcal{T}^{\beta,\alpha_{\mathcal{E}}^{\beta}}_{0}(X)$. Since $v^{\beta}_{2}(\mathcal{H}^{0}_{\beta}(G))=(\alpha-\displaystyle\frac{1}{2}\beta^{2})v^{\beta}_{0}(\mathcal{H}^{0}_{\beta}(G))$, then we will have $v^{\beta}_{2}(T')=(\alpha-\displaystyle\frac{1}{2}\beta^{2})v^{\beta}_{0}(T')$ and $T'$ is a zero diemensional sheaf or zero.
	
	Now we have an injection $T' \to \mathbb{C}(x)$ in $\mathrm{Coh}_{0}^{\beta}(X)$, and it means $T'=0$ or $T' \cong T \cong \mathbb{C}(x)$.
	
	If $T'=0$, then we have a short exact sequence in $\mathrm{Coh}_{0}^{\beta}(X)$ as follows:
	\begin{align*}
		0 \to i_{*}\mathcal{E}^{\oplus r'} \to \mathbb{C}(x) \to \mathcal{H}^{0}_{\beta}(G) \to 0.
	\end{align*}However, both of them are sheaves and $\mathbb{C}(x)$ is a simple object in $\mathrm{Coh}_{0}(X)$. Then $r'=0$, $T=0$, $\mathcal{H}^{-1}_{\beta}(G) \cong \mathcal{H}^{0}_{\beta}(K)=0$ and $K=0$, it implies that $G \cong \mathbb{C}(x)$ and $G \cong \mathcal{E}^{x}[1]$.
	
	If $T \cong \mathbb{C}(x)$, then $\mathcal{H}^{0}_{\beta}(G)=0$ and we have a short exact sequence in $\mathrm{Coh}_{0}^{\beta}(X)$:
	\begin{align*}
		0 \to \mathcal{H}^{-1}_{\beta}(G) \to \mathcal{H}^{0}_{\beta}(K) \to \mathbb{C}(x) \to 0.
	\end{align*}
	
	We also have a decomposition of $\mathcal{H}^{0}_{\beta}(K)$ in $\mathrm{Coh}_{0}^{\beta}(X)$ as follows:
	\begin{align*}
		0 \to T'' \to \mathcal{H}^{0}_{\beta}(K) \to i_{*}\mathcal{E}^{\oplus r''} \to 0,
	\end{align*}where $r''$ is a nonnegative integer and $T'' \in \mathcal{T}^{\beta,\alpha}_{0}(X)$ is a zero dimensional sheaf or zero. Note that $\mathcal{H}^{0}_{\beta}(K)$ is a sheaf and then $\mathcal{H}^{-1}_{\beta}(G)$ is also a sheaf.
	
	We consider the pullback of $\mathcal{H}^{-1}_{\beta}(G) \to \mathcal{H}^{0}_{\beta}(K)$ and $T'' \to \mathcal{H}^{0}_{\beta}(K)$ in $\mathrm{Coh}_{0}(X)$ and it will be zero. Then we will have a commutative diagram in $\mathrm{Coh}_{0}(X)$ as follows:
	\begin{align*}
		\xymatrix{
			0 \ar[r] & 0 \ar[d] \ar[r] & T'' \ar@{=}[r] \ar[d] & T'' \ar[r] \ar[d] & 0\\
			0 \ar[r] & \mathcal{H}^{-1}_{\beta}(G) \ar[r] & \mathcal{H}^{0}_{\beta}(K) \ar[r] & \mathbb{C}(x) \ar[r] & 0.
		}
	\end{align*} 
	Notice that $T'' \to \mathbb{C}(x)$ should be an injection in $\mathrm{Coh}_{0}(X)$ by the property of pullback. Then $T''=0$ or $T'' \cong \mathbb{C}(x)$. Moreover, by the snake lemma, we have a short exact sequence 
	$0 \to \mathcal{H}^{-1}_{\beta}(G) \to i_{*}\mathcal{E}^{\oplus r''} \to \mathrm{coker}(T'' \to \mathbb{C}(x)) \to 0$ in $\mathrm{Coh}_{0}(X)$.
	
	If $T'' \cong \mathbb{C}(x)$, then $\mathcal{H}^{-1}_{\beta}(G) \cong i_{*}\mathcal{E}^{\oplus r''}$ and it implies that $\mathcal{H}^{-1}_{\beta}(G)=0$ and $r''=0$. In this case, $G=0$.
	
	If $T''=0$ and we assume that $r''>0$, by taking $\mathrm{Hom}(i_{*}\mathcal{E},-)$ functor on
	\begin{align*}
		0 \to \mathcal{H}^{-1}_{\beta}(G) \to i_{*}\mathcal{E}^{\oplus r''} \to \mathbb{C}(x) \to 0,
	\end{align*}
	we could know that $r'' \leqslant r$. On the other hand, we have a commutative diagram in $\mathcal{I}_{0}^{\beta,i_{*}\mathcal{E}}(X)$ as follows:
	\begin{align*}
		\xymatrix{
			0 \ar[r] & i_{*}\mathcal{E}^{\oplus r} \ar[d] \ar[r] & \mathbb{C}(x) \ar[r] \ar[d] & \mathcal{E}^{x}[1] \ar[r] \ar[d] & 0\\
			0 \ar[r] & 0 \ar[r] & \mathcal{H}^{-1}_{\beta}(G)[1] \ar@{=}[r] & \mathcal{H}^{-1}_{\beta}(G)[1] \ar[r] & 0.
		}
	\end{align*}By the snake lemma, we have a short exact sequence in $\mathcal{I}_{0}^{\beta,i_{*}\mathcal{E}}(X)$:
	\begin{align*}
	0 \to i_{*}\mathcal{E}^{\oplus r} \to i_{*}\mathcal{E}^{\oplus r''} \to F \to 0.
	\end{align*}It implies that $r''=r$ and thus $F=0$. In conclusion, $\mathcal{E}^{x}[1] \in \mathcal{I}_{0}^{\beta,i_{*}\mathcal{E}}$ is a simple object.
\end{pf}

\begin{lem}\cite[Lemma 5.9]{bayer2011space}
	Let $\mathcal{D}$ be a triangulated category and $v \colon K(\mathcal{D}) \to \lambda$ be a group homomorphism and $E \in \mathcal{D}$. Suppose that $\sigma \in \mathrm{Stab}_{\Lambda}(\mathcal{D})$ is a stability condition such that $E$ is $\sigma$-semistable and if:
	
	(1) there is a Jordan-H\"older filtration 
	\begin{align*}
		0 \to F^{\oplus r} \to E \to G \to 0
	\end{align*} of $E$ such that $F,G$ are $\sigma$-stable;
	
	(2) $\mathrm{Hom}(E,F)=0$;
	
	(3) $v(E)$ and $v(F)$ are linearly independant in $\Lambda$.
	
	Then $\sigma$ is in the closure of the set of stability conditions where $E$ is stable.
\end{lem}

\begin{cor}\label{boundary}
	Let the hypotheses be as in Proposition \ref{generalproposition}. Then the stability condition $\sigma^{\beta,i_{*}\mathcal{E},a}=(Z^{\beta,\alpha_{\mathcal{E}}^{\beta},a},\mathrm{Coh}^{\beta,i_{*}\mathcal{E}}_{0}(X))$ is on the boundary of the set of geometric stability conditions. In particular, if $-\displaystyle\frac{1}{2}<\beta<0$ is a rational number and $\displaystyle\frac{3\beta^{3}+6\beta^{2}-4}{6(3\beta+2)}<a<\displaystyle\frac{1}{6}\beta^{2}$ is a real number, then $(Z^{\beta,\beta^{2},a},\mathrm{Coh}^{\beta,i_{*}\mathcal{O}}_{0}(X)) \in \partial\mathrm{Stab}_{H}^{\mathrm{geo}}(D_{0}^{b}(X))$.
\end{cor}

\begin{pf}
	Notice that we have a short exact sequence in $\mathcal{I}_{0}^{\beta,i_{*}\mathcal{E}}(X)$:
	\begin{align*}
		0 \to i_{*}\mathcal{E}^{\oplus r} \to \mathbb{C}(x) \to \mathcal{E}^{x}[1] \to 0.
	\end{align*}And $i_{*}\mathcal{E},\mathcal{E}^{x}[1]$ are both $\sigma^{\beta,i_{*}\mathcal{E},a}$-stable by Proposition \ref{exceptionalissimple2} and \ref{exceptionalissimple3}. Therefore, by Lemma \cite[Lemma 5.9]{bayer2011space} above, we get our conclusion.
\end{pf}

Finally, we have:
\begin{prop}
	If $-\displaystyle\frac{1}{2}<\beta<0$ is a rational number and $\displaystyle\frac{3\beta^{3}+6\beta^{2}-4}{6(3\beta+2)}<a<\displaystyle\frac{1}{6}\beta^{2}$ is a real number, then $\sigma^{\beta,i_{*}\mathcal{O},a}=(Z^{\beta,\beta^{2},a},\mathrm{Coh}^{\beta,i_{*}\mathcal{O}}_{0}(X)) \in \overline{\mathrm{Stab}_{H}^{\mathrm{geo}}(D_{0}^{b}(X))} \cap \overline{\mathrm{ST}_{i_{*}\mathcal{O}}(\mathrm{Stab}_{H}^{\mathrm{geo}}(D_{0}^{b}(X))})$.
\end{prop}

\begin{pf}
Note that\begin{align*}
	\mathrm{ST}_{i_{*}\mathcal{O}}(\mathbb{C}(x)) \cong \mathrm{Cone}(\mathrm{RHom}_{D_{0}^{b}(X)}(i_{*}\mathcal{O},\mathbb{C}(x)) \otimes i_{*}\mathcal{O} \to \mathbb{C}(x))
\end{align*}

But $\mathrm{RHom}_{D_{0}^{b}(X)}(i_{*}\mathcal{O},\mathbb{C}(x)) \cong \bigoplus_{l \in \mathbb{Z}}\mathrm{Hom}_{D_{0}^{b}(X)}(i_{*}\mathcal{O},\mathbb{C}(x)[l])[-l] \cong\mathbb{C} \oplus \mathbb{C}[-1]$ and we have an exact sequence in $\mathrm{Coh}_{0}(X)$:
\begin{align*}
	&0 \to \mathcal{H}^{-1}(\mathrm{ST}_{i_{*}\mathcal{O}}(\mathbb{C}(x))) \to \mathrm{Hom}_{\mathrm{Coh}_{0}(X)}(i_{*}\mathcal{O},\mathbb{C}(x)) \otimes i_{*}\mathcal{O} \to \mathbb{C}(x) \to \mathcal{H}^{0}(\mathrm{ST}_{i_{*}\mathcal{O}}(\mathbb{C}(x))) \to \\ &\mathrm{Hom}_{\mathrm{Coh}_{0}(X)}(i_{*}\mathcal{O},\mathbb{C}(x)[1]) \otimes i_{*}\mathcal{O} \to 0.
\end{align*}

Thus $\mathcal{H}^{-1}(\mathrm{ST}_{i_{*}\mathcal{O}}(\mathbb{C}(x))) \cong \mathcal{O}^{x}$ and $\mathcal{H}^{0}(\mathrm{ST}_{i_{*}\mathcal{O}}(\mathbb{C}(x))) \cong \mathrm{Hom}_{\mathrm{Coh}_{0}(X)}(i_{*}\mathcal{O},\mathbb{C}(x)[1]) \otimes i_{*}\mathcal{O} \to 0 \cong i_{*}\mathcal{O}$ by the definition of $\mathcal{O}^{x}$. Then we have a distinguished triangle:
\begin{align*}
	\mathcal{O}^{x}[1] \to \mathrm{ST}_{i_{*}\mathcal{O}}(\mathbb{C}(x)) \to i_{*}\mathcal{O} \to \mathcal{O}^{x}[2].
\end{align*}
Take the inverse of $\mathrm{ST}_{i_{*}\mathcal{O}}$ on the above triangle:
\begin{align*}
	\mathrm{ST}_{i_{*}\mathcal{O}}^{-1}(\mathcal{O}^{x}[1]) \to \mathbb{C}(x) \to \mathrm{ST}_{i_{*}\mathcal{O}}^{-1}(i_{*}\mathcal{O}) \to \mathrm{ST}_{i_{*}\mathcal{O}}^{-1}(\mathcal{O}^{x}[2]).
\end{align*}
Note that: $\mathrm{ST}_{i_{*}\mathcal{O}}(i_{*}\mathcal{O})=i_{*}\mathcal{O}[-3]$, thus:
\begin{align*}
\mathrm{ST}_{i_{*}\mathcal{O}}^{-1}(\mathcal{O}^{x}[1]) \to \mathbb{C}(x) \to i_{*}\mathcal{O}[3] \to \mathrm{ST}_{i_{*}\mathcal{O}}^{-1}(\mathcal{O}^{x}[2]).
\end{align*}
Let $\eta^{\beta,i_{*}\mathcal{O},a}=\mathrm{ST}_{i_{*}\mathcal{O}}^{-1}\cdot\sigma^{\beta,i_{*}\mathcal{O},a} \in \mathrm{Stab}_{H}(D_{0}^{b}(X))$, then $\mathrm{ST}_{i_{*}\mathcal{O}}^{-1}(\mathcal{O}^{x}[1])$ and $i_{*}\mathcal{O}[3]$ are  $\eta^{\beta,i_{*}\mathcal{O},a}$-stable object. By \cite[Lemma 5.9]{bayer2011space}, we have $\eta^{\beta,i_{*}\mathcal{O},a} \in \partial\mathrm{Stab}_{H}^{\mathrm{geo}}(D_{0}^{b}(X))$, thus $\sigma^{\beta,i_{*}\mathcal{O},a} \in \overline{\mathrm{Stab}_{H}^{\mathrm{geo}}(D_{0}^{b}(X))} \cap \overline{\mathrm{ST}_{i_{*}\mathcal{O}}(\mathrm{Stab}_{H}^{\mathrm{geo}}(D_{0}^{b}(X))})$.
\end{pf}

\appendix

\section{Relative derived dual functor}\label{appendix2}

In this appendix, we introduce the notion and some properties of relative derived dual functor.

\subsection{Basic definitions and properties}

First, for any smooth quasi-projective variety $W$ over  $\mathbb{C}$ with dimension $n$, we define its \textit{derived dual functor} as usual: 
\begin{align*}
	\mathbb{D}_{W} \triangleq R\mathcal{H}om(-,\mathcal{O}_{W}) \colon D^{b}(W) \to D^{b}(W).
\end{align*} Note that we have the following basic proposition of this functor and we can find the proof in  \cite[Section1.1]{huybrechts2010geometry} written by Huybrechts and Lehn:

(1) We have $\mathbb{D}_{W} \circ \mathbb{D}_{W} \cong 1_{D^{b}(W)}$, where $1_{D^{b}(W)}$ means the identity functor of $D^{b}(W)$.

(2) Let $E$ be a nonzero coherent sheaf on $W$ and the dimension of $E$ is $d$, then we have:
\begin{align*}
	\mathcal{H}^{l}(\mathbb{D}_{W}(E))=\mathcal{E}xt^{l}_{W}(E,\mathcal{O}_{W})=0 
\end{align*} unless $l \in \{n-d,\ldots,n\}$. And the dimension of the sheaf $\mathcal{H}^{l}(\mathbb{D}_{W}(E))$ is less than or equal to $n-l$ when $l \in \{n-d,\ldots,n\}$. 

Moreover, if $E$ is a pure sheaf of dimension $d$, then:

(1) For any $l \notin \{n-d,\ldots,n-1\}$, we have $\mathcal{H}^{l}(\mathbb{D}_{W}(E))=0$.

(2) The dimension of sheaf $\mathcal{H}^{l}(\mathbb{D}_{W}(E))$ is less than or equal to $n-l-1$ when $l \in \{n-d+1,\ldots,n-1\}$, and the sheaf $\mathcal{H}^{n-d}(\mathbb{D}_{W}(E))$ is pure of dimension $d$ with the same support set of $E$. Moreover, if $E$ is just a nonzero coherent sheaf on $W$ with dimension $d$, then $\mathcal{H}^{n-d}(E)$ is still pure of dimension $d$.

(3) If $W$ is projective, then we have:
\begin{align*}
	\mathrm{ch}_{i}(\mathbb{D}_{W}(E))=(-1)^{i}\mathrm{ch}_{i}(E)
\end{align*} for any $i \in \{0,\ldots,n\}$ and $E \in \mathrm{Coh}(W)$.

Now we introduce the notation of relative derived dual functor. We still use the notion we have used in the main body of this paper. Explicitly, suppose that $Y$ is a smooth projective variety over  $\mathbb{C}$ with dimension $n$ and $X=\mathrm{Tot}_{Y}(\mathcal{E}_{0})$ is the total space of a locally free sheaf $\mathcal{E}_{0}$ with constant rank $r$ on $Y$. Let $\pi \colon X \to Y$ be the natural projection and $i \colon Y \to X$ be the zero section of $\pi$. We always fix a polarization $H$ of $Y$ and use $v_{i}(E)$ to represent the number
$H^{n-i}\mathrm{ch}_{i}(\pi_{*}E)$ for any $E \in D_{0}^{b}(X)$.

We define the relative derived dual functor with respect to $\pi \colon X \to Y$ as follows:
\begin{align*}
	\mathbb{D}_{X/Y} \triangleq R\mathcal{H}om(-,\pi^{*}\mathrm{det}\mathcal{E}_{0}^{\lor}[r]) \colon D^{b}_{0}(X) \to D^{b}_{0}(X).
\end{align*} One can quickly check that $\mathbb{D}_{X/Y} \circ \mathbb{D}_{X/Y} \cong 1_{D^{b}_{0}(X)}$ and $\mathbb{D}_{X/Y}$ is the composition of $\mathbb{D}_{X}$ and tensoring with the object $\pi^{*}\mathrm{det}\mathcal{E}_{0}^{\lor}[r]$.

\begin{prop}
	For any $E \in D_{0}^{b}(X)$, we have 
	\begin{align*}
		\pi_{*}\mathbb{D}_{X/Y}(E) \cong \mathbb{D}_{Y}(\pi_{*}E).
	\end{align*} in $D^{b}(Y)$. Thus, for any $i \in \{0,\ldots,n\}$, $v_{i}(\mathbb{D}_{X/Y}(E))=(-1)^{i}v_{i}(E)$.
\end{prop}

\begin{pf}
	We conclude this proposition directly by the Grothendieck-Verdier duality. 
\end{pf}

If $E \in \mathrm{Coh}_{0}(X)$ is a sheaf with dimension $d$, then the sheaf $\mathcal{H}^{n-d}(\mathbb{D}_{X/Y}(E)) \in \mathrm{Coh}_{0}(X)$ is a pure sheaf of dimension $d$. We call the sheaf $\mathcal{H}^{n-d}(\mathbb{D}_{X/Y}(E))$ as the relative dual sheaf of $E$ and still denote it by $E^{\lor}$. We may get confused about whether the notion $E^{\lor}$ represents $\mathcal{H}^{n-d}(\mathbb{D}_{X/Y}(E))$ or $\mathcal{H}^{n+r-d}(\mathbb{D}_{X}(E))$. But in this appendix, without special statement, $E^{\lor}$ always represents $\mathcal{H}^{n-d}(\mathbb{D}_{X/Y}(E))$ where $E \in \mathrm{Coh}_{0}(X)$. 

Before our further discussion, there is a useful lemma about Bridgeland stability condition:

\begin{lem}\label{importantlemmaB1}
	Let $\mathcal{D}$ be a triangulated category and $\sigma=(Z,\mathcal{A})$ be a weak pre-stability condition on $\mathcal{D}$ with slope function $\mu$. Suppose that there is a short exact sequence in $\mathcal{A}$:
	\begin{align*}
		0 \to F \to E \to G \to 0,
	\end{align*}where $F,E \neq 0$ and $Z(G)=0$.
	
	(1) We have $\mu_{\min}(F)=\mu_{\min}(E)$, and if $\mathrm{Hom}(C,E)=0$ for any $C \in \mathcal{A}$ with $Z(C)=0$, then $\mu_{\max}(F)=\mu_{\max}(E)$.
	
	(2) \cite[Sublemma 2.16]{piyaratne2019moduli} If $F$ is $\sigma$-semistable and $\mathrm{Hom}(C,E)=0$ for any $C \in \mathcal{A}$ with $Z(C)=0$, then $E$ is $\sigma$-semistable with slope $\mu(F)$.
\end{lem}

\begin{pf}
	(1) We do induction on the length $m$ of Harder-Narasimhan filtration of $E$.
	
	If $m=1$, it means that $E$ is $\sigma$-semistable. For any nontrivial subobject $F'$ of $F$, it is also a subobject of $E$, then $\mu(F') \leqslant \mu(E)$. Meanwhile, $Z(G)=0$ implies that $\mu(E)=\mu(F)$, then $\mu(F') \leqslant \mu(F)$. Thus, $F$ is $\sigma$-semistable with slope $\mu(E)$.
	
	Assume that our conclusion holds if $m<k$. When $m=k>1$, we consider the Harder-Narasimhan factor $E_{k}$ of $E$ with the minimum slope. Let $f$ be the composition of $F \to E \to E_{k}$, then there will be a commutative diagram in $\mathcal{A}$ with two rows are both exact:
	\begin{align*}
	\xymatrix{
			0 \ar[r] & F \ar[d] \ar[r] & E \ar[r] \ar[d] & G  \ar[r] \ar[d] & 0\\
			0 \ar[r] & \Im f \ar[r] & E_{k} \ar[r] & E_{k}/\Im f \ar[r] & 0.
			}
	\end{align*}
	By the snake lemma, the morpshism $G \to E_{k}/\Im f$ is an epimorphism in $\mathcal{A}$ and then $Z(E_{k}/\Im f)=0$. Note that $\mathrm{Im}f \neq 0$, otherwise, we will have $Z(E_{k})=0$ and it is a contradiction. Thus by our proof in the first paragraph, we have $\Im f$ is also $\sigma$-semistable with $\mu_{\min}(E)=\mu(\mathrm{Im}f)$.
	
	On the other hand, we denote the kernel of $E \to E_{k}$ by $E'$, the kernel of $F \to \Im f$ by $F'$ and the kernel of $G \to E_{n}/\Im f$ by $G'$, then $0 \to F' \to E' \to G' \to 0$ is an exact sequence in $\mathcal{A}$ with $Z(G')=0$ by the snake lemma. If $F'=0$, then $\mu_{\min}(F)=\mu(F)=\mu(\mathrm{Im}f)=\mu_{\min}(E)$. If $F' \neq 0$, then by the assumption of induction, we have $\mu_{\min}(F')=\mu_{\min}(E')>\mu(E_{n})=\mu(\Im f)$, thus $\mu_{\min}(F)=\mu(\Im f)=\mu_{\min}(E)$.
	
	Moreover, if $\mathrm{Hom}(C,E)=0$ for any $C \in \mathcal{A}$ with $Z(C)=0$, $F'$ will not be $0$, because it implies $Z(E')=0$ and it is a contradiction. Meanwhile, we conclude that $\mathrm{Hom}(C,E')=0$ for any $C \in \mathcal{A}$ with $Z(C)=0$ and by our assumption, $\mu_{\max}(F')=\mu_{\max}(E')$.
	Thus $\mu_{\max}(F)=\mu_{\max}(E)$.
	
	(2) It is a direct corollary of (1).
\end{pf}

\begin{prop}\label{propositionB1}
	Suppose that and $E \in \mathrm{Coh}_{0}(X)$ is pure of dimension $n$, then:
	
	(1) We have a short exact sequence in $\mathrm{Coh}_{0}(X)$:
	\begin{align*}
		0 \to E \stackrel{\iota_{E}}{\longrightarrow} E^{\lor\lor} \to E^{\lor\lor}/E \to 0,
	\end{align*}where $E^{\lor\lor}/E$ is $0$ or the dimension of $E^{\lor\lor}/E$ is less than or equal to $n-2$. In other words, $v_{0}(E^{\lor\lor}/E)=v_{1}(E^{\lor\lor}/E)=0$.
	
	(2) If $E$ is slope semistable, then $E^{\lor\lor}$ is also slope semistable with the same slope. 
	
	(3) If $E$ is slope semistable, then $E^{\lor}$ is also slope semistable with slope $-\mu(E)$.
	
	(4) We always have $\mu_{\min}(E^{\lor})=-\mu_{\max}(E)$ and $\mu_{\max}(E^{\lor})=-\mu_{\min}(E)$.
	
	(5) Moreover, suppose that there is a short exact sequence in $\mathrm{Coh}_{0}(X)$:
	\begin{align*}
		0 \to F \stackrel{f}{\longrightarrow} E \to G \to 0
	\end{align*}with $F \neq 0$ and $G$ is $0$ or the dimension of $G$ is less that or equal to $n-2$. Then there is a commutative diagram in $\mathrm{Coh}_{0}(X)$:
	\begin{align*}
		\xymatrix{0 \ar[r] & F \ar^{\iota_{F}}[r] \ar_{f}[d] & F^{\lor\lor} \ar^{f^{\lor\lor}}[d] \\
			0 \ar[r] & E \ar_{\iota_{E}}[r] & E^{\lor\lor}
		}
	\end{align*}with $f^{\lor\lor}$ is an isomorphism.
\end{prop}

\begin{pf}
	(1) Because $\mathcal{H}^{l}(\mathbb{D}_{X/Y}(E))=0$ for any $l<0$, then we have a natural morphism: $\phi \colon E^{\lor} \to \mathbb{D}_{X/Y}(E)$, such that $\mathcal{H}^{0}(\phi)$ is an isomorphism. Consider $E'$ as the cone of this morphism, then we have a distinguished triangle in $D^{b}_{0}(X)$:
	\begin{align}
	\label{dualsequence1}
		E^{\lor} \to \mathbb{D}_{X/Y}(E) \to E' \to E^{\lor}[1].
	\end{align}After taking cohomology and note that $\mathcal{H}^{n}(\mathbb{D}_{X/Y}(E))=0$, we have:
	\begin{align*}
		\mathcal{H}^{l}(E') \cong \begin{cases}
			0 & l \leqslant 0 \mbox{ or } l \geqslant n; \\
		\mathcal{H}^{l}(\mathbb{D}_{X/Y}(E)) & 0<l<n.
		\end{cases}
	\end{align*}
	Thus, when $0<l<n$, then $\dim \mathcal{H}^{l}(E')=\dim \mathcal{H}^{l}(\mathbb{D}_{X/Y}(E))=\dim \mathcal{H}^{r+l}(\mathbb{D}_{X}(E)) \leqslant n-l-1$.
	
	We take the relatived dual functor of the sequence (\ref{dualsequence1}) and we have a distinguished triangle in $D_{0}^{b}(X)$ as follows:
	\begin{align*}
		\mathbb{D}_{X/Y}(E') \to E \to \mathbb{D}_{X/Y}(E^{\lor}) \to \mathbb{D}_{X/Y}(E')[1].
	\end{align*}After taking cohomology with respect to the heart $\mathrm{Coh}_{0}(X)$, we have an exact sequence in $\mathrm{Coh}_{0}(X)$:
	\begin{align*}
		\mathcal{H}^{0}(\mathbb{D}_{X/Y}(E')) \to E \to E^{\lor\lor} \to \mathcal{H}^{1}(\mathbb{D}_{X/Y}(E')) \to 0
	\end{align*}Finally, we just need to claim that $\mathcal{H}^{0}(\mathbb{D}_{X/Y}(E'))=0$ and $\dim \mathcal{H}^{1}(\mathbb{D}_{X/Y}(E')) \leqslant n-2$.
	
	If $n=1$, then we find that $E'=0$ and then $E \cong E^{\lor\lor}$. If $n \geqslant 2$, we consider a filtration of $E'$ as follows:
	\begin{align*}
		\xymatrix{
		0=E_{0}' \ar[r]  & E_{1}' \ar[r]\ar[d]  & \cdots \ar[r]  & E_{n-2}' \ar[r]  & E_{n-1}'=E'\ar[d]\\
		& \mathcal{H}^{1}(E')[-1]\ar@{-->}[lu]  & & & \mathcal{H}^{n-1}(E')[1-n]\ar@{-->}[lu]
		}
	\end{align*}In other words, there is a series of distinguished triangle here: $E_{l-1}' \to E_{l}' \to \mathcal{H}^{l}(E')[-l] \to E_{l-1}[1]$ for any $l \in \{1,\ldots,n-1\}$. After taking the relatived dual functor, we have a series of distinguished triangle in $D_{0}^{b}(X)$: $\mathbb{D}_{X/Y}(\mathcal{H}^{l}(E'))[l] \to \mathbb{D}_{X/Y}(E_{l}') \to \mathbb{D}_{X/Y}(E_{l-1}') \to \mathbb{D}_{X/Y}(\mathcal{H}^{l}(E'))[l+1]$. After taking cohomology with respect to the heart $\mathrm{Coh}_{0}(X)$, we have an exact sequence in $\mathrm{Coh}_{0}(X)$:
	\begin{align*}
	 \cdots	\to & \mathcal{H}^{l}(\mathbb{D}_{X/Y}(\mathcal{H}^{l}(E'))) \to \mathcal{H}^{0}(\mathbb{D}_{X/Y}(E_{l}')) \to \mathcal{H}^{0}(\mathbb{D}_{X/Y}(E_{l-1}')) \to \\ & \mathcal{H}^{l+1}(\mathbb{D}_{X/Y}(\mathcal{H}^{l}(E'))) \to \mathcal{H}^{1}(\mathbb{D}_{X/Y}(E_{l}')) \to \mathcal{H}^{1}(\mathbb{D}_{X/Y}(E_{l-1}')) \to \cdots.
	\end{align*}Note that we only need to prove that: $\mathcal{H}^{l}(\mathbb{D}_{X/Y}(\mathcal{H}^{l}(E')))=0$ and $\dim \mathcal{H}^{l+1}(\mathbb{D}_{X/Y}(\mathcal{H}^{l}(E'))) \leqslant n-2$ for any $l \in \{1,\ldots,n-1\}$.
	
	In fact, $\dim \mathcal{H}^{l}(E') \leqslant n-1-l$, then $l< n-\dim \mathcal{H}^{l}(E')$. It means $\mathcal{H}^{l}(\mathbb{D}_{X/Y}(\mathcal{H}^{l}(E')))=0$. Meanwhile, if $\dim \mathcal{H}^{l}(E')<n-1-l$, then $\mathcal{H}^{l+1}(\mathbb{D}_{X/Y}(\mathcal{H}^{l}(E')))=0$ for the same reason. If $\dim \mathcal{H}^{l}(E')=n-1-l$, then $\dim \mathcal{H}^{l+1}(\mathbb{D}_{X/Y}(\mathcal{H}^{l}(E'))) \leqslant \dim \mathcal{H}^{l}(E')=n-1-l \leqslant n-2$. And we finish the proof.
	
	(2) It is a direct conclusion of Lemma \ref{importantlemmaB1} and (1).
	
	(3) First, note that $(-1)^{i}v_{i}(E)=v_{i}(\mathbb{D}_{X/Y}(E))=v_{i}(E^{\lor})+v_{i}(E')$. And when $i=0$ or $1$, we have $v_{i}(E')=\displaystyle\sum_{l=1}^{n-1}(-1)^{l}v_{i}(\mathcal{H}^{l}(E'))=0$ because $\dim \mathcal{H}^{l}(E') \leqslant n-1-l \leqslant n-2$ for any $l \in \{1,\ldots,n-1\}$. Thus, $v_{0}(E^{\lor})=v_{0}(E)$, $v_{1}(E^{\lor})=-v_{1}(E)$ and $\mu(E^{\lor})=-\mu(E)$. Moreover, if $H \in \mathrm{Coh}_{0}(X)$ is a nonzero sheaf with dimension $n$, then there is a torsion decomposition in $\mathrm{Coh}_{0}(X)$:
	\begin{align*}
		0 \to H_{t} \to H \to H_{tf} \to 0
	\end{align*}where $H_{tf} \neq 0$ is pure of dimension $n$ and $\dim H_{t} \leqslant n-1$. It means that:
	\begin{align*}
		\mu(H^{\lor})=\mu(H_{tf}^{\lor})=-\mu(H_{tf})=-\frac{v_{1}(H_{tf})}{v_{0}(H_{tf})}=-\frac{v_{1}(H)-v_{1}(H_{t})}{v_{0}(H)} \geqslant -\mu(H).
	\end{align*}
	
	Now, suppose that $E^{\lor}$ is not slope semistable, then there is a short exact sequence in $\mathrm{Coh}_{0}(X)$:
	\begin{align*}
		0 \to A \to E^{\lor} \to B \to 0
	\end{align*}with $\mu(A)>\mu(E^{\lor}) \geqslant \mu(B)$ and $A,B \neq 0$. Because $E^{\lor}$ is pure of dimension $n$, then $A$ is also pure of dimension $n$ and $+\infty>\mu(A)>\mu(E^{\lor})>\mu(B)$. After taking the relative dual functor and the cohomology functor, we have an exact sequence in $\mathrm{Coh}_{0}(X)$:
	\begin{align*}
		0 \to B^{\lor} \to E^{\lor\lor} \to A^{\lor}
	\end{align*}Now $E^{\lor\lor}$ is $\sigma$-semistable by (2) and $B^{\lor}$ is a nonzero subsheaf of $E^{\lor\lor}$. Thus
	\begin{align*}
		-\mu(B) \leqslant \mu(B^{\lor}) \leqslant \mu(E^{\lor\lor})=\mu(E).
	\end{align*}It is a contradiction. 
	
	(4) We do induction on the length $m$ of Harder-Narasimhan filtration of $E$. If $m=1$, then it is the case in (3). Assume that our conclusion holds if $m<k$. When $m=k>1$, we consider the Harder-Narasimhan factor $E_{k}$ of $E$ with the minimum slope and then we have a short exact sequence in $\mathrm{Coh}_{0}(X)$:
	\begin{align*}
		0 \to E' \to E \to E_{k} \to 0
	\end{align*} After taking relative dual functor and the cohomology functor, we have an exact sequence in $\mathrm{Coh}_{0}(X)$:
	\begin{align*}
		0 \to E_{k}^{\lor} \to E^{\lor} \to E'^{\lor} \to \mathcal{H}^{1}(\mathbb{D}_{X/Y}(E_{k})).
	\end{align*}Divide it into two short exact sequences:
	\begin{align*}
		& 0 \to E_{k}^{\lor} \to E^{\lor} \to H \to 0 \\
		& 0 \to H \to E'^{\lor} \to H' \to 0,
	\end{align*}where $H'$ is a subsheaf of $\mathcal{H}^{1}(\mathbb{D}_{X/Y}(E_{k}))$. 
	
	Note that $\dim H' \leqslant n-2$ and by Lemma \ref{importantlemmaB1}, we have $\mu_{\min}(H)=\mu_{\min}(E'^{\lor})$ and $\mu_{\max}(H)=\mu_{\max}(E'^{\lor})$. By our assumption of  induction, there are $\mu_{\min}(E'^{\lor})=-\mu_{\max}(E')$ and $\mu_{\max}(E'^{\lor})=-\mu_{\min}(E')$. Meanwhile, we have $\mu(E_{k}^{\lor})=-\mu(E_{k})>-\mu_{\min}(E')=\mu_{\max}(H)$. It means that $\mu_{\max}(E^{\lor})=\mu(E_{k}^{\lor})=-\mu(E_{k})=-\mu_{\min}(E)$ and $\mu_{\min}(E^{\lor})=\mu_{\min}(H)=\mu_{\min}(E'^{\lor})=-\mu_{\max}(E')=-\mu_{\max}(E)$. Thus we get our conclusion.
	
	(5) It is not hard to see that $\iota_{E} \circ f=f^{\lor\lor} \circ \iota_{F}$. By taking the functor $\mathcal{H}^{0}(\mathbb{D}_{X/Y}(-))$, we know that $\mathcal{H}^{0}(\mathbb{D}_{X/Y}(G))=0$. It implies $f^{\lor}$ is an isomorphism and $f^{\lor\lor}$ is also an isomorphism. 
\end{pf}

\begin{rmk}
	For any pure sheaf $E$, if the morphism $E \to E^{\lor\lor}$ which we construct in the above proposition is an isomorphism, then we call that $E$ is reflexive. In fact, if $E$ is reflexive with dimension $d$, then the dimension of the sheaf $\mathcal{H}^{l}(\mathbb{D}_{X/Y}(E))$ is less than or equal to $n-l-2$ when $l \in \{n-d+1,\ldots,n-2\}$. The proof is the same as the proof in \cite{huybrechts2010geometry}.
\end{rmk}

In fact, there is a similar proposition in the first tilting herat $\mathrm{Coh}^{\beta}_{0}(X)$ for any $\beta \in \mathbb{R}$ when $n=2,3$.

\begin{defn}
	Let $\beta \in \mathbb{R}$ and $0 \neq E \in \mathrm{Coh}^{\beta}_{0}(X)$. If $v_{1}^{\beta}(H)>0$ for any nonzero subobject $H$ of $E$ in $\mathrm{Coh}^{\beta}_{0}(X)$, then we call $E$ is tilt-torsionfree. 
\end{defn}

\begin{lem}\label{importantlemmaB2}
	Let $\beta \in \mathbb{R}$ and $0 \neq E \in \mathrm{Coh}^{\beta}_{0}(X)$, then for any $l<0$ and $l>n$, we have $\mathcal{H}^{l}(\mathbb{D}_{X/Y}(E))=0$. Moreover, for any $l \in \{2,\ldots,n\}$, we have $\dim \mathcal{H}^{l}(\mathbb{D}_{X/Y}(E)) \leqslant n-l$.
\end{lem}

\begin{pf}
	We denote $\mathcal{H}^{-1}(E)$ by $F$ and $\mathcal{H}^{0}(E)$ by $T$. Note that if $F \neq 0$ then it must be a pure sheaf of dimension $n$.
	
	We have a short exact sequence in $\mathrm{Coh}_{0}^{\beta}(X)$ as follows:
	\begin{align*}
		0 \to F[1] \to E \to T \to 0.
	\end{align*} Then we have a distinguished triangle in $D_{0}^{b}(X)$ after taking the relative dual functor:
	\begin{align}
		\label{dualsequence2}	\mathbb{D}_{X/Y}(T) \to \mathbb{D}_{X/Y}(E) \to \mathbb{D}_{X/Y}(F)[-1] \to \mathbb{D}_{X/Y}(T)[1].
	\end{align}
	
	By taking cohomology functor on the above distinguished triangle (\ref{dualsequence2}) with respect to the heart $\mathrm{Coh}_{0}(X)$, we have a long exact sequence:
	\begin{align*}
		\cdots \to \mathcal{H}^{l}(\mathbb{D}_{X/Y}(T)) \to \mathcal{H}^{l}(\mathbb{D}_{X/Y}(E)) \to \mathcal{H}^{l-1}(\mathbb{D}_{X/Y}(F)) \to  \cdots
	\end{align*}Now, if $l<0$ or $l>n$, then $\mathcal{H}^{l}(\mathbb{D}_{X/Y}(T))=0$ and $\mathcal{H}^{l-1}(\mathbb{D}_{X/Y}(F))=0$; Moreover, if $l \in \{2,\ldots,n\}$, then $\dim \mathcal{H}^{l}(\mathbb{D}_{X/Y}(T)) \leqslant n-l$ and $\dim \mathcal{H}^{l-1}(\mathbb{D}_{X/Y}(F)) \leqslant n-l$. And then $\dim \mathcal{H}^{l-1}(\mathbb{D}_{X/Y}(E)) \leqslant n-l$ 
\end{pf}

\begin{lem}\label{importantlemmaB3}
	Let $\beta \in \mathbb{R}$ and suppose that $0 \neq E \in \mathrm{Coh}^{\beta}_{0}(X)$ is tilt-torsionfree, then:
	
	(1) The sheaf $\mathcal{H}^{0}(\mathbb{D}_{X/Y}(E)) \in \mathrm{Coh}^{<-\beta}_{0}(X)$ and $\mathcal{H}^{1}(\mathbb{D}_{X/Y}(E)) \in  \mathrm{Coh}^{>-\beta}_{0}(X)$.
	
	(2) The sheaf $\mathcal{H}^{n}(\mathbb{D}_{X/Y}(E))=0$.
	
	(3) When $n \geqslant 3$, we have $\dim \mathcal{H}^{n-1}(\mathbb{D}_{X/Y}(E))=0$.
	
\end{lem}

\begin{pf}

We denote $\mathcal{H}^{-1}(E)$ by $F$ and $\mathcal{H}^{0}(E)$ by $T$.

(1) We still take the cohomology functor on (\ref{dualsequence2}). Then we get $\mathcal{H}^{0}(\mathbb{D}_{X/Y}(T)) \cong \mathcal{H}^{0}(\mathbb{D}_{X/Y}(E))$ and an exact sequence in $\mathrm{Coh}_{0}(X)$:
\begin{align}
	\label{dualsequence3}
	0 \to \mathcal{H}^{1}(\mathbb{D}_{X/Y}(T)) \to \mathcal{H}^{1}(\mathbb{D}_{X/Y}(E)) \to \mathcal{H}^{0}(\mathbb{D}_{X/Y}(F)) \to \mathcal{H}^{2}(\mathbb{D}_{X/Y}(T)).
\end{align}

If $T=0$ or $\dim T<n$, then $\mathcal{H}^{0}(\mathbb{D}_{X/Y}(T))=0$ and it means that $\mathcal{H}^{0}(\mathbb{D}_{X/Y}(E)) \in \mathrm{Coh}^{\leqslant -\beta}_{0}(X)$. Otherwise, if $\dim T=n$, then there is a short exact sequence of $\mathrm{Coh}_{0}(X)$:
\begin{align*}
	0 \to T_{t} \to T \to T_{tf} \to 0,
\end{align*}where $T_{tf} \neq 0$ is a pure sheaf of dimension $n$ and $\dim T_{t} \leqslant n-1$. Then $\mu_{\min}(T_{tf})=\mu_{\min}(T)>\beta$ and by Proposition \ref{propositionB1}, we have:
\begin{align*}
	\mu_{\max}(\mathcal{H}^{0}(\mathbb{D}_{X/Y}(E)))=\mu_{\max}(\mathcal{H}^{0}(\mathbb{D}_{X/Y}(T)))=\mu_{\max}(T_{tf}^{\lor})=-\mu_{\min}(T_{tf})<-\beta.
\end{align*}

On the other hand, we split the exact sequence (\ref{dualsequence3}) into two short exact sequneces as follows:
\begin{align*}
	& 0 \to \mathcal{H}^{1}(\mathbb{D}_{X/Y}(T)) \to \mathcal{H}^{1}(\mathbb{D}_{X/Y}(E)) \to H_{1} \to 0. \\
	& 0 \to H_{1} \to \mathcal{H}^{0}(\mathbb{D}_{X/Y}(F)) \to H_{2} \to 0,
\end{align*}where $H_{2}$ is a subsheaf of $\mathcal{H}^{2}(\mathbb{D}_{X/Y}(T))$ and $\dim H_{2} \leqslant n-2$ if it is nonzero. By Lemma \ref{importantlemmaB1}, we have:
\begin{align*}
	\mu_{\min}(H_{1})=\mu_{\min}(\mathcal{H}^{0}(\mathbb{D}_{X/Y}(F)))=-\mu_{\max}(F)\geqslant-\beta
\end{align*} if $F$ is nonzero. Moreover, we have $\mu_{\max}(F)<\beta$, otherwise, there is a subobject $F_{\max}$ of $F$ which has slope $\beta$ and $F_{\max}[1]$ is also a nonzero subobject of $E$. However, $v_{1}^{\beta}(F[1])=0$, it is a contradiction. Thus $H_{1} \in \mathrm{Coh}_{0}^{>-\beta}(X)$. On the other hand, $\dim \mathcal{H}^{1}(\mathbb{D}_{X/Y}(T)) \leqslant n-1$ if it is nonzero, thus  $\mathcal{H}^{1}(\mathbb{D}_{X/Y}(T)) \in \mathrm{Coh}_{0}^{>-\beta}(X)$. Therefore, $\mathcal{H}^{1}(\mathbb{D}_{X/Y}(E)) \in \mathrm{Coh}_{0}^{>-\beta}(X)$.

(2) If $\mathcal{H}^{n}(\mathbb{D}_{X/Y}(E)) \neq 0$, then it should be a zero dimensional sheaf. Now we have a nonzero morphism $\phi \colon \mathbb{D}_{X/Y}(E) \to \mathcal{H}^{n}(\mathbb{D}_{X/Y}(E))[-n]$ and we can take the relative derived dual functor of it to get a nonzero morphism $\mathbb{D}_{X/Y}(\phi) \colon A \to E$ for some zero dimensional sheaf $A$. However, we know that $\mathrm{Hom}_{D_{0}^{b}(X)}(A,E)=0$ because $E$ is tilt-torsionfree.

(3) By Lemma \ref{importantlemmaB2}, we know that $\dim \mathcal{H}^{n-1}(\mathbb{D}_{X/Y}(E)) \leqslant 1$. If $\dim \mathcal{H}^{n-1}(\mathbb{D}_{X/Y}(E))=1$, then we denote $\mathcal{H}^{n-1}(\mathbb{D}_{X/Y}(E))$ by $B_{1}$. Note that we always have an epimorphism $\phi_{1} \colon B_{1} \to B_{2}$, where $B_{2}$ is a pure sheaf of dimension one. Now we have a nonzero morphism $\phi_{2} \colon \mathbb{D}_{X/Y}(E) \to B_{1}[-n+1]$, where $H^{n-1}(\phi_{2}) $ is an isomorphism. Therefore, $\phi_{1}[-n+1] \circ \phi_{2}$ is a nonzero morphism from $\mathbb{D}_{X/Y}(E)$ to $B_{2}[-n+1]$. Thus $\mathrm{Hom}_{D_{0}^{b}(X)}(B_{3},E)=\mathrm{Hom}_{D_{0}^{b}(X)}(\mathbb{D}_{X/Y}(E),B_{2}[-n+1]) \neq 0$, where $B_{3}= \mathcal{H}^{n-1}(\mathbb{D}_{X/Y}(B_{2}))$ is a pure sheaf of dimension one. But it is a contradiction because $E$ is tilt-torsionfree.
\end{pf}

Note that for any object $0 \neq E \in \mathrm{Coh}_{0}^{\beta}(X)$, we have a filtration of $\mathbb{D}_{X/Y}(E)$ in $D_{0}^{b}(X)$:
\begin{align}\label{dualfiltration}
	\xymatrix{
		0=E_{0} \ar[r]  & E_{1} \ar[r]\ar[d]  & \cdots \ar[r]  & E_{n-1} \ar[r]  & E_{n}=\mathbb{D}_{X/Y}(E)\ar[d]\\
		& \mathcal{H}^{0}(\mathbb{D}_{X/Y}(E))\ar@{-->}[lu]  & & & \mathcal{H}^{n-1}(\mathbb{D}_{X/Y}(E))[-n+1]\ar@{-->}[lu]
	}
\end{align} and there is a distinguished triangle in $D_{0}^{b}(X)$ here:
\begin{align*}
	E_{2} \to \mathbb{D}_{X/Y}(E) \to E_{2}' \to E_{2}[1].
\end{align*}Here, $E_{2}$ is a two-term complex with the distinguished triangle:
\begin{align*}
	\mathcal{H}^{0}(\mathbb{D}_{X/Y}(E)) \to E_{2} \to \mathcal{H}^{1}(\mathbb{D}_{X/Y}(E))[-1] \to \mathcal{H}^{0}(\mathbb{D}_{X/Y}(E))[1]
\end{align*} 

If $E$ is tilt-torsionfree, then by Lemma \ref{importantlemmaB3},  we know that $E_{2}[1] \in \mathrm{Coh}_{0}^{-\beta}(X)$. We define $E_{2}[1]$ as the tilt-dual object of $E$ and denote it by $E^{\lor_{\beta}}$. In fact, by Lemma \ref{importantlemmaB2}, for any $l \in \{2,\ldots,n-1\}$, we have $\dim \mathcal{H}^{l}(\mathbb{D}_{X/Y}(E)) \leqslant n-l \leqslant n-2$. It means that $\mathcal{H}^{l}(\mathbb{D}_{X/Y}(E)) \in \mathrm{Coh}_{0}^{-\beta}(X)$ and the filtration (\ref{dualfiltration}) is a Harder-Narasimhan filtration with respect to the bounded heart $\mathrm{Coh}_{0}^{-\beta}(X)$ if we consider that $E_{2}$ is a object in $\mathrm{Coh}_{0}^{-\beta}(X)[-1]$. Explicitly, we have: 
\begin{align}\label{tiltcohomologyofdual}
	\mathcal{H}^{l}_{-\beta}(\mathbb{D}_{X/Y}(E)) \cong \begin{cases}
		0 & l \leqslant 0 \mbox{ or } l \geqslant n; \\
		E^{\lor_{\beta}}                    &  l=1;   \\
		\mathcal{H}^{l}(\mathbb{D}_{X/Y}(E)) & 1<l<n,
	\end{cases}
\end{align}where $\mathcal{H}^{l}_{-\beta}$ is the cohomology functor with respect to the bounded heart $\mathrm{Coh}_{0}^{-\beta}(X)$. 

Before our discussion, we have an important lemma:
\begin{lem}
	If $E \in \mathrm{Coh}_{0}^{\beta}(X)$ with $v_{1}^{\beta}(E)=0$. Then for any $l \leqslant 1$, we have $\mathcal{H}^{l}_{-\beta}(\mathbb{D}_{X/Y}(E))=0$.
\end{lem}

\begin{pf}
	We denote $\mathcal{H}^{-1}(E)$ by $F$ and $\mathcal{H}^{0}(E)$ by $T$.

	Because $v_{1}^{\beta}(E)=0$, we know that $\dim T \leqslant n-2$ or $T=0$ and $F$ is slope semistable with slope $\beta$ or $F=0$. Thus, after taking the cohomology functor of (\ref{dualsequence2}), we have: $\mathcal{H}^{0}(\mathbb{D}_{X/Y}(E))=0$ and an exact sequence in $\mathrm{Coh}_{0}(X)$:
	\begin{align*}
		0 \to \mathcal{H}^{1}(\mathbb{D}_{X/Y}(E)) \to \mathcal{H}^{0}(\mathbb{D}_{X/Y}(F)) \to \mathcal{H}^{2}(\mathbb{D}_{X/Y}(T)).
	\end{align*}
	If $F=0$, it is trivial. If $F \neq 0$, then $F^{\lor}=\mathcal{H}^{0}(\mathbb{D}_{X/Y}(F))$ is also slope semistable with slope $-\beta$ by propositon \ref{propositionB1}. Thus, $F^{\lor} \in \mathrm{Coh}_{0}^{\leqslant -\beta}(X)$ and $\mathcal{H}^{1}(\mathbb{D}_{X/Y}(E))$ is a subobject of $F^{\lor}$ in $\mathrm{Coh}_{0}(X)$. Therefore, $\mathcal{H}^{1}(\mathbb{D}_{X/Y}(E))[-1] \in \mathrm{Coh}_{0}^{-\beta}(X)[-2]$. Combine the filtration (\ref{dualfiltration}), we know that for any $l \leqslant 1$, we have $\mathcal{H}^{l}_{-\beta}(\mathbb{D}_{X/Y}(E))=0$.
\end{pf}

Now, suppose that $E \in \mathrm{Coh}_{0}^{\beta}(X)$ with $v_{1}^{\beta}(E)>0$, then we have a decomposition in $\mathrm{Coh}_{0}^{\beta}(X)$:
\begin{align*}
	0 \to E_{t} \to E \to E_{tf} \to 0,
\end{align*}where $E_{tf} \neq 0$ is tilt-torsionfree and $v_{1}^{\beta}(E_{t})=0$. Note that $\mathcal{H}^{1}_{-\beta}(\mathbb{D}_{X/Y}(E)) \cong \mathcal{H}^{1}_{-\beta}(\mathbb{D}_{X/Y}(E_{tf}))=E_{tf}^{\lor_{\beta}}$ by the above lemma. In this case, we define the tilt-dual object of $E$ as $\mathcal{H}^{1}_{-\beta}(\mathbb{D}_{X/Y}(E))\in \mathrm{Coh}_{0}^{-\beta}(X)$ and denote it by $E^{\lor_{\beta}}$.

\begin{lem}
	Let $\beta \in \mathbb{R}$ and suppose that $0 \neq E \in \mathrm{Coh}^{\beta}_{0}(X)$ is tilt-torsionfree, then $E^{\lor_{\beta}} \in \mathrm{Coh}_{0}^{-\beta}(X)$ is also tilt-torsionfree.
\end{lem}

\begin{pf}
	If $E^{\lor_{\beta}}$ is not tilt-torsionfree, then we have a nonzero subobject $A$ of $E^{\lor_{\beta}}$ in $\mathrm{Coh}_{0}^{-\beta}(X)$ with $v_{1}^{\beta}(A)=0$. 
	If $\mathcal{H}^{-1}(A) \neq 0$, then on one hand, $\mathcal{H}^{-1}(A) \in \mathrm{Coh}_{0}(X)$ is a slope semistable sheaf with slope $-\beta$. On the other hand $\mathcal{H}^{-1}(A)$ is a subsheaf of $\mathcal{H}^{-1}(E^{\lor_{\beta}}) \cong \mathcal{H}^{0}(\mathbb{D}_{X/Y}(E)) \in \mathrm{Coh}^{<-\beta}_{0}(X)$ by Lemma \ref{importantlemmaB3} and it is a contradiction, Thus $\mathcal{H}^{-1}(A)=0$ and $A$ is a nonzero sheaf with $\dim A \leqslant n-2$.
	
	By (\ref{tiltcohomologyofdual}), we have a distinguished triangle here:
	\begin{align*}
		E^{\lor_{\beta}} \to \mathbb{D}_{X/Y}(E)[1] \to E'[1] \to E^{\lor_{\beta}}[1]
	\end{align*} and $\mathcal{H}^{l}(E') \cong  \mathcal{H}^{l}(\mathbb{D}_{X/Y}(E)) \in \mathrm{Coh}_{0}^{-\beta}(X)$ when $1<l<n$ thus
	\begin{align*}
		E'[1] \in \langle \mathrm{Coh}_{0}^{-\beta}(X)[-1],\ldots, \mathrm{Coh}_{0}^{-\beta}(X)[-1+n] \rangle.
	\end{align*}Then we have:
	\begin{align*}
		0 \neq \mathrm{Hom}_{D_{0}^{b}(X)}(A,E^{\lor_{\beta}}) \cong \mathrm{Hom}_{D_{0}^{b}(X)}(A,\mathbb{D}_{X/Y}(E)[1])
	\end{align*} by taking Hom functor. However, $\mathrm{Hom}_{D_{0}^{b}(X)}(A,\mathbb{D}_{X/Y}(E)[1]) \cong \mathrm{Hom}_{D_{0}^{b}(X)}(E,\mathbb{D}_{X/Y}(A)[1])$.
	Because $\dim A \leqslant n-2$, then $\mathcal{H}^{l}_{-\beta}(\mathbb{D}_{X/Y}(A)) \cong \mathcal{H}^{l}(\mathbb{D}_{X/Y}(A))$ for any $l \in \{2,\ldots,n\}$ and thus
	$\mathbb{D}_{X/Y}(A) \in \langle \mathrm{Coh}_{0}^{-\beta}(X)[-2],\ldots, \mathrm{Coh}_{0}^{-\beta}(X)[-n] \rangle$. Thus $\mathrm{Hom}_{D_{0}^{b}(X)}(A,\mathbb{D}_{X/Y}(E)[1])=0$ and it is a contradiction.
\end{pf}

\begin{lem}
	Let $\beta \in \mathbb{R}$ and $E \in \mathrm{Coh}_{0}^{\beta}(X)$ with $v_{1}^{\beta}(E)>0$, then $v_{1}^{-\beta}(E^{\lor_{\beta}})=v_{1}^{\beta}(E)>0$ and $\nu^{-\beta,\alpha}(E^{\lor_{\beta}}) \geqslant -\nu^{\beta,\alpha}(E)$.
\end{lem}

\begin{pf}
	First of all, we assume that $E$ is tilt-torsionfree. Then by Lemma \ref{importantlemmaB2}, we know that for any $l \in \{2,\ldots,n\}$, $\dim \mathcal{H}^{l}(\mathbb{D}_{X/Y}(E)) \leqslant n-l \leqslant n-2$. Thus we have
	\begin{align*}
		& v_{2}(\mathbb{D}_{X/Y}(E))=v_{2}(\mathcal{H}^{2}(\mathbb{D}_{X/Y}(E)))-v_{2}(E^{\lor_{\beta}}) \geqslant -v_{2}(E^{\lor_{\beta}}); \\
		& v_{1}(\mathbb{D}_{X/Y}(E))=-v_{1}(E^{\lor_{\beta}}); \\
		& v_{0}(\mathbb{D}_{X/Y}(E))=-v_{0}(E^{\lor_{\beta}}).
	\end{align*}by (\ref{tiltcohomologyofdual}). Meanwhile, we know that $v_{i}(\mathbb{D}_{X/Y}(E))=(-1)^{i}v_{i}(E)$ for $i \in \{0,1,2\}$. Then
	\begin{align*}
		v_{1}^{-\beta}(E^{\lor_{\beta}})=v_{1}(E^{\lor_{\beta}})+\beta v_{0}(E^{\lor_{\beta}})=v_{1}(E)-\beta v_{0}(E)=v_{1}^{\beta}(E)>0
	\end{align*} and
	\begin{align*}
		v_{2}(E^{\lor_{\beta}})-\alpha v_{0}(E^{\lor_{\beta}}) \geqslant -v_{2}(E)+\alpha v_{0}(E)
	\end{align*}Thus $\nu^{-\beta,\alpha}(E^{\lor_{\beta}}) \geqslant -\nu^{\beta,\alpha}(E)$.
	
	If $E$ is not tilt-torsionfree, then we have a decomposition in $\mathrm{Coh}_{0}^{\beta}(X)$:
	\begin{align*}
		0 \to E_{t} \to E \to E_{tf} \to 0,
	\end{align*}where $E_{tf} \neq 0$ is tilt-torsionfree and $v_{1}^{\beta}(E_{t})=0$. Then
	\begin{align*}
		v_{1}^{-\beta}(E^{\lor_{\beta}})=v_{1}^{-\beta}(E_{tf}^{\lor_{\beta}})=v_{1}^{\beta}(E_{tf})=v_{1}^{\beta}(E).
	\end{align*} and
	\begin{align*}
		\nu^{-\beta,\alpha}(E^{\lor_{\beta}})=\nu^{-\beta,\alpha}(E_{tf}^{\lor_{\beta}}) \geqslant -\nu^{\beta,\alpha}(E_{tf})=-\frac{(v_{2}(E)-\alpha v_{0}(E))-(v_{2}(E_{t})-\alpha v_{0}(E_{t}))}{v_{1}(E)-\beta v_{0}(E)} \geqslant -\nu^{\beta,\alpha}(E)
	\end{align*}because $v_{1}^{\beta}(E_{t})=0$ and $v_{2}(E_{t})-\alpha v_{0}(E_{t}) \geqslant 0$.
\end{pf}

\begin{rmk}
	Let $\beta \in \mathbb{R}$ and suppose that $0 \neq E \in \mathrm{Coh}^{\beta}_{0}(X)$ is tilt-torsionfree, we expect the following property below: For any $1<l<n$, $\dim \mathcal{H}^{l}(\mathbb{D}_{X/Y}(E)) \leqslant n-l-1$. If this conjecture holds, then we have $\nu^{-\beta,\alpha}(E^{\lor_{\beta}})=-\nu^{\beta,\alpha}(E)$.

	If this conjecture holds, then we can see the concept of tilt-torsionfree is a variation of pureness. In fact, when $n=2$ and $n=3$, the conjecture holds. But we still don't know the consequence when $n \geqslant 4$.
\end{rmk}

If $n=2$, then $\mathbb{D}_{X/Y}(E) \cong E^{\lor_{\beta}}[-1]$ for any tilt-torsionfree object $E \in \mathrm{Coh}_{0}^{\beta}(X)$. Then $E \cong \mathbb{D}_{X/Y}(E^{\lor_{\beta}})[1]$ and by definition $E^{\lor_{\beta}\lor_{-\beta}} \cong \mathcal{H}^{1}_{\beta}(\mathbb{D}_{X/Y}(E^{\lor_{\beta}})) \cong E$.

If $n=3$ and $0 \neq E \in \mathrm{Coh}_{0}^{\beta}(X)$ is a tilt-torsionfree object, then we have a distinguished triangle:
\begin{align}
	\label{dualsequence4}E^{\lor_{\beta}}[-1] \to \mathbb{D}_{X/Y}(E) \to \mathcal{H}^{2}(\mathbb{D}_{X/Y}(E))[-2] \to E^{\lor_{\beta}},
\end{align}where $\mathcal{H}^{2}(\mathbb{D}_{X/Y}(E))$ is a zero dimensional sheaf.

\begin{prop}\label{propositionB2}
	Let $\beta \in \mathbb{R}$ and $0 \neq E \in \mathrm{Coh}_{0}^{\beta}(X)$ is tilt-torsionfree. If $\dim Y=3$, then:
	
	(1) We have a short exact sequence in $\mathrm{Coh}_{0}^{\beta}(X)$:
	\begin{align*}
		0 \to E \stackrel{\iota_{E}}{\longrightarrow} E^{\lor_{\beta}\lor_{-\beta}} \to E^{\lor_{\beta}\lor_{-\beta}}/E \to 0,
	\end{align*}where $E^{\lor_{\beta}\lor_{-\beta}}/E$ is $0$ or $\dim E^{\lor_{\beta}\lor_{-\beta}}/E=0$.
	
	(2) If $E$ is $\nu^{\beta,\alpha}$-semistable, then $E^{\lor_{\beta}\lor_{-\beta}}$ is $\nu^{\beta,\alpha}$-semistable.
	
	(3) If $E$ is $\nu^{\beta,\alpha}$-semistable, then $E^{\lor_{\beta}}$ is $\nu^{-\beta,\alpha}$-semistable with tilt-slope $-\nu^{\beta,\alpha}(E)$.
	
	(4) Moreover, suppose that there is a short exact sequence in $\mathrm{Coh}_{0}^{\beta}(X)$:
	\begin{align*}
		0 \to F \stackrel{f}{\longrightarrow} E \to G \to 0
	\end{align*}with $F \neq 0$ and $G$ is $0$ or $\dim G=0$. Then there is a commutative diagram in $\mathrm{Coh}_{0}^{\beta}(X)$:
	\begin{align*}
		\xymatrix{0 \ar[r] & F \ar^{\iota_{F}}[r] \ar_{f}[d] & F^{\lor_{\beta}\lor_{-\beta}} \ar^{f^{\lor_{\beta}\lor_{-\beta}}}[d] \\
			0 \ar[r] & E \ar_{\iota_{E}}[r] & E^{\lor_{\beta}\lor_{-\beta}}
		}
	\end{align*}with $f^{\lor_{\beta}\lor_{-\beta}}$ is an isomorphism.
\end{prop}

\begin{pf}
	(1) After taking relative dual functor of (\ref{dualsequence4}), we have a distinguished triangle in $D_{0}^{b}(X)$:
	\begin{align*}
		E'[-1] \to E \to \mathbb{D}_{X/Y}(E^{\lor_{\beta}})[1] \to E',
	\end{align*}where $E'$ is a zero dimensional sheaf or $0$. Then we take the cohomology functor $\mathcal{H}_{\beta}^{l}$ on the above sequence, and there is an exact sequence in $\mathrm{Coh}_{0}^{\beta}(X)$:
	\begin{align*}
		0 \to E \to E^{\lor_{\beta}\lor_{-\beta}} \to E' \to 0
	\end{align*} and we get the conclusion.
	
	(2) It is a direct conclusion of Lemma \ref{importantlemmaB1} and (1).
	
	(3) Suppose that $E^{\lor_{\beta}}$ is not $\nu^{-\beta,\alpha}$-semistable, then there is a short exact sequence in $\mathrm{Coh}_{0}^{-\beta}(X)$:
	\begin{align*}
		0 \to A \to E^{\lor_{\beta}} \to B \to 0
	\end{align*}with $+\infty>\nu^{-\beta,\alpha}(A)>\nu^{-\beta,\alpha}(E^{\lor_{\beta}})>\nu^{-\beta,\alpha}(B)$ and $A,B \neq 0$ since $E^{\lor_{\beta}}$ is tilt-torsionfree. After taking the relative dual functor and the cohomology functor, we have an exact sequence in $\mathrm{Coh}_{0}(X)$:
	\begin{align*}
		0 \to B^{\lor_{-\beta}} \to E^{\lor_{\beta}\lor_{-\beta}} \to A^{\lor_{-\beta}}
	\end{align*}Now $E^{\lor_{\beta}\lor_{-\beta}}$ is $\nu^{\beta,\alpha}$-semistable by (2) and $B^{\lor_{-\beta}}$ is a nonzero subobject of $E^{\lor_{\beta}\lor_{-\beta}}$. Thus
	\begin{align*}
		-\nu^{-\beta,\alpha}(B) \leqslant \nu^{\beta,\alpha}(B^{\lor_{-\beta}}) \leqslant \nu^{\beta,\alpha}(E^{\lor_{\beta}\lor_{-\beta}})=-\nu^{-\beta,\alpha}(E^{\lor_{\beta}})
	\end{align*}and this is a contradiction. 
	
	(4) It is not hard to see that $\iota_{E} \circ f=f^{\lor_{\beta}\lor_{-\beta}} \circ \iota_{F}$. By taking the functor $\mathcal{H}^{1}_{-\beta}(\mathbb{D}_{X/Y}(-))$, we know that $\mathcal{H}^{1}_{-\beta}(\mathbb{D}_{X/Y}(G))=0$. It implies $f^{\lor}$ is an isomorphism and $f^{\lor_{\beta}\lor_{-\beta}}$ is also an isomorphism. 
\end{pf}

\subsection{Some applications}

We will give the proof about Notherian property of some tilting hearts. But before that, we will introduce some basic techniques about it.

Let $\mathcal{D}$ be a triangulated category. Suppose that $\sigma=(Z,\mathcal{A})$ is a weak pre-stability condition on $\mathcal{D}$  and $\mu$ is the corresponding slope function. For any $\beta \in \mathbb{R}$, let 
\begin{align*}
	&\mathcal{T}^{\beta}=\langle E\in \mathcal{A} \mid E \mbox{ is } \sigma\mbox{-semistable with }\mu(E)>\beta \rangle, \\
	&\mathcal{F}^{\beta}=\langle E\in \mathcal{A} \mid E \mbox{ is } \sigma\mbox{-semistable with }\mu(E) \leqslant \beta \rangle.
\end{align*} Then $(\mathcal{T}^{\beta},\mathcal{F}^{\beta})$ is a torsion pair of $\mathcal{A}$ by the existence of Harder-Narasimhan filtration. Let $\mathcal{A}^{\beta}=\langle\mathcal{F}^{\beta}[1],\mathcal{T}^{\beta}\rangle$, then it is a bounded heart of $\mathcal{D}$.

\begin{prop}\label{Noetherian}
	Let $\beta \in \mathbb{R}$ and we assume that:
	
	(1) the set $\{-\mathrm{Re}Z(E)-\beta\mathrm{Im}Z(E) \in \mathbb{R} \mid E \in \mathcal{A}^{\beta}\}$ is discrete in $\mathbb{R}$;
	
	(2) the abelian category $\mathcal{A}$ is Noetherian;
	
	(3) there is no infinite sequence in $\mathcal{A}$:
	\begin{align*}
		0 \neq G_{1} \subseteq G_{2} \subseteq \cdots \subseteq G_{n} \subseteq G_{n+1} \subseteq \cdots
	\end{align*} with $G_{n} \in \mathcal{F}^{\beta}$, $0 \neq G_{n+1}/G_{n} \in \mathcal{C}$ for any $n \in \mathbb{N}^{+}$. Then the abelian category $\mathcal{A}^{\beta}$ is Noetherian.
\end{prop}

\begin{pf}
		Suppose that there is an infinite sequence of epimorphisms in $\mathcal{A}^{\beta}$:
	\begin{align*}
		E_{0} \twoheadrightarrow E_{1} \twoheadrightarrow E_{2} \twoheadrightarrow \cdots \twoheadrightarrow E_{n} \twoheadrightarrow E_{n+1} \twoheadrightarrow \cdots.
	\end{align*} Then we have
	\begin{align*}
		-\Re Z(E_{0})-\beta\Im Z(E_{0}) \geqslant -\Re Z(E_{1})-\beta\Im Z(E_{1}) \geqslant -\Re Z(E_{2})-\beta\Im Z(E_{2}) \geqslant \cdots \geqslant 0.
	\end{align*}However, $\{-\Re Z(E)-\beta\Im Z(E) \in \mathbb{R} \mid E \in \mathcal{A}^{\beta}\}$ is discrete in $\mathbb{R}$. Thus, when $n \gg 0$, we have $-\Re Z(E_{n})-\beta\Im Z(E_{n})=-\Re Z(E_{n+1})-\beta\Im Z(E_{n+1})=\cdots$. Without loss of generality, we assume that $-\Re Z(E_{n})-\beta\Im Z(E_{n})$ is constant.
	
	Without loss of generality, we assume that $E_{n} \neq E_{n+1}$ for any $n \in \mathbb{N}$ and give a contradiction of it.
	
	For any $n \in \mathbb{N}^{+}$, let $F_{n}=\ker(E_{0} \to E_{n}) \in \mathcal{A}^{\beta}$, then there is an infinite sequence in $\mathcal{A}^{\beta}$:
	\begin{align*}
		F_{1} \subseteq F_{2} \subseteq \cdots \subseteq E_{0}.
	\end{align*}By using the snake lemma, we have $\ker(E_{n} \to E_{n+1}) \cong F_{n+1}/F_{n}$ for any $n \in \mathbb{N}^{+}$ in $\mathcal{A}^{\beta}$. And we note that $\Re Z(F_{n})+\beta\Im Z(F_{n})=0$. Because $E_{n} \neq E_{n+1}$, then $F_{n} \subsetneq F_{n+1}$.
	
	After taking the cohomology functor with respect to the bounded heart $\mathcal{A}$, we have:
	\begin{align*}
		\mathcal{H}^{-1}(F_{1}) \subseteq \mathcal{H}^{-1}(F_{2}) \subseteq \cdots \subseteq \mathcal{H}^{-1}(F_{n}) \subseteq \mathcal{H}^{-1}(F_{n+1})  \subseteq \cdots \subseteq \mathcal{H}^{-1}(E_{0}). 
	\end{align*} and
	\begin{align*}
		\mathcal{H}^{0}(E_{0}) \twoheadrightarrow \mathcal{H}^{0}(E_{1}) \twoheadrightarrow \mathcal{H}^{0}(E_{2}) \twoheadrightarrow \cdots \twoheadrightarrow \mathcal{H}^{0}(E_{n}) \twoheadrightarrow \mathcal{H}^{0}(E_{n+1}) \twoheadrightarrow \cdots.
	\end{align*} Becasue $\mathcal{A}$ is Noetherian, these two filtrations must be stable. Thus we assume that $\mathcal{H}^{-1}(F_{n}) \cong \mathcal{H}^{-1}(F_{n+1}) \subseteq \mathcal{H}^{-1}(E_{0})$ and $\mathcal{H}^{0}(E_{n-1}) \cong \mathcal{H}^{0}(E_{n})$ for any $n \in \mathbb{N}^{+}$. Let $Q=\mathcal{H}^{-1}(E_{0})/\mathcal{H}^{-1}(F_{n})$. 
	
	Now, we will take the cohomology functor $\mathcal{H}^{l}$ with respect to $\mathcal{A}$ on the following short exact sequences in $\mathcal{A}^{\beta}$ for any $n \in \mathbb{N}^{+}$:
	\begin{align*}
		& 0 \to F_{n} \to E_{0} \to E_{n} \to 0,\\
		& 0 \to F_{n} \to F_{n+1} \to F_{n+1}/F_{n} \to 0,\\
		& 0 \to F_{n+1}/F_{n} \to E_{n} \to E_{n+1} \to 0.
	\end{align*}
	Then we get exact sequences in $\mathcal{A}$ for any $n \in \mathbb{N}^{+}$:
	\begin{align*}
		& 0 \to \mathcal{H}^{-1}(F_{n}) \to  \mathcal{H}^{-1}(E_{0}) \to   \mathcal{H}^{-1}(E_{n}) \to \mathcal{H}^{0}(F_{n}) \to 0, \\
		& 0 \to \mathcal{H}^{-1}(F_{n+1}/F_{n}) \to  \mathcal{H}^{0}(F_{n}) \to   \mathcal{H}^{0}(F_{n+1}) \to \mathcal{H}^{0}(F_{n+1}/F_{n}) \to 0,\\
		& 0 \to \mathcal{H}^{-1}(F_{n+1}/F_{n}) \to  \mathcal{H}^{-1}(E_{n}) \to   \mathcal{H}^{-1}(E_{n+1}) \to \mathcal{H}^{0}(F_{n+1}/F_{n}) \to 0.
	\end{align*}
	
	Because $\Re Z(F_{n})+\beta\Im Z(F_{n})=0$, then  $\mathcal{H}^{-1}(F_{n})$ is a $\sigma$-semistable object with slope $\beta$ or $\mathcal{H}^{-1}(F_{n})=0$ and $\mathcal{H}^{0}(F_{n}) \in \mathcal{C}$ by Lemma \ref{imaginarypart}. By the second exact sequence, the obejct $\mathcal{H}^{-1}(F_{n+1}/F_{n})$ also need belong to $\mathcal{C}$, and obviously $\mathcal{H}^{-1}(F_{n+1}/F_{n}) \in \mathcal{F}^{\beta}$. Then $\mathcal{H}^{-1}(F_{n+1}/F_{n})=0$ for any $n \in \mathbb{N}^{+}$. Meanwhile, $\mathcal{H}^{0}(F_{n+1}/F_{n})$ is a quotient object of $\mathcal{H}^{0}(F_{n+1})$ and then $0 \neq F_{n+1}/F_{n} \cong \mathcal{H}^{0}(F_{n+1}/F_{n}) \in \mathcal{C}$ for any $n \in \mathbb{N}^{+}$.
	
	We also have a short exact sequence in $\mathcal{A}$ as follows:
	\begin{align*}
		0 \to Q \to \mathcal{H}^{-1}(E_{n}) \to \mathcal{H}^{0}(F_{n}) \to 0.
	\end{align*} If $Q=0$, it means for any $n \in \mathbb{N}^{+}$, $\mathcal{H}^{0}(F_{n}) \cong \mathcal{H}^{-1}(E_{n})=0$ and it will be contradiction because $0 \neq F_{n+1}/F_{n} \cong \mathcal{H}^{0}(F_{n+1}/F_{n})$. Thus, $\mathcal{H}^{-1}(E_{n}) \neq 0$ for any $n \in \mathbb{N}^{+}$.
	
	Finally, let $G_{n}=\mathcal{H}^{-1}(E_{n}) \in \mathcal{F}^{\beta}$ ,then we will have an infinite sequence in $\mathcal{A}$:
	\begin{align*}
		0 \neq G_{1} \subseteq G_{2} \subseteq \cdots \subseteq G_{n} \subseteq G_{n+1} \subseteq \cdots
	\end{align*} with $G_{n} \in \mathcal{F}^{\beta}$, and $0 \neq F_{n+1}/F_{n} \cong G_{n+1}/G_{n} \in \mathcal{C}$ for any $n \in \mathbb{N}^{+}$. It is a contradiction with our condition on $\mathcal{A}$. Thus, $E_{n}=E_{n+1}$ for $n \gg 0$.
\end{pf}

\begin{prop}\label{Noetherian1}
	For any $\beta \in \mathbb{Q}$, the abelian category $\mathrm{Coh}^{\beta}_{0}(X)$ is Noetherian.
\end{prop}

\begin{pf}
	By Proposition \ref{Noetherian}, we only need to prove the following: there is no infinite sequence in $\mathrm{Coh}_{0}(X)$:
	\begin{align*}
		0 \neq G_{1} \subseteq G_{2} \subseteq \cdots \subseteq G_{m} \subseteq G_{m+1} \subseteq \cdots
	\end{align*} with $G_{m} \in \mathrm{Coh}^{\leqslant \beta}_{0}(X)$, and the sheaf $G_{m+1}/G_{m}$ is nonzero for any $m \in \mathbb{N}^{+}$ with $\dim G_{m+1}/G_{m} \leqslant n-2$. Note that the sheaf $G_{m}$ is pure of dimension $n$. 
	
	We denote the cokernel of $G_{1} \to G_{m}$ by $T_{m}$. We have another infinite sequence:
	\begin{align}
		\label{Noetherianseq1} 0=T_{1} \subseteq T_{2} \subseteq \cdots \subseteq T_{m} \subseteq T_{m+1} \subseteq \cdots
	\end{align} with $T_{m+1}/T_{m} \cong G_{m+1}/G_{m} \neq 0$ for any $m \in \mathbb{N}^{+}$ and $\dim T_{m} \leqslant n-2$ by the snake lemma. By Lemma \ref{propositionB1}, we have injections $T_{m} \to G_{1}^{\lor\lor}/G_{1}$ in $\mathrm{Coh}_{0}(X)$. By the Noetherian property of $\mathrm{Coh}_{0}(X)$, $T_{m}=T_{m+1}$ for $m \gg 0$. and thus the sequence (\ref{Noetherianseq1}) must be stable and it is a contradiction.
\end{pf}

\begin{prop}\label{Noetherian2}
	For any $(\beta,\alpha) \in U \cap \mathbb{Q}^2$, the abelian category $\mathrm{Coh}^{\beta,\alpha}_{0}(X)$ is Noetherian.
\end{prop}

\begin{pf}
	By proposition \ref{Noetherian}, we only need to prove the following: there is no infinite sequence in $\mathrm{Coh}_{0}^{\beta}(X)$ for any $\beta$:
	\begin{align*}
		0 \neq G_{1} \subseteq G_{2} \subseteq \cdots \subseteq G_{m} \subseteq G_{m+1} \subseteq \cdots
	\end{align*} with $G_{m} \in \mathrm{Coh}^{\leqslant \beta,\alpha}_{0}(X)$, and the object $G_{m+1}/G_{m}$ is a nonzero sheaf for any $m \in \mathbb{N}$ with $\dim G_{m+1}/G_{m}=0$. Note that the object $G_{m}$ is tilt-torsionfree.  
	
	We denote the cokernel of $G_{1} \to G_{m}$ by $T_{m}$. We have another infinite sequence:
	\begin{align}
		\label{seq2} 0=T_{1} \subseteq T_{2} \subseteq \cdots \subseteq T_{m} \subseteq T_{m+1} \subseteq \cdots
	\end{align} with $T_{m+1}/T_{m} \cong G_{m+1}/G_{m} \neq 0$ and $T_{m}$ is a nonzero dimensional sheaf by the snake lemma for any $m \in \mathbb{N}^{+}$. By Lemma \ref{propositionB2}, we have injections $T_{m} \to G_{1}^{\lor_{\beta}\lor_{-\beta}}/G_{1}$ in $\mathrm{Coh}_{0}^{\beta}(X)$. By the Noetherian property of $\mathrm{Coh}_{0}^{\beta}(X)$, we have $T_{m}=T_{m+1}$ for $m \gg 0$. and thus the sequence (\ref{seq2}) must be stable and it is a contradiction.
\end{pf}

\begin{prop}\label{reduce4B}
	Suppose that $\dim Y=3$ and $(\beta,\alpha) \in U$ satisfies the Bogomolov-Gieseker type inequality. Then $(-\beta,\alpha) \in U$ also satisfies the Bogomolov-Gieseker type inequality.
\end{prop}

\begin{pf}
	Assume that there exists $\nu^{-\beta,\alpha}$-semistable object $E \in \mathrm{Coh}_{0}^{-\beta}(X)$ with tilt-slope $-\beta$, which does not satisfy the Bogomolov-Gieseker type inequality, in other words, $v_{3}^{-\beta}(E)> \displaystyle\frac{2\alpha-\beta^{2}}{6}v_{1}^{-\beta}(E)$. However, by \ref{propositionB2}, $E^{\lor_{-\beta}}$ is a $\nu^{\beta,\alpha}$-semistable object with tilt-slope $\beta$.
	Then 
	\begin{align*}
		v_{3}^{\beta}(E^{\lor_{-\beta}}) \leqslant  \displaystyle\frac{2\alpha-\beta^{2}}{6}v_{1}^{\beta}(E^{\lor_{-\beta}}).
	\end{align*} We can verify that $v_{3}^{-\beta}(E) \leqslant v_{3}^{\beta}(E^{\lor_{-\beta}})$ and $v_{1}^{\beta}(E^{\lor_{-\beta}})=v_{1}^{-\beta}(E)$. Then we have a contradiction.
\end{pf}

\bibliographystyle{alpha}

\newcommand{\etalchar}[1]{$^{#1}$}

\end{document}